\documentclass[11pt]{amsart}
\usepackage{
  a4wide,
  amsmath, 
  mathtools,
  amsfonts,
  amssymb,
  graphicx, 
  times,
  color,
  hyphenat,
  pinlabel,
  stmaryrd 
}
\input xy
\xyoption{all}
\newcommand{\lai}{\lambda_{i}}
\newcommand{\laii}{\lambda_{i+1}}

\newcommand{\und}[1]{\underline{#1}}

\newcommand{\n}{\noindent}

\newcommand{\qbin}[2]{\left[{{#1}\atop {#2}}\right]}
\newcommand{\schurD}{S_q(n,\Delta_N)}

\DeclareMathOperator{\End}{End}

\DeclareMathOperator{\lev}{L}

\newcommand{\figins}[3] 
{\raisebox{#1pt}{\includegraphics[height=#2 in]{figs/#3}}}

\newtheorem{thm}{Theorem}[section]
\newtheorem{lem}[thm]{Lemma}

\newtheorem{prop}[thm]{Proposition}

\newtheorem{conj}{Conjecture}

\theoremstyle{definition}
\newtheorem{defn}[thm]{Definition}
\newtheorem{prob}{Problem}

\newcommand{\bN}{\mathbb{N}}
\newcommand{\bZ}{\mathbb{Z}}
\newcommand{\bQ}{\mathbb{Q}}

\newcommand{\bC}{\mathbb{C}}

\newcommand{\cA}{\mathcal{A}}

\newcommand{\cC}{\mathcal{C}}

\newcommand{\cU}{\mathcal{U}}

\newcommand{\cX}{\mathcal{X}}
\newcommand{\cY}{\mathcal{Y}}

\DeclareMathOperator{\Image}{im}
\DeclareMathOperator{\id}{Id}

\DeclareMathOperator{\MOY}{MOY}
\DeclareMathOperator{\bmw}{BMW}
\DeclareMathOperator{\skein}{Skein}
\DeclareMathOperator{\rot}{rot}
\DeclareMathOperator{\res}{res}

\newcommand{\xra}[1]{\xrightarrow{#1}}

\catcode`\@=11
\long\def\@makecaption#1#2{%
    \vskip 10pt
    \setbox\@tempboxa\hbox{%
\small{#1: }\ignorespaces #2}%
    \ifdim \wd\@tempboxa >\captionwidth {%
        \rightskip=\@captionmargin\leftskip=\@captionmargin
        \unhbox\@tempboxa\par}%
      \else
        \hbox to\hsize{\hfil\box\@tempboxa\hfil}%
    \fi}
\newdimen\@captionmargin\@captionmargin=2\parindent
\newdimen\captionwidth\captionwidth=\hsize
\catcode`\@=12
\title{A remark on BMW algebra, $q$-Schur algebras 
and categorification}
\author{Pedro Vaz}
\address{
CAMGSD\\Instituto Superior T\'{e}cnico 
\\ Avenida Rovisco Pais\\ 1049-001 Lisboa\\ Portugal}
\address{ 
Institut de Recherche en Math\'ematique et Physique\\
Universit\'e Catholique de Louvain\\ 
Chemin du Cyclotron 2\\ 
B 1348 Louvain-la-Neuve\\ 
Belgique}
\email{pedro.vaz@uclouvain.be}
\author{Emmanuel Wagner}
\address{Institut de Math\'ematiques de Bourgogne UMR 5584 du CNRS\\
7 avenue Alain Savary BP 47870 21078 Dijon Cedex\\ France}
\email{emmanuel.wagner@u-bourgogne.fr}
\begin{document}
%
%
\newdimen\captionwidth\captionwidth=\hsize
%
%
%
\begin{abstract} 
We prove that the 2-variable BMW algebra
embeds into an algebra constructed from the HOMFLY-PT polynomial.
We also prove that the $\mathfrak{so}_{2N}$-BMW algebra embeds in the $q$-Schur algebra
of type $A$.
We use these results 
to suggest a schema providing categorifications of the $\mathfrak{so}_{2N}$-BMW algebra.
\\[6ex]
\noindent 2010 \emph{Mathematics Subject Classification.}
57M27, 81R50; 17B37, 16W99.
\end{abstract}
\maketitle 
\tableofcontents
\pagestyle{myheadings}
\markboth{\em\small Pedro Vaz and Emmanuel Wagner}{\em\small A remark on BMW algebra, $q$-Schur algebras and categorification.}
%
%
%
\section{Introduction}\label{sec:intro}


The most popular 2-variable link polynomials nowadays are the \emph{HOMFLY-PT polynomial}~\cite{homfly, pt} 
and the \emph{Kauffman polynomial}~\cite{Kauff}.
Recall the HOMFLY-PT polynomial $P=P(a,q)$ is the unique invariant of oriented  links satisfying
\begin{equation*}
a P\biggl(\figins{-10}{0.35}{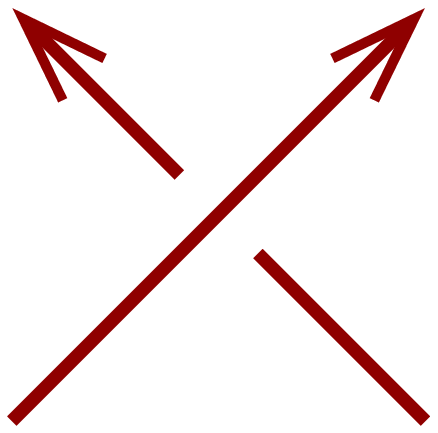}\biggr) \ \ - \ \ 
a^{-1}P\biggl(\figins{-10}{0.35}{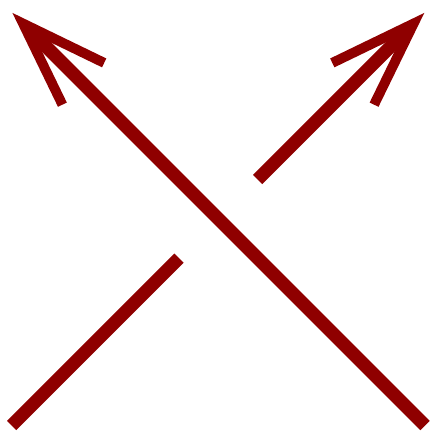}\biggr)
=\ \ (q - q^{-1})P \biggl(
\figins{-10}{0.35}{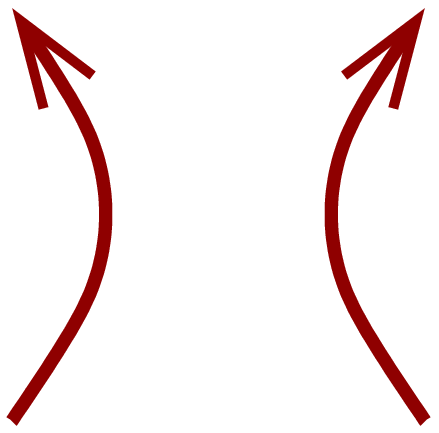}
\biggr)
\end{equation*}
and taking the value 
$\tfrac{a-a^{-1}}{q-q^{-1}}$ on the unknot, and that 
the Kauffman polynomial $F=F(a,q)$ is the two-variable invariant uniquely defined on framed unoriented links 
by
\begin{gather*}
F\biggl(\figins{-10}{0.35}{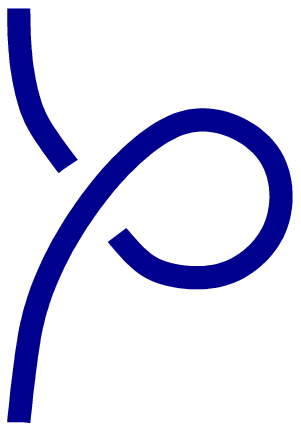}\biggr)\ \ 
= \ a^{-2}q\ \ 
F\biggl(\ \figins{-10}{0.35}{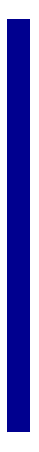}\ \biggr)
\displaybreak[0]\\[2ex]  
F\biggl(\figins{-10}{0.35}{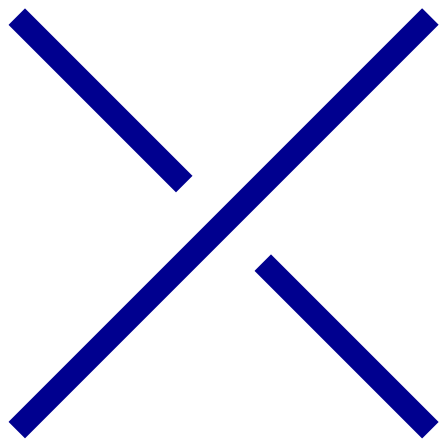}\biggr) \ \ - \ \ 
F \biggl(\figins{-10}{0.35}{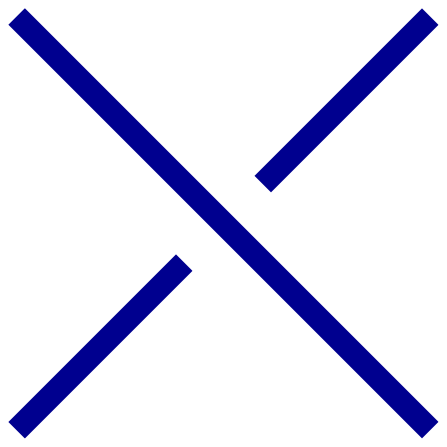}\biggr)\
=\ \ (q - q^{-1})\Biggl(\ 
F\biggl(\figins{-10}{0.35}{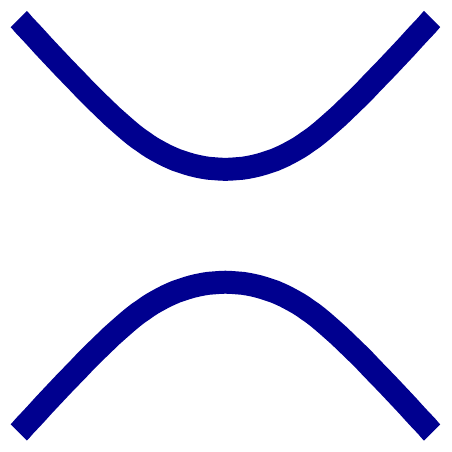}\biggr)\  -  \
F\biggl(\figins{-10}{0.35}{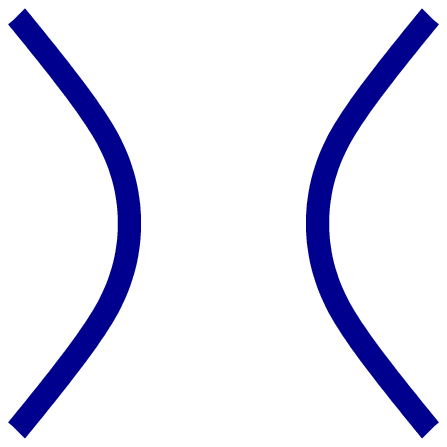}\biggr)\
\Biggr)
\end{gather*}
and taking the value $\tfrac{a^2q^{-1}-a^{-2}q}{q-q^{-1}}+1$ on the unknot.

\medskip

These two polynomials are known to be distinct link invariants since they distinguish between different sets of links.
Nevertheless there is a connection between them, found by F.~Jaeger in 1989 (see~\cite{kauff-phys}). 
Jaeger gives a state expansion of the Kauffman polynomial of a link $L$ in terms of HOMFLY-PT 
polynomials of certain links associated to $L$.
Unfortunately Jaeger passed away before including it in an article, 
but this result appears in Kauffman's book~\cite{kauff-phys}.
Jaeger formula can be roughly stated as follows:

\medskip

\n{\bf Theorem} (Jaeger's Theorem, 89'){\bf .}
\emph{For each unoriented link $L$ there is a family 
of oriented link diagrams $\{D_i\}_{\in I}$ and 
coefficients $\{c_i(a,q)\}_{i\in I}$ in $\bC(a,q)$  
such that
\begin{equation*}
F(L) = \sum_{i\in I}c_i(a,q)P(D_i).
\end{equation*}
}


\n The coefficient $c_i(a,q)$ is determined from the combinatorial data of $D_i$. 
We recall the family of link diagrams and compute the corresponding coefficients in Section~\ref{sec:jaeger}.

\medskip

The polynomials mentioned above have algebraic counterparts.
The algebra underlying the Kauffman polynomial was found by J.~Birman and H.~Wenzl in~\cite{birmanw} and independently
by J.~Murakami in~\cite{Mur} and became known as \emph{Birman-Murakami-Wenzl algebra} 
(BMW algebra). 
Later Kauffman and Vogel~\cite{kauff-vogel}  gave algebraic constructions 
of the Kauffman and HOMFLY-PT polynomials in terms of diagrams modulo relations.
We call the latter \emph{HOMFLY-PT skein algebras}.
This motivates the following.

\begin{prob}
Is there a version of Jaeger's theorem for the algebras underlying the respective polynomials?
In other words, can we extend Jaeger's theorem to a statement between the BMW and HOMFLY-PT skein algebras? 
\end{prob}

We answer this question affirmatively (see Theorem~\ref{thm:jaegmap} for the details).

\smallskip

\n{\bf Theorem 1.}  \emph{There is an injective homomorphism of algebras from the BMW algebra to the 
HOMFLY-PT skein algebra
inducing Jaeger's formula for polynomials.}

\medskip

This homomorphism has a positivity property, 
it gives a standard basis element of the BMW algebra as a linear combination 
of elements of the HOMFLY-PT skein algebra with coefficients in $\bN(a,q)$.
In face of the categorifications of the HOMFLY-PT polynomial available~\cite{Kh-bim, KR2} 
it seems natural to ask the following 
question.
\begin{prob}
Can Jaeger's theorem be used to produce a link homology categorifying the two-variable Kauffman polynomial
of links?
\end{prob}
We find a partial answer to this question.
The first difficulty we encounter is that in none of the categorifications of the 
HOMFLY-PT polynomials are there sufficiently many direct sum decompositions to 
guarantee we obtain the $\bmw_n(a,q)$.  
This is explained in Subsection~\ref{ssec:matfac} and is related with the problem 
of producing a link homology categorifying the HOMFLY-PT polynomial for tangles. 
As a matter of fact such a strategy has been applied by Wu to obtain a categorification of the 
(1-variable) $\mathfrak{so}_6$-Kauffman polynomial~\cite{Wu}.

It turns out that a consequence of this difficulty is a new result about the BMW algebra. 
The aforementioned direct sum decompositions do exist for the 1-variable specializations $a=q^N$ 
and give a categorification of the 1-variable specialization of the BMW algebra.
Moreover, in the 1-variable picture we can also bring another algebra into the play, the $q$-Schur algebra  $S_q(n,d)$
(regarding $q$-Schur algebras we follow the definitions and conventions of~\cite{MSV}).
In this setting the specialized HOMFLY-PT skein algebra $\skein_q(n,N)$ is related with the representation
 theory of quantum $\mathfrak{sl}_N$ and we use this to prove it embeds in the Schur algebra.
More precisely, in the 1-variable picture the HOMFLY-PT skein algebra describes
the algebra of intertwiners between tensor products of the fundamental representation of 
 $U_q(\mathfrak{sl}_N)$ (which is $V=\bQ^n$) and its dual.
The isomorphism between $V^*$ and the wedge power $\wedge^{N-1}V$ induces an injection of the 
HOMFLY-PT skein algebra into the $q$-Schur algebra and the 
composite with the Jaeger homomorphism results in the following  
(see Proposition~\ref{prop:injschur} for the details):

\medskip

\n{\bf Theorem 2.} \emph{There is a injective homomorphism of algebras from the BMW algebra to the $q$-Schur algebra.}

\medskip

\n One of the consequences is that we can use the tools of the $q$-Schur categorification to 
produce another categorification of the 1-variable specialization $\bmw_q(n,N)$ (another categorification, 
different in flavour from this one, comes from applying the results of~\cite{KR1} to $\skein_q(n,N)$).

\medskip

The plan of the paper is the following.
In Section~\ref{sec:algebras} we introduce two of the algebras involved in Jaeger's theorem,
the BMW algebra and the HOMFLY-PT skein algebra. They are introduced diagrammatically from
algebras of tangles induced from the corresponding link polynomials.
In Section~\ref{sec:jaeger} we state Jaeger's theorem connecting the Kauffman and the HOMFLY-PT 
link polynomials, explain the proof, and state and prove a version of Jaeger's theorem for the underlying algebras.
In Section~\ref{sec:emb} we set $a=q^N$ and start working with the 1-variable specialization of the algebras mentioned above.
This specialization allows the introduction of the $q$-Schur algebra $S_q(n,d)$. 
We prove in this section that $\bmw_q(n,N)$ embeds in a certain 
direct sum of $q$-Schur algebras.
Finally in Section~\ref{sec:cat} we explain how to combine the results of the previous sections with the 
categorifications of $\skein_q(n,N)$ and $S_q(n,d)$ to obtain categorifications of $\bmw_q(n,N)$.
These categorifications do not have the \emph{Krull-Schmidt property} though, 
which is a desirable  property 
for reasons we explain later.
Nevertheless we propose an extension of these categorifications which 
we conjecture to have this property.

\medskip

We have tried to make the paper self contained with the exception of 
Subsections~\ref{ssec:matfac} and~\ref{ssec:scat} where we assume familiarity
with the papers~\cite{KR1, KR2, MSV}.

%
\section{BMW and Skein algebras}\label{sec:algebras}

\subsection{The Birman-Murakami-Wenzl algebra $\bmw_n(a ,q)$}      %
\label{ssec:bmw}                                                   %

Let $R=\bC(a,q)$ be the field of rational functions in two variables and $n$ a positive integer. 
We give two different presentations of the Birman-Murakami-Wenzl algebra $\bmw_n(a,q)$ over $R$. 
Both of them have a diagrammatic description.
The first one is usually known in the literature as the \emph{Kauffman tangle algebra}~\cite{Kauff} 
and consists of a quotient of the framed tangle algebra by local relations
which come from the Kauffman polynomial. 
\begin{defn}
Define $\bmw^{\tau}_n(a,q)$ as the free algebra over $R$ generated by $(n,n)$-framed unoriented 
tangles up to regular isotopies modulo the local relations~\eqref{eq:bmw-kauff} below.
\begin{equation}
\label{eq:bmw-kauff}
\begin{split}
\figins{-10}{0.35}{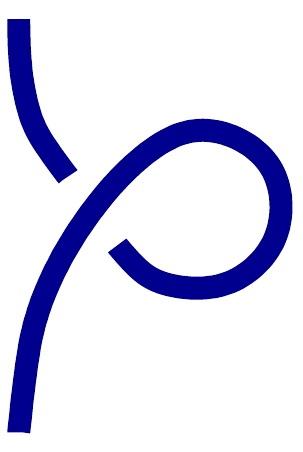}\ \ 
= \ a^{-2}q\ \ 
\figins{-10}{0.35}{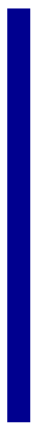}
&\mspace{70mu}
\figins{-10}{0.35}{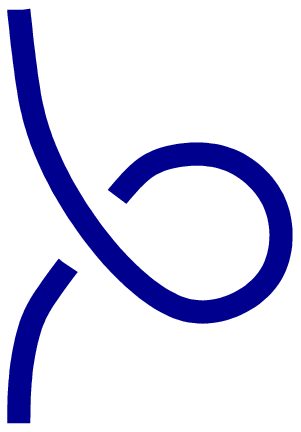}\ \ 
= \ a^2q^{-2}\ \ 
\figins{-10}{0.35}{one}
\\[2ex]
\figins{-10}{0.35}{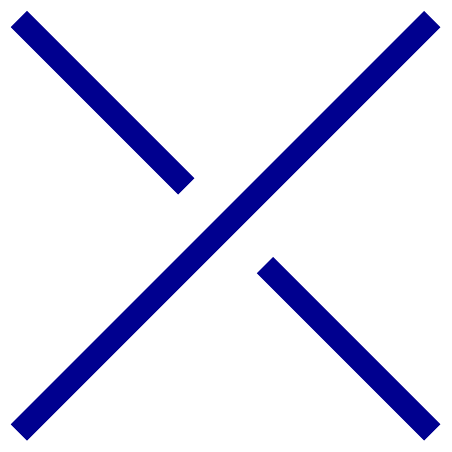} \ \ - \ \ 
\figins{-10}{0.35}{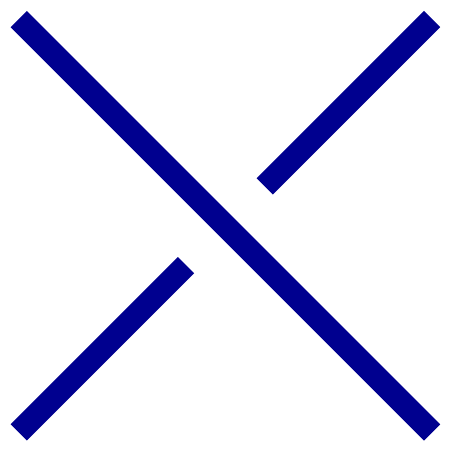}
=&\  (q - q^{-1})\biggl(\ 
\figins{-10}{0.35}{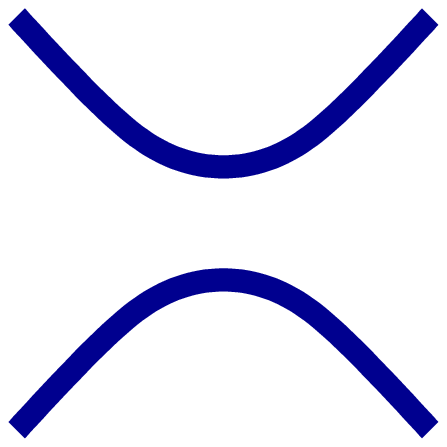}\  -  \
\figins{-10}{0.35}{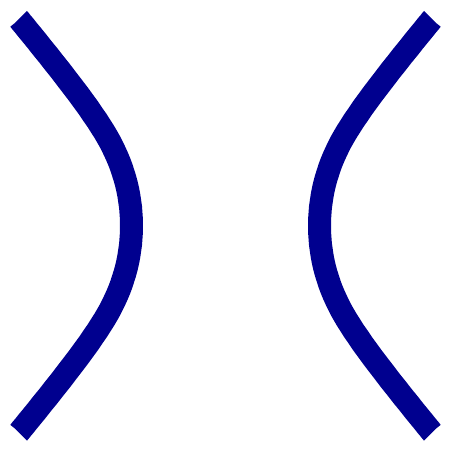}\
\biggr)
\end{split}
\end{equation}
\end{defn}
As usual in the presentation of algebras by diagrams the product is given by stacking one tangle over another.
We always read diagrams from bottom to top and this means the product
$ab$ corresponds to stacking the diagram for $a$ on the top 
of the diagram for $b$.
We will always assume this throughout this paper.
Notice the unusual normalization we are using 
(this is forced by Jaeger's theorem and will be clear later). 
 
\smallskip

To give the next presentation, due to 
Kauffman and Vogel~\cite{kauff-vogel}, we need to introduce some concepts.
A \emph{$(n,n)$ 4-graph} is a planar graph with $2n$ univalent vertices 
and the rest of the vertices are  4-valent. 
It can be embedded in a rectangle with $n$ of the univalent vertices lying on the bottom segment and 
$n$ lying on the top one.
We think of an $(n,n)$ 4-graph as the singularization of a $(n,n)$
 unoriented tangle diagram, which is the graph obtained by applying the transformation
\begin{equation*}
\figins{-10}{0.35}{Xing-p}\ \ \longmapsto\ \ 
\figins{-10}{0.35}{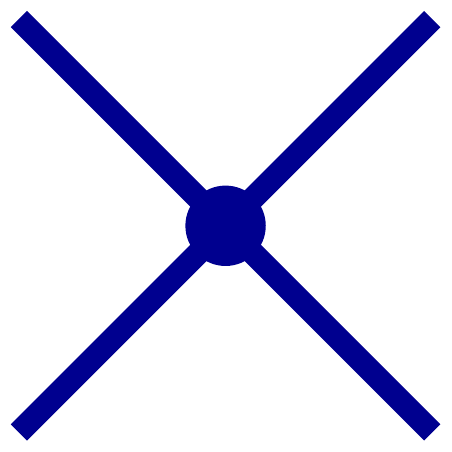}\ \text{\reflectbox{$\longmapsto$}}\ \
\figins{-10}{0.35}{Xing-n}
\end{equation*}
to all its crossings.

We also introduce some short-hand notation in order to simplify 
many of the expressions that we will have to handle.
For a formal parameter $a$ and for $n,k\in\bZ$ we denote $[a^n,k]:=\frac{a^nq^{-k}-a^{-n}q^k}{q-q^{-1}}$.
We allow a further simplification by writing $[a^n]$ instead of $[a^n,0]$. 
Moreover, when dealing with 1-variable specializations we
use $[m+k]$ for $[q^m,k]=\tfrac{q^{m+k} -q^{-m-k} }{q-q^{-1}}$, 
which is the usual \emph{quantum integer}.

\begin{defn}
Define $\bmw_n(a,q)$ as the free algebra over $R$ generated by $(n,n)$
 4-graphs up to planar isotopies fixing the univalent vertices modulo the following local relations:
\begin{align}\nonumber
\figins{-13}{0.45}{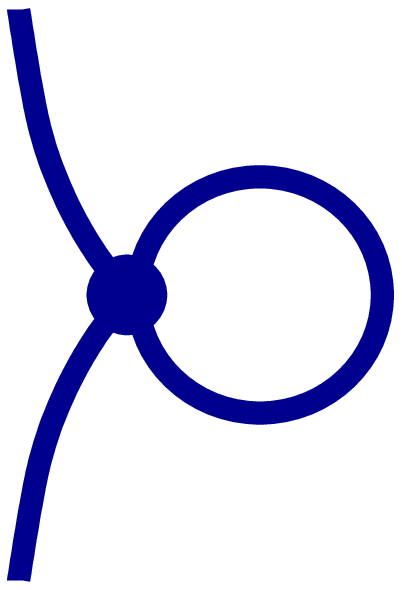}\
=& \ 
[a](q^2a^{-1}+q^{-2}a)\ \ 
\figins{-13}{0.45}{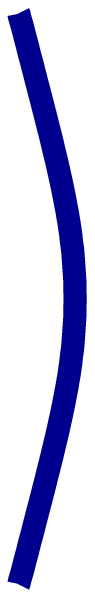}
\displaybreak[0]\\[2ex]\nonumber
\figins{-19}{0.60}{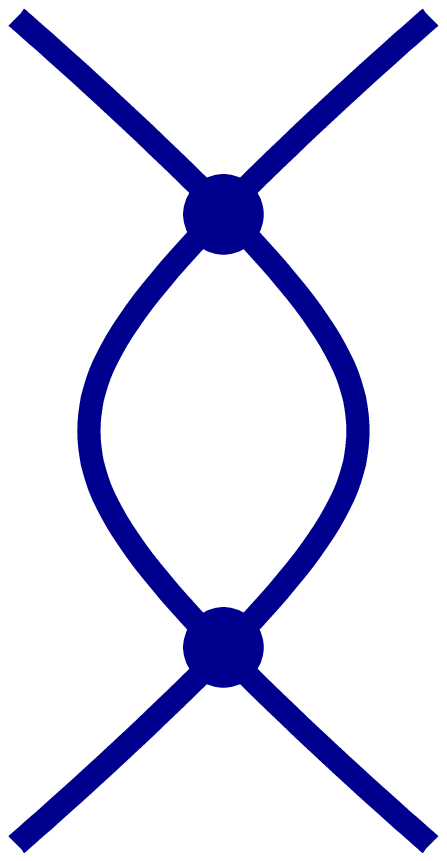}\
=&\ (q+q^{-1})
\figins{-19}{0.60}{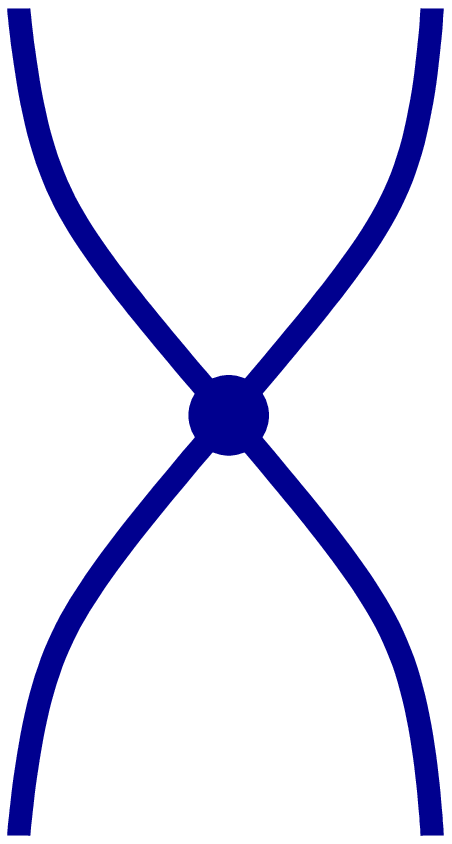}\
+\ \biggl( [a^3,-3] + 1\biggr)\
\figins{-19}{0.60}{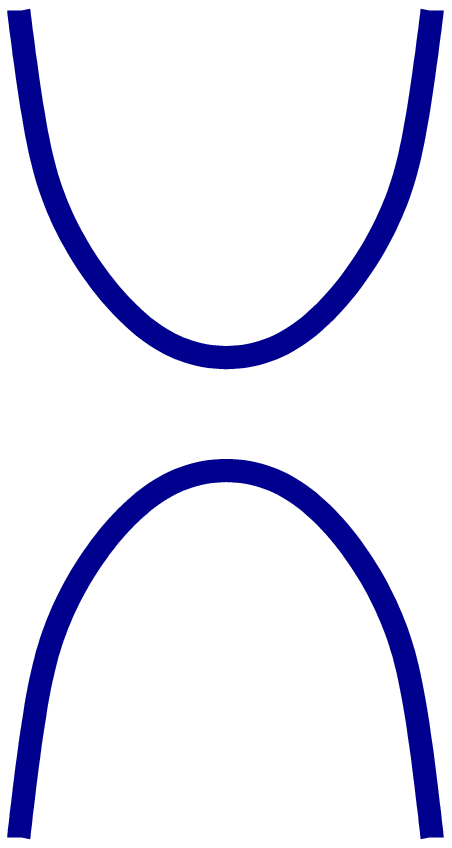}
\end{align}

\begin{equation}
\begin{split}
\figins{-19}{0.60}{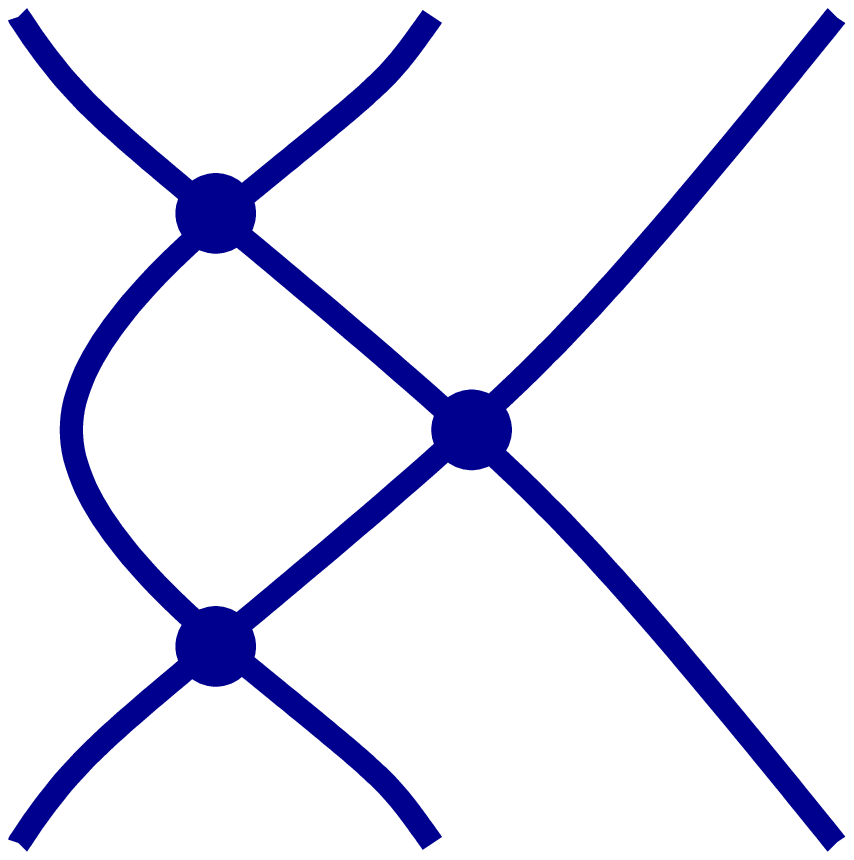}\ +\ 
\figins{-19}{0.60}{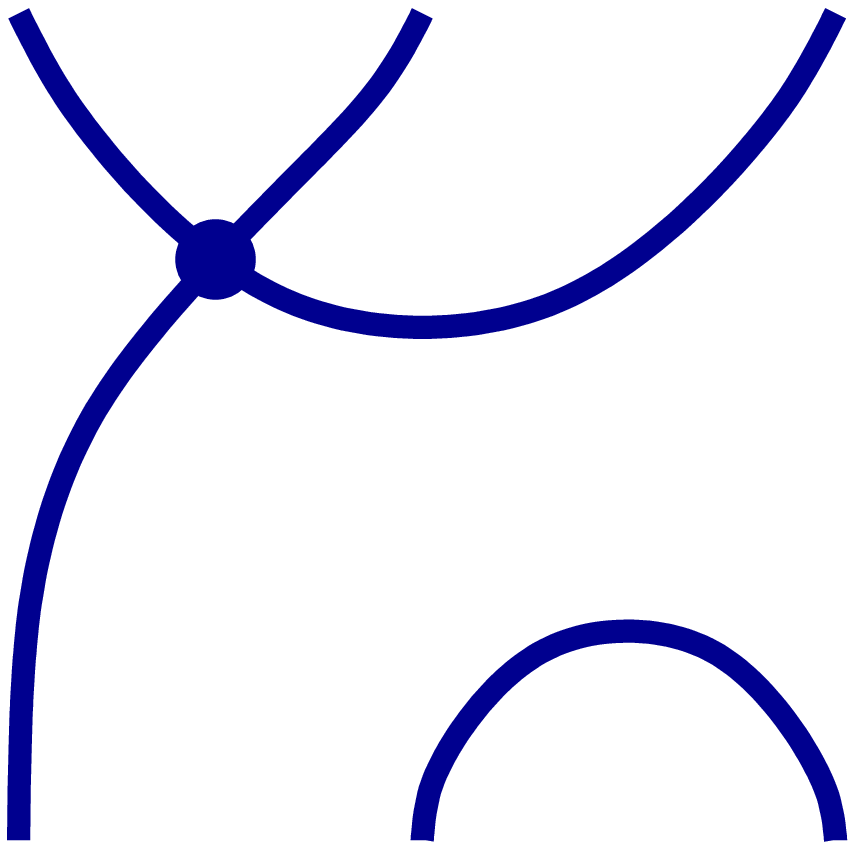}\ +\ &
\figins{-19}{0.60}{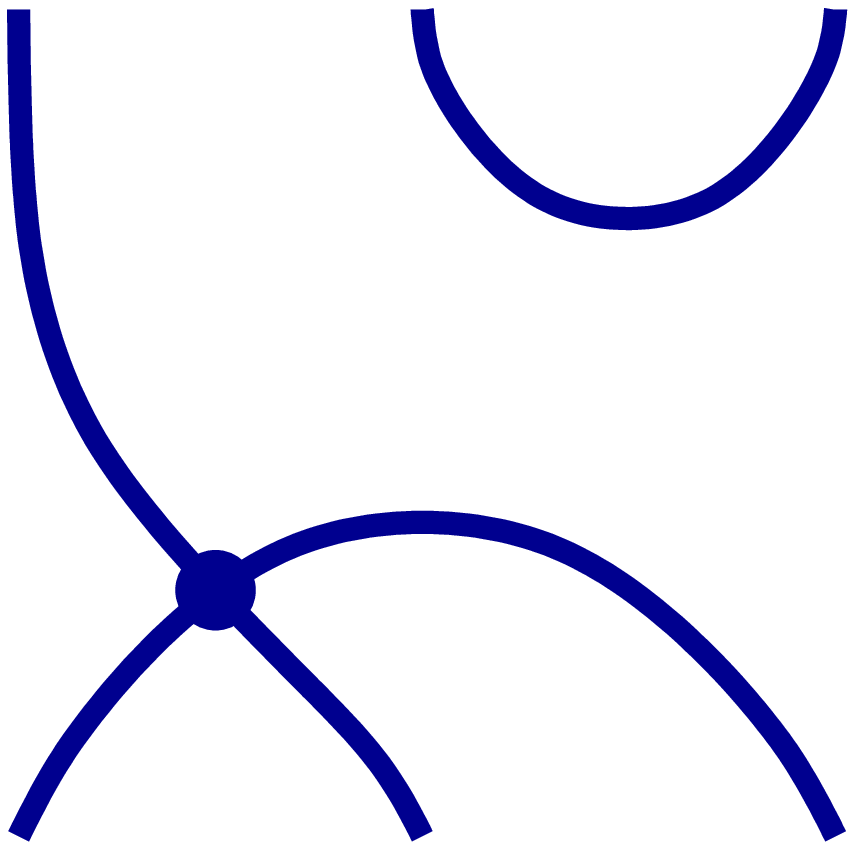}\ +\ 
\figins{-19}{0.60}{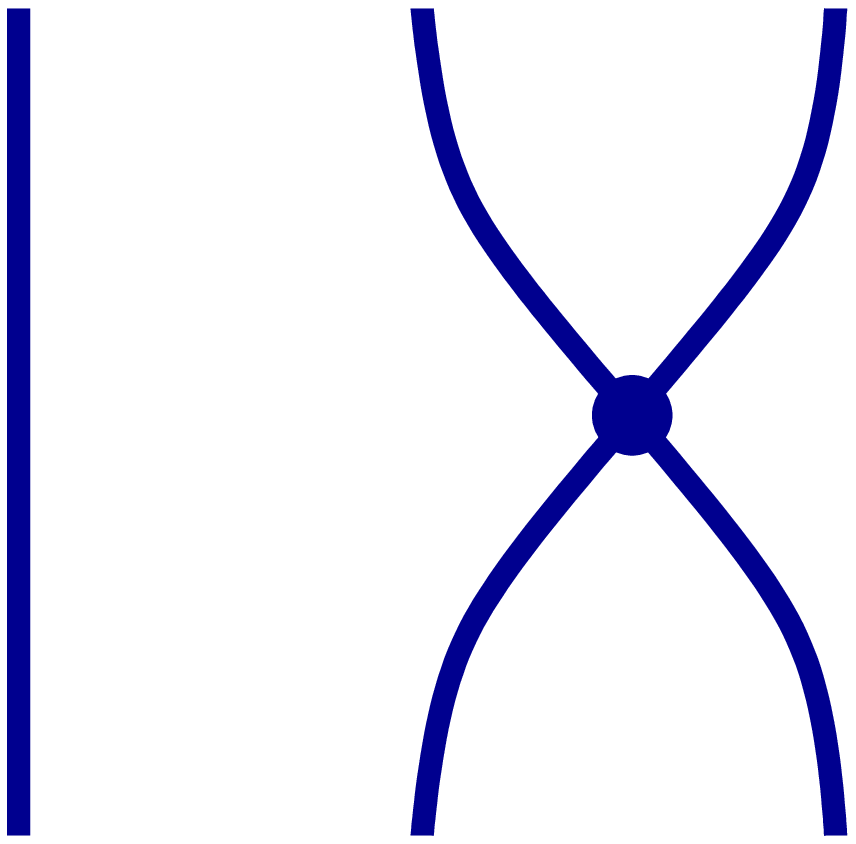}\ +\ [a^2,-4]\
\figins{-19}{0.60}{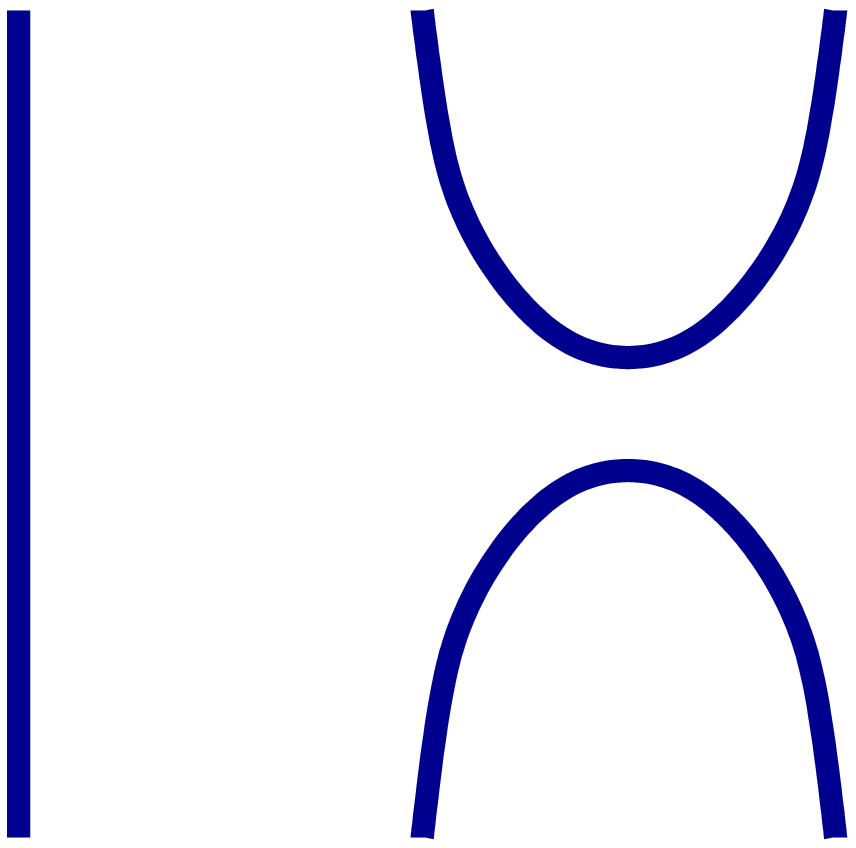}
\\[2ex]\mspace{100mu}
=\ 
\figins{-19}{0.60}{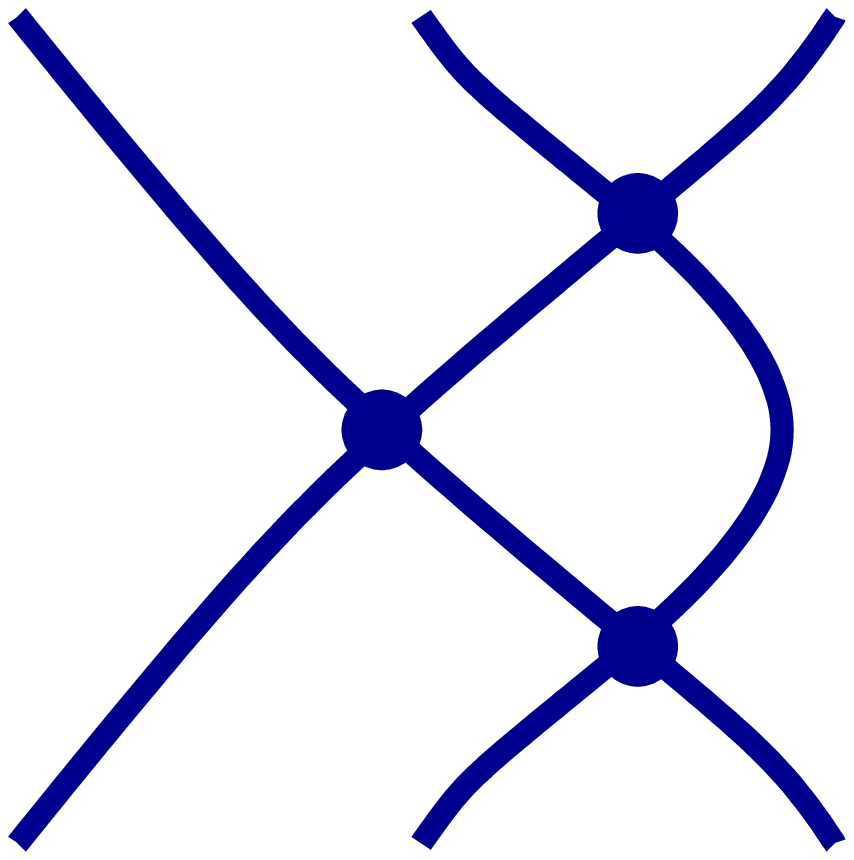}\ +\ &
\figins{-19}{0.60}{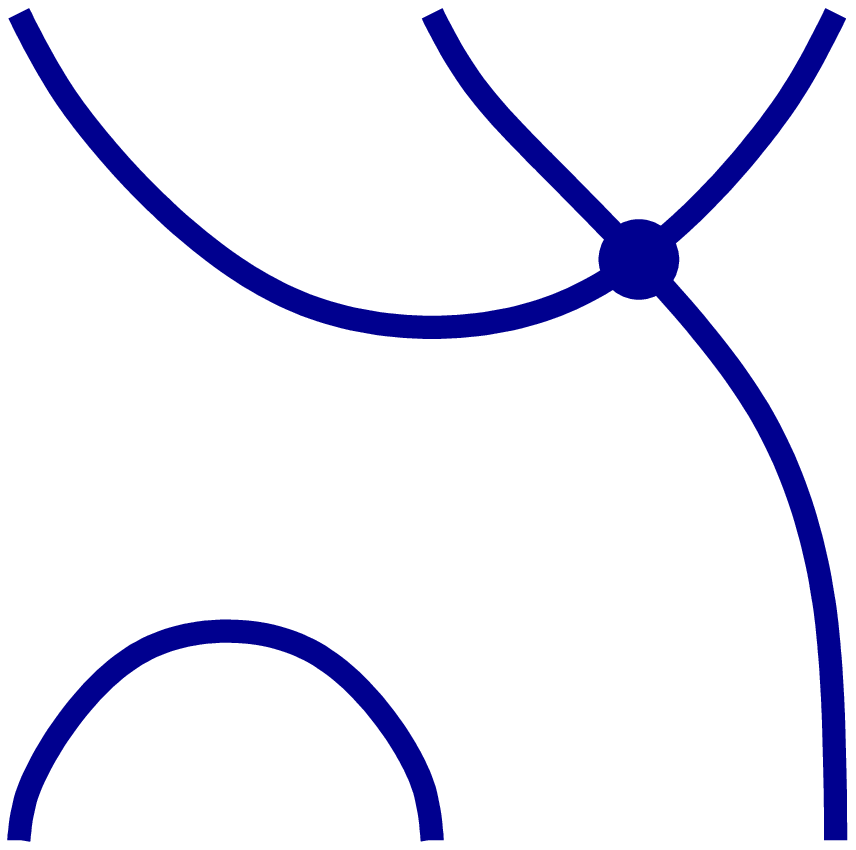}\ +\  
\figins{-19}{0.60}{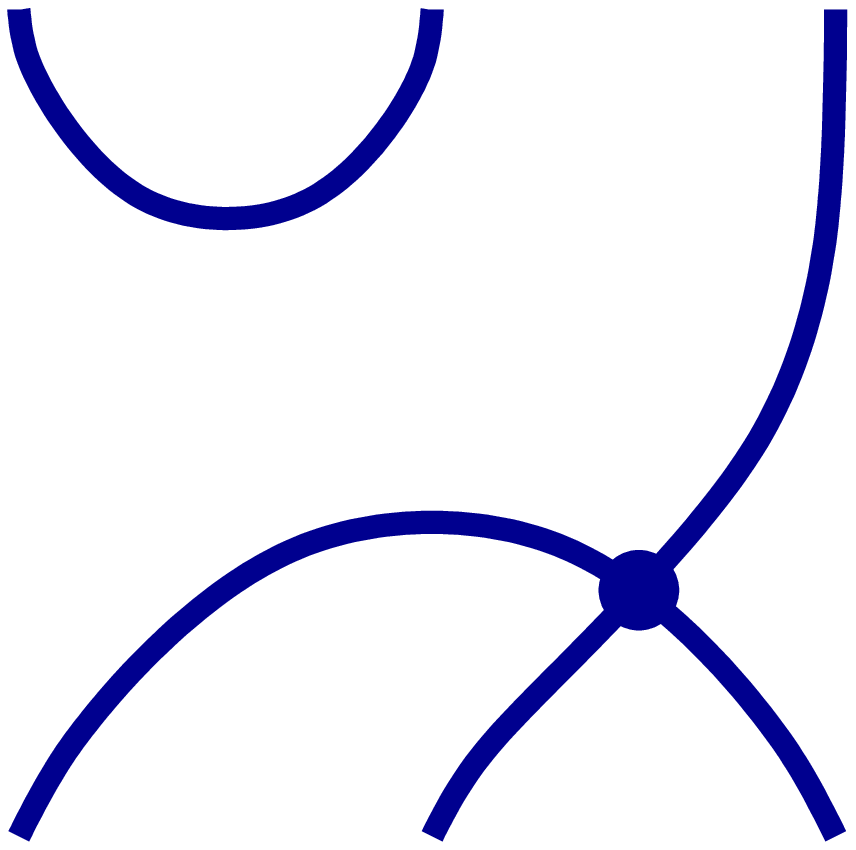}\ +\ 
\figins{-19}{0.60}{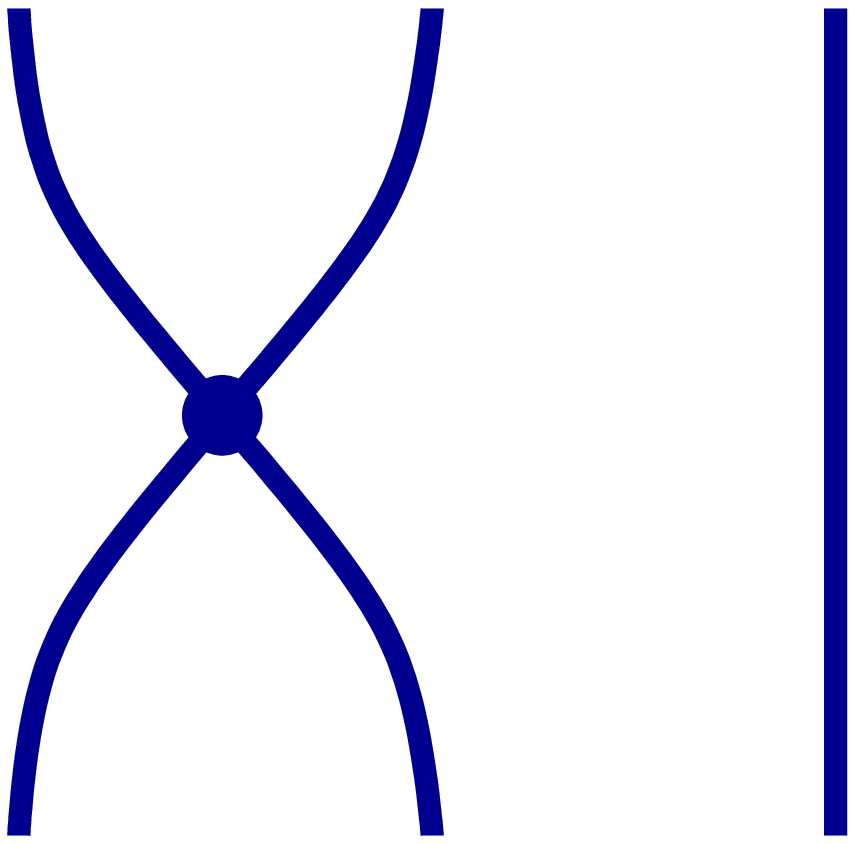}\ +\ [a^2,-4]\ 
\figins{-19}{0.60}{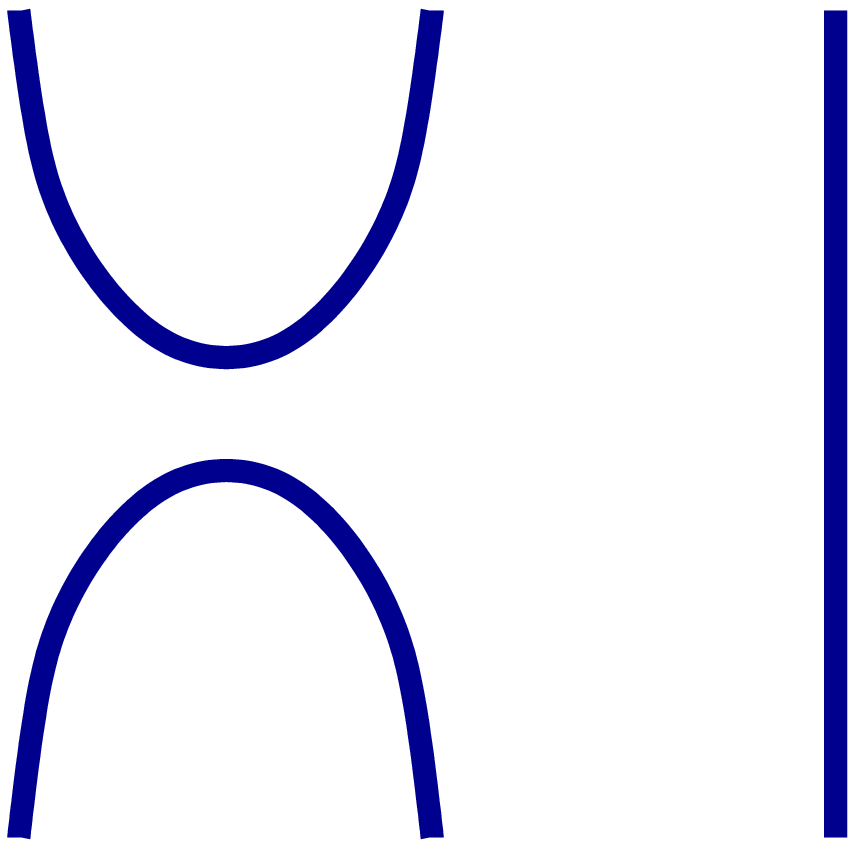}
\end{split}
\end{equation}
\begin{align*}
\figins{-12.5}{0.43}{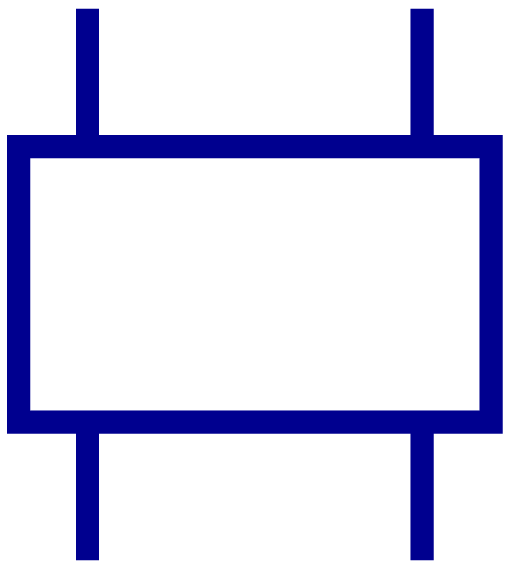}\ \figins{-8}{0.30}{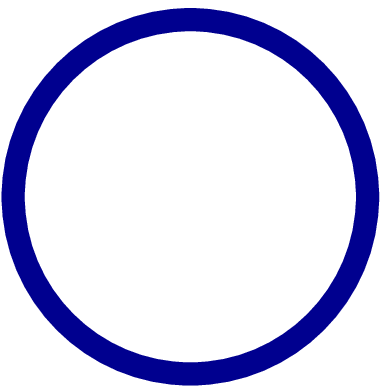}\ 
&=\  \delta\
\figins{-12.5}{0.43}{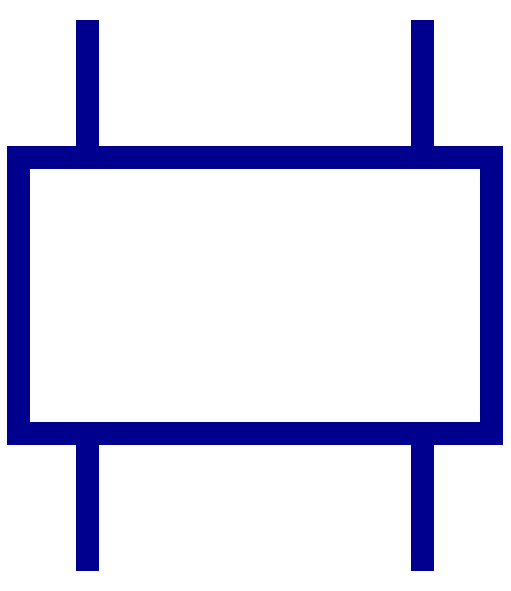}
\displaybreak[0]\\[2ex]\nonumber
\figins{-8}{0.30}{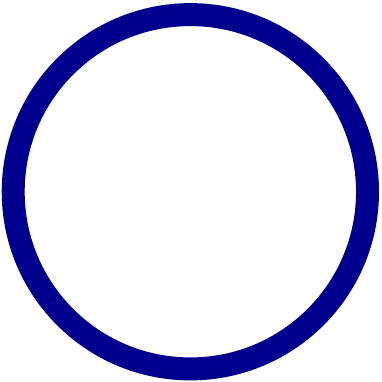}\ 
&=\  \delta
\end{align*}
\label{kauffgraphrelation}
where $\delta=[a^2,-1]+1$.
\end{defn}

We remark that the special elements
\begin{align}
\label{eq:genX}
\rho_i\ &=\quad 
\labellist
\pinlabel $\dotsc$ at 110 125
\pinlabel $\dotsc$ at 665 125
\tiny\hair 2pt
\pinlabel $1$   at   5 -25
\pinlabel $i$   at 325 -25
\pinlabel $i+1$ at 445 -27
\pinlabel $n$   at 765 -25
\endlabellist
\figins{-19}{0.60}{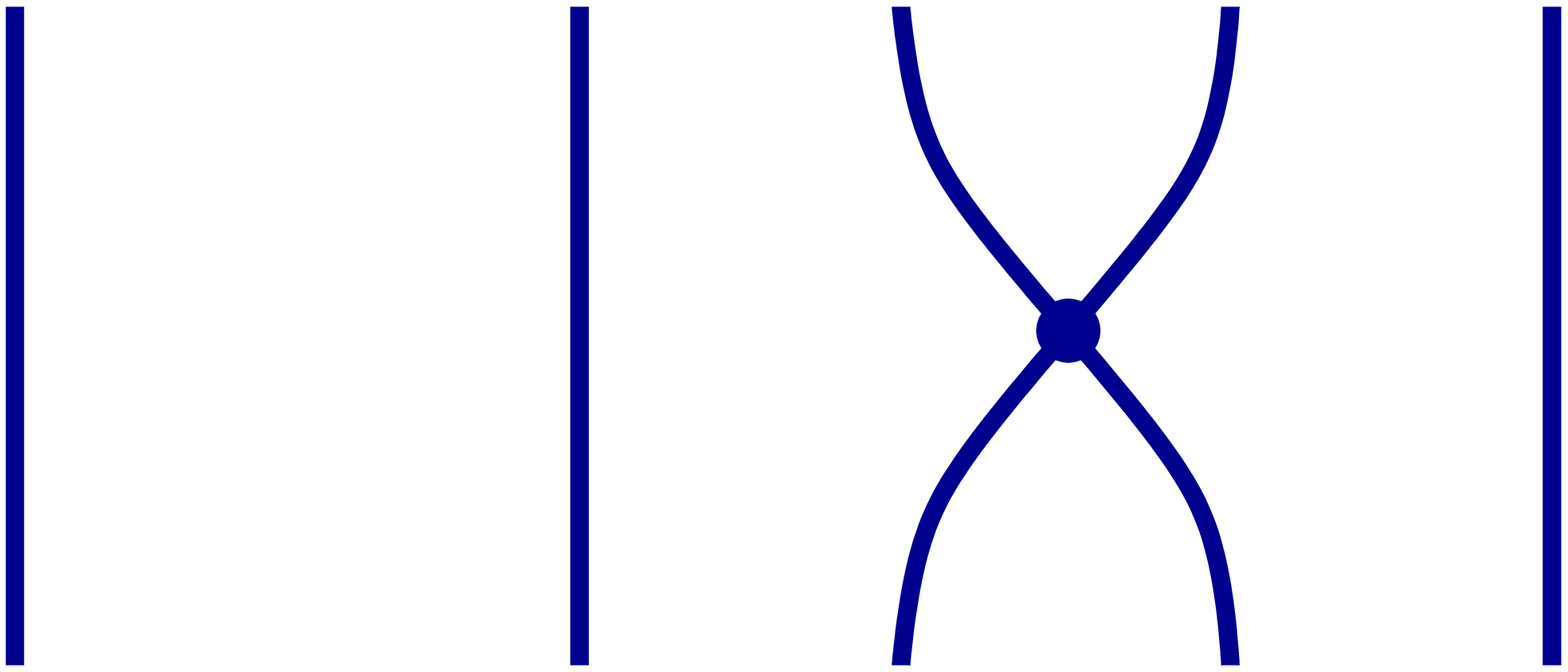}
\rlap{\qquad\quad $(i=1,\dotsc ,n-1)$}
\displaybreak[0]\\[2ex]
\label{eq:genU}
e_i\ &=\quad 
\labellist
\pinlabel $\dotsc$ at 110 125
\pinlabel $\dotsc$ at 665 125
\tiny\hair 2pt
\pinlabel $1$   at   5 -25
\pinlabel $i$   at 325 -25
\pinlabel $i+1$ at 445 -27
\pinlabel $n$   at 765 -25
\endlabellist
\figins{-19}{0.60}{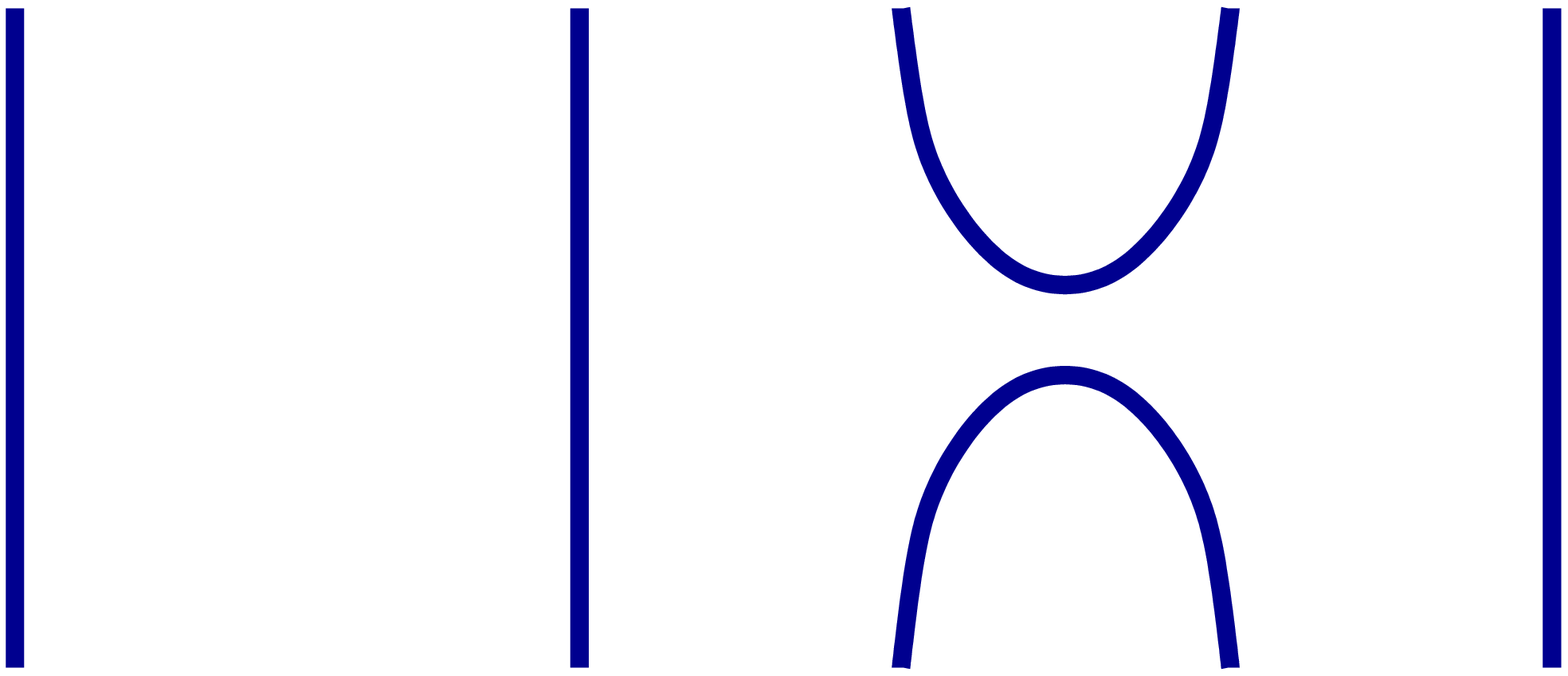}
\rlap{\qquad\quad $(i=1,\dotsc ,n-1)$}
\intertext{together with}
\label{eq:genum}
1\ &=\quad 
\labellist
\pinlabel $\dotsc$ at 110 125
\pinlabel $\dotsc$ at 665 125
\endlabellist
\figins{-19}{0.60}{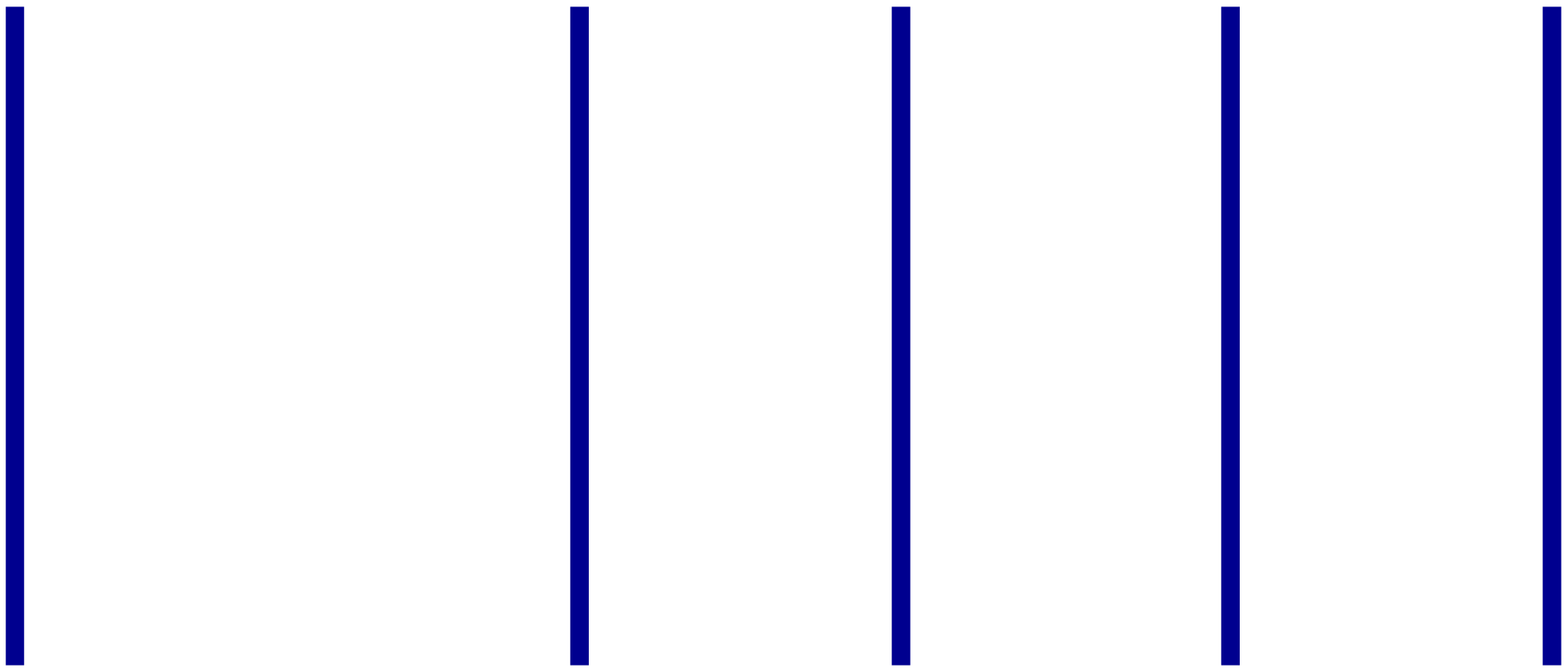}
\vspace*{2ex}
\end{align}
generate $\bmw_n(a,q)$ and can be used to give a presentation of  $\bmw_n(a,q)$ by generators and relations (see \cite{morton}).

\medskip

The algebras $\bmw^{\tau}_n(a,q)$ and $\bmw_n(a,q)$ are isomorphic; it follows directly from 
the results in~\cite{kauff-vogel}.
\begin{lem}
The homomorphism $\phi\colon\bmw^{\tau}_n(a,q)\to\bmw_n(a,q)$ given by
\begin{align*}
\figins{-10}{0.35}{Xing-p} 
&\longmapsto\ \ 
q\ \figins{-10}{0.35}{cupcap}\  -  \
\figins{-10}{0.35}{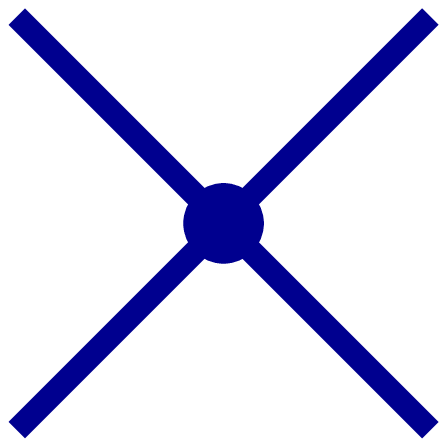}\ + \
q^{-1}\figins{-10}{0.35}{cupcap-vert}
\intertext{and} 
\figins{-10}{0.35}{cupcap}
&\longmapsto\ \ 
\figins{-10}{0.35}{cupcap}
\end{align*}
is an isomorphism of algebras.
\label{KViso}
\end{lem}

The homomorphism $\phi$ is the basis of Kauffman-Vogel's state-sum model 
for the Kauffman polynomial~\cite{kauff-vogel}.

\subsection{The HOMFLY-PT skein algebra $\skein_n(a,q)$}         %
\label{ssec:skein}                                                   %

In this subsection we define another algebra, denoted $\skein_n(a,q)$, which we call the \emph{HOMFLY-PT skein
 algebra}~\cite{kauff-vogel}  and is very similar in spirit to $\bmw_n(a,q)$.  
As in the previous subsection we give two presentations, one of as the free algebra 
over $R$ generated by $(n,n)$
 oriented tangles up to isotopies modulo the HOMFLY-PT skein relations,
and another one using 4-valent oriented graphs.

\medskip

\begin{defn}
Define $\skein^{\tau}_n(a,q)$ as the free algebra over $R$ generated by isotopy classes 
of $(n,n)$-oriented tangles modulo the (HOMFLY-PT) local relation (\ref{HOMFLYPTskein}).
\begin{equation}
a\ \     \figins{-10}{0.35}{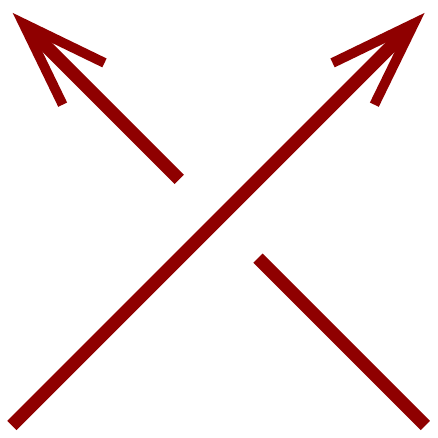}\ \ - \ \ 
a^{-1}\ \ \figins{-10}{0.35}{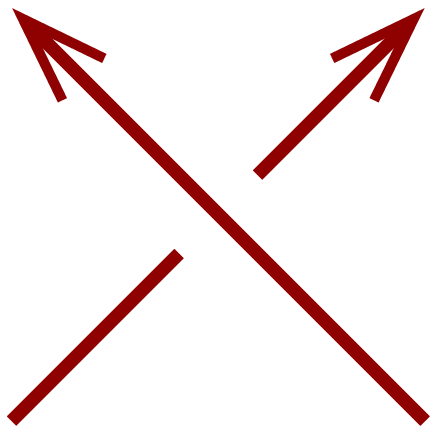}\ \ 
=\ \ (q - q^{-1})\ \
\figins{-10}{0.35}{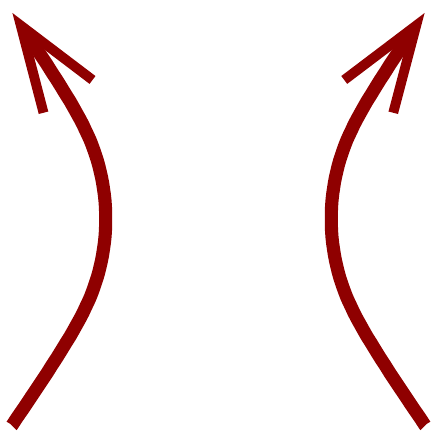} 
\label{HOMFLYPTskein}
\end{equation}
\end{defn}

\medskip

Paralleling the unoriented case we define an \emph{oriented $(n,n)$ 4-graph} 
to be an oriented planar 4-valent graph with $2n$ univalent vertices 
such that the graph can be embedded in a rectangle with $n$ of the univalent vertices lying on the bottom 
segment and $n$ on the top one. In addition, we require the orientation 
near a 4-valent vertex to be as
\begin{equation*}
\figins{-10}{0.35}{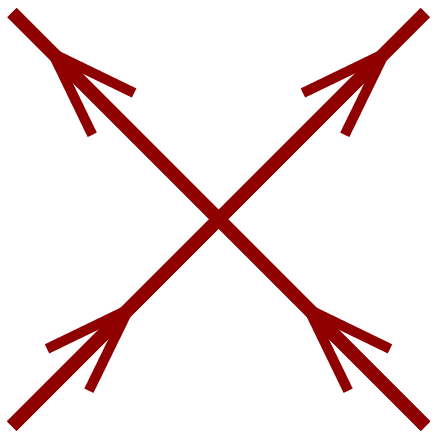}\ .
\end{equation*}
In other words, an oriented $(n,n)$ 4-graph is the singularization of a $(n,n)$
 oriented tangle diagram.
All the oriented 4-valent graphs we consider are of this type.

\begin{defn}
\label{def:moy-aq}
Define the algebra $\skein_n(a,q)$ as the free algebra over $R$ generated by $(n,n)$
 oriented 4-valent graphs up to planar isotopies modded out by the following relations:
\begin{align}
\figins{-13}{0.45}{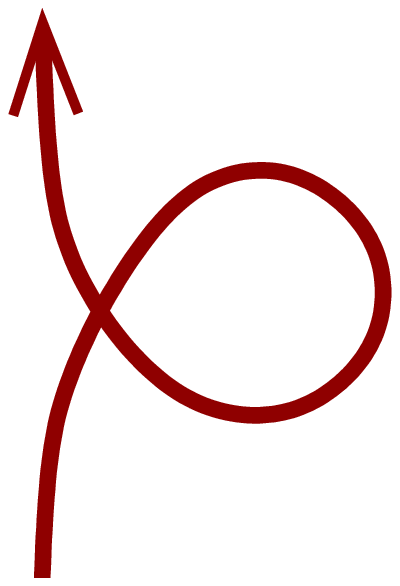}\
&= \ 
[a,-1]\ \ 
\figins{-13}{0.45}{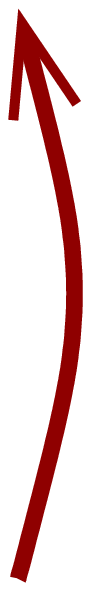}
\label{R1MOY}
\displaybreak[0]\\[2ex]
\figins{-19}{0.60}{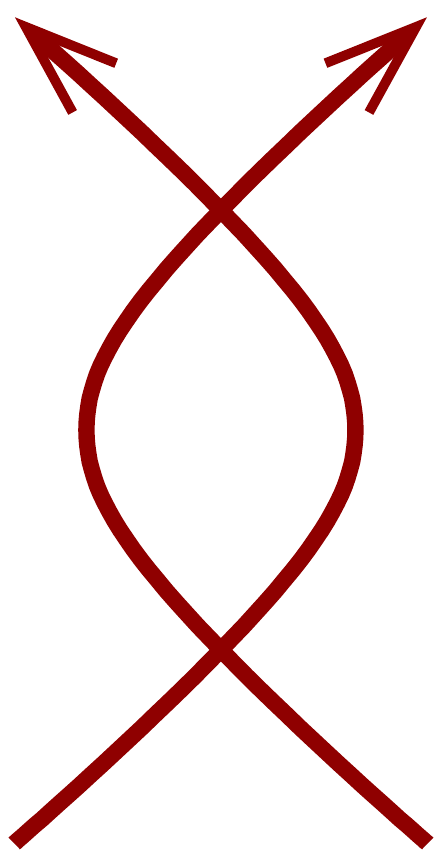}\ 
&=\ (q+q^{-1})\ \
\figins{-19}{0.60}{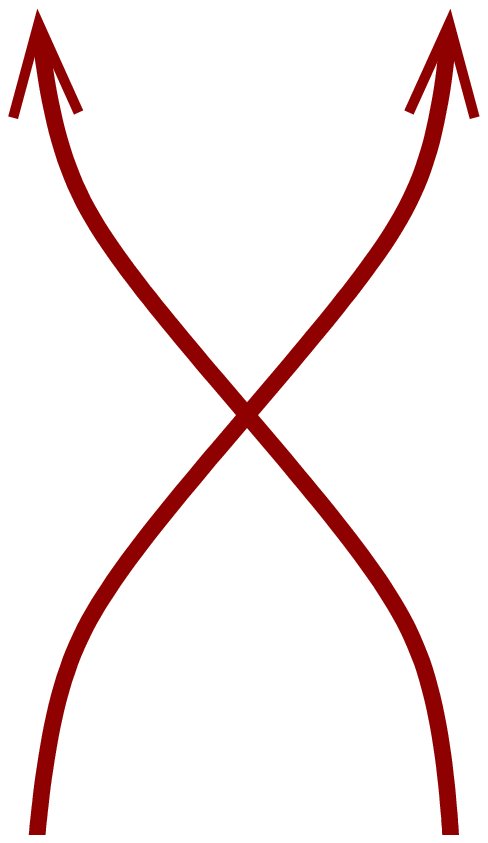}
\label{R2aMOY}
\displaybreak[0]\\[2ex]
\figins{-19}{0.60}{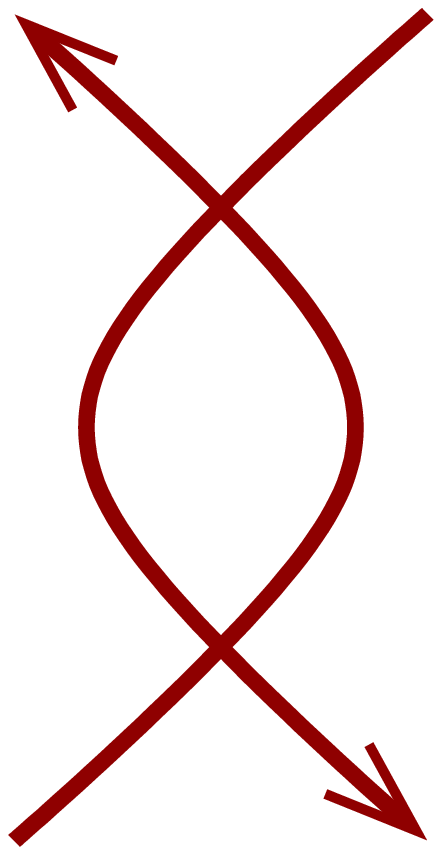}\ 
&=\ 
\figins{-19}{0.60}{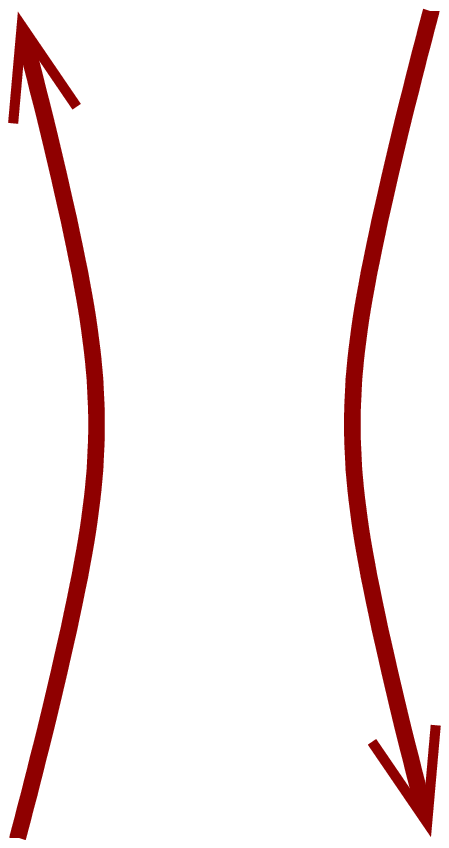}\
+\ [a,-2]\ \
\figins{-19}{0.60}{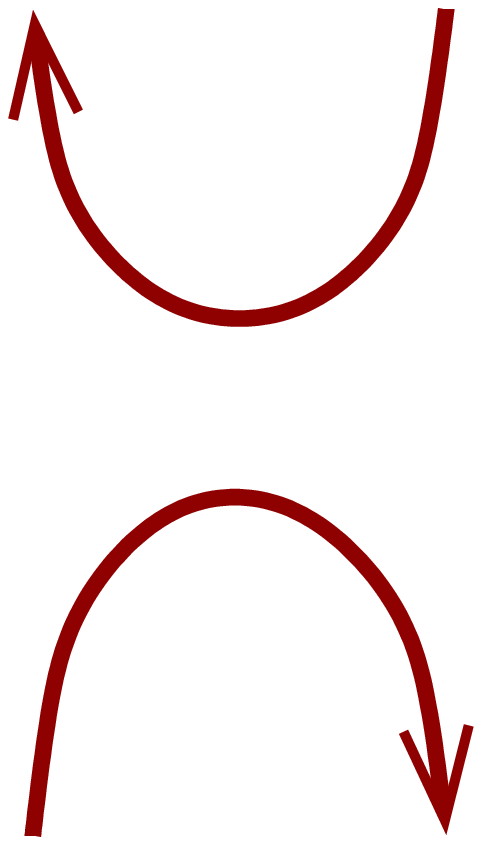}\
\label{R2bMOY}
\displaybreak[0]\\[2ex]
\figins{-19}{0.60}{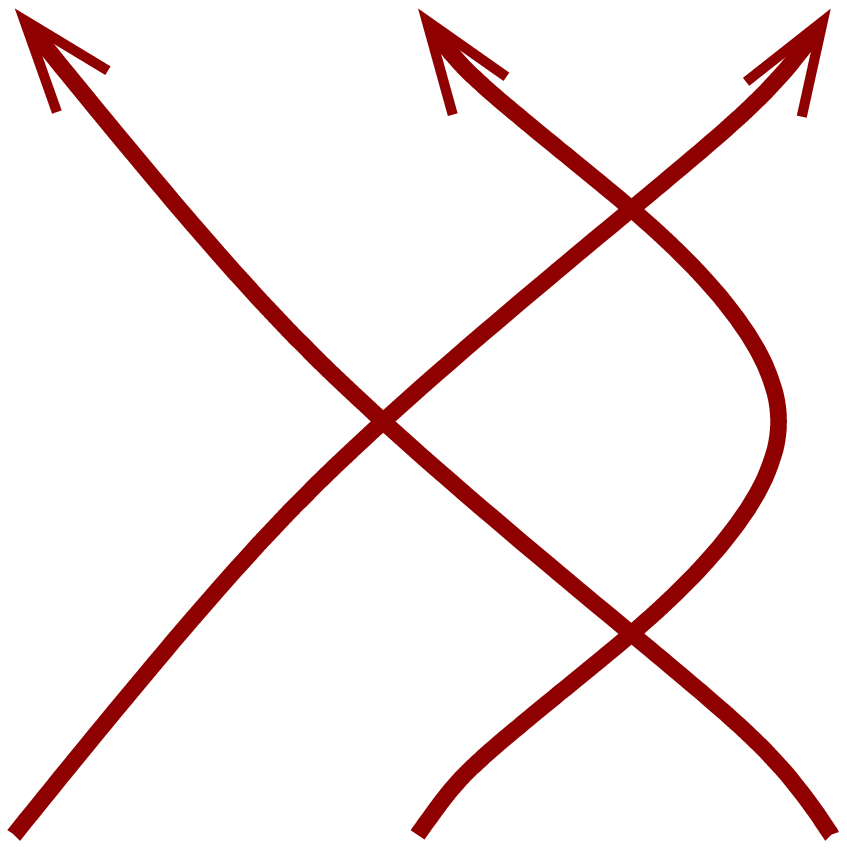}\ +\ 
\figins{-19}{0.60}{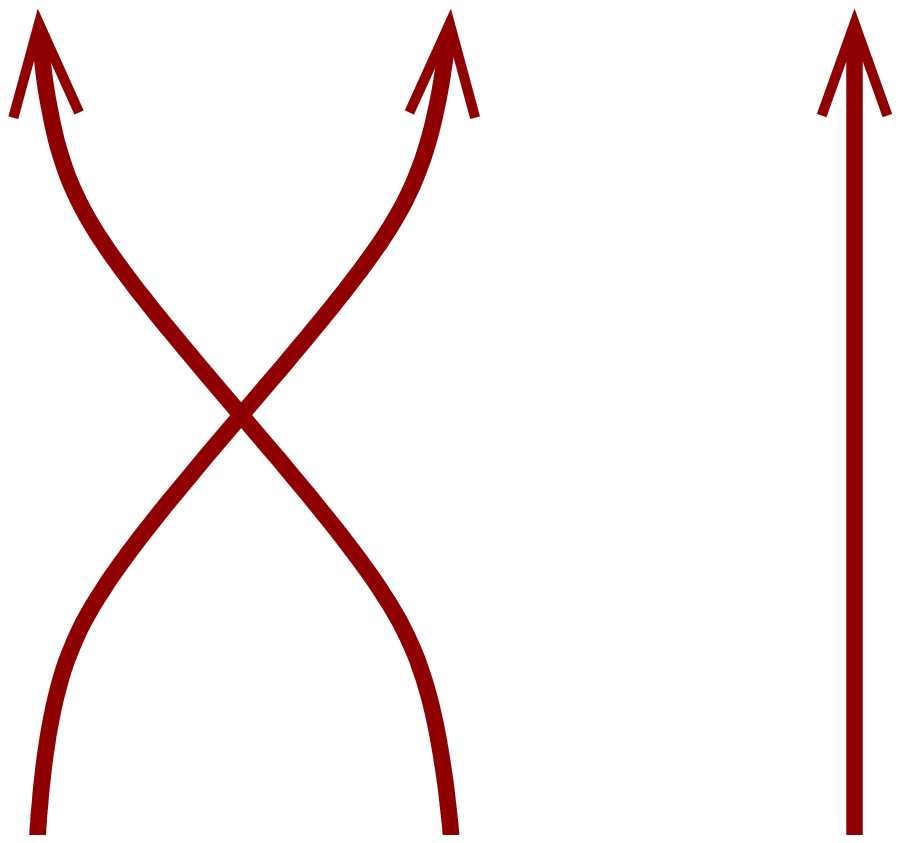}\
&=\ 
\figins{-19}{0.60}{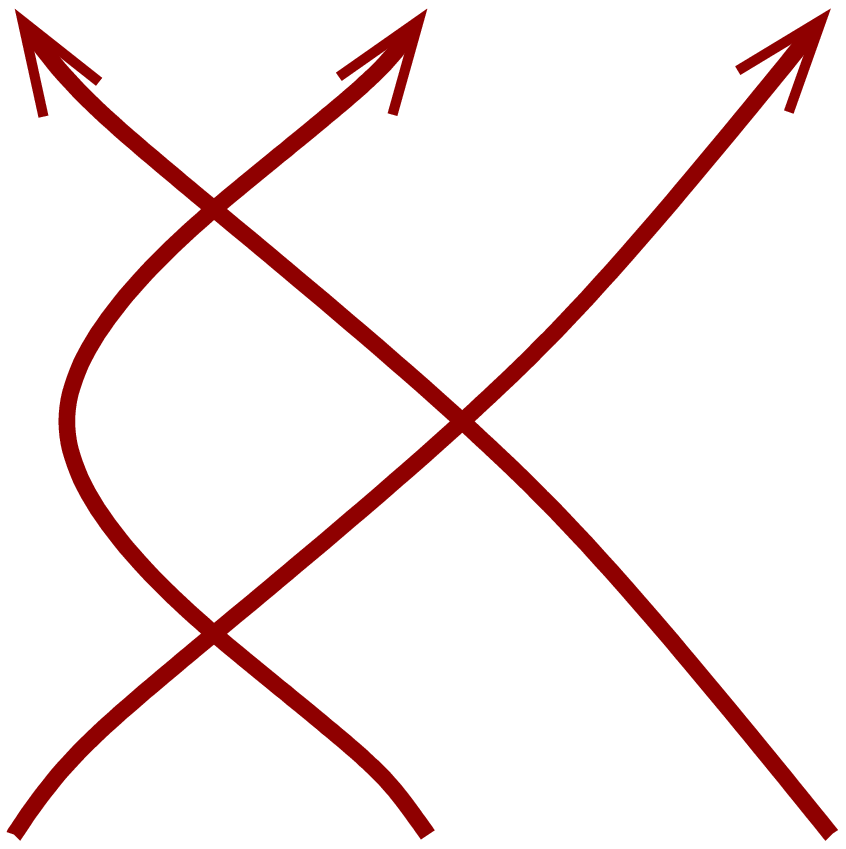}\ +\ 
\figins{-19}{0.60}{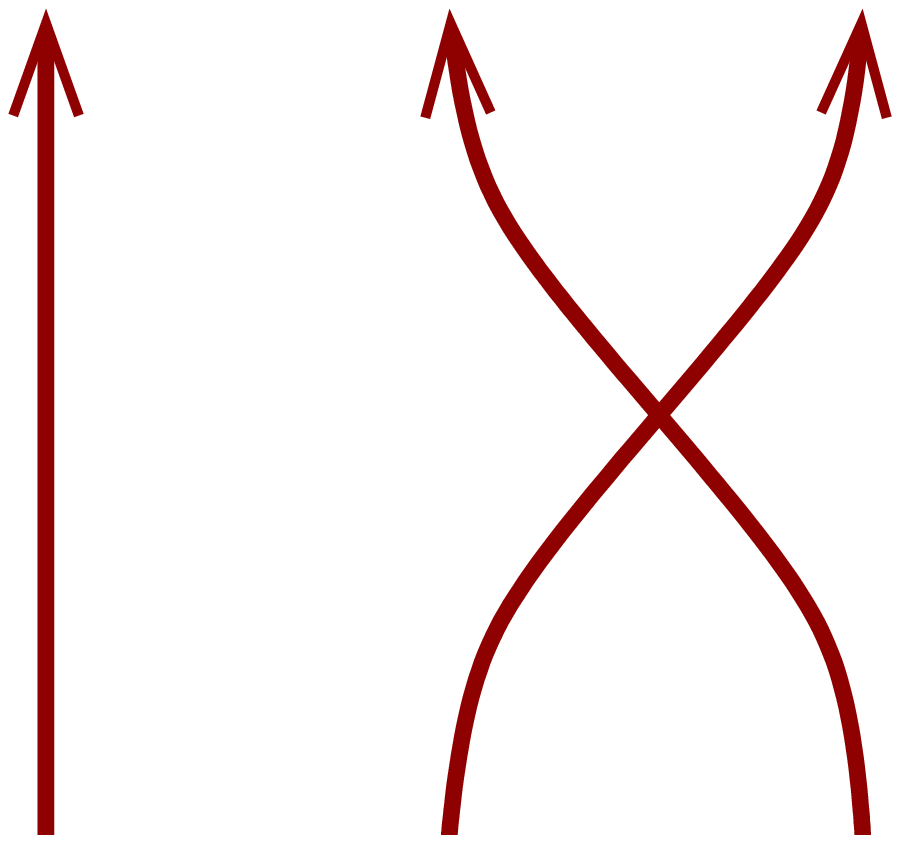}
\label{R3MOY}
\displaybreak[0]\\[2ex]
\figins{-8}{0.30}{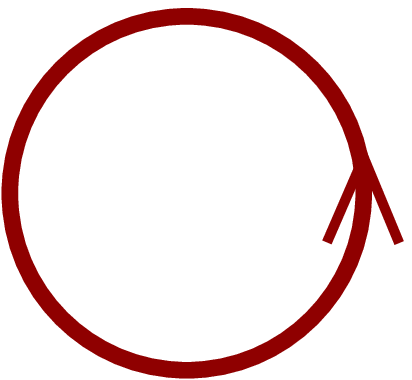}\ =\  [a]
\mspace{40mu} &\mspace{60mu}
\figins{-8}{0.30}{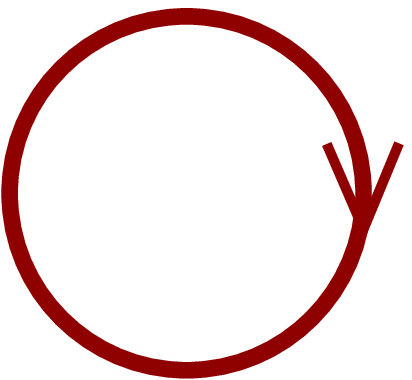}\ =\  [a]
\displaybreak[0]\\[2ex]
\figins{-12.5}{0.43}{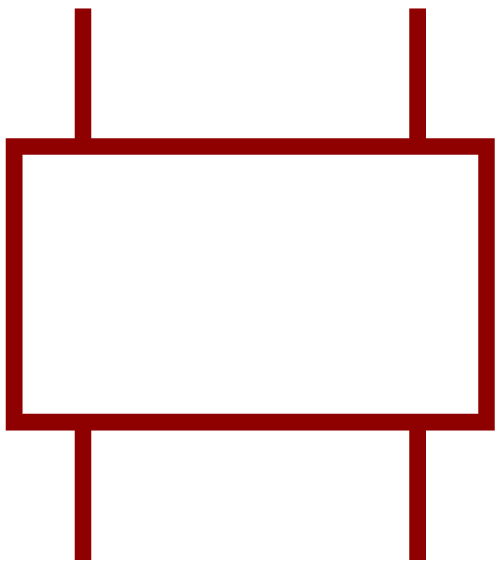}\ \figins{-8}{0.30}{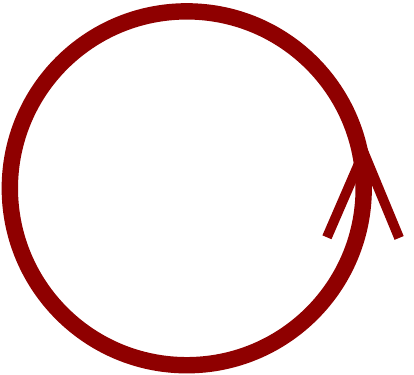}\ 
&=\  [a]\
\figins{-12.5}{0.43}{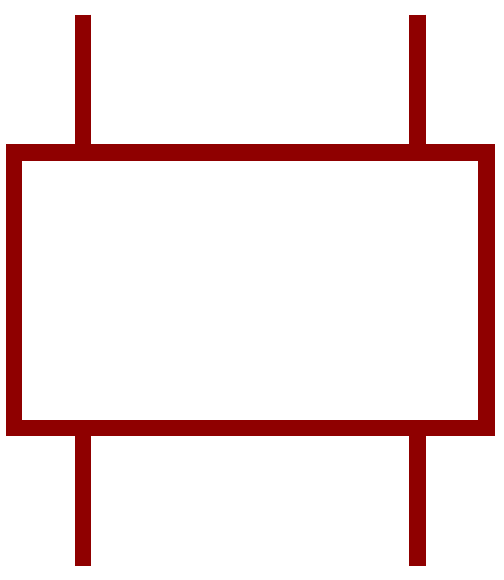}
\label{SPLITMOY}
\end{align}
\end{defn}

As before the product structure is given by stacking one graph over the other being zero if the orientations do not match.
In addition,
\begin{equation}
\figins{-19}{0.60}{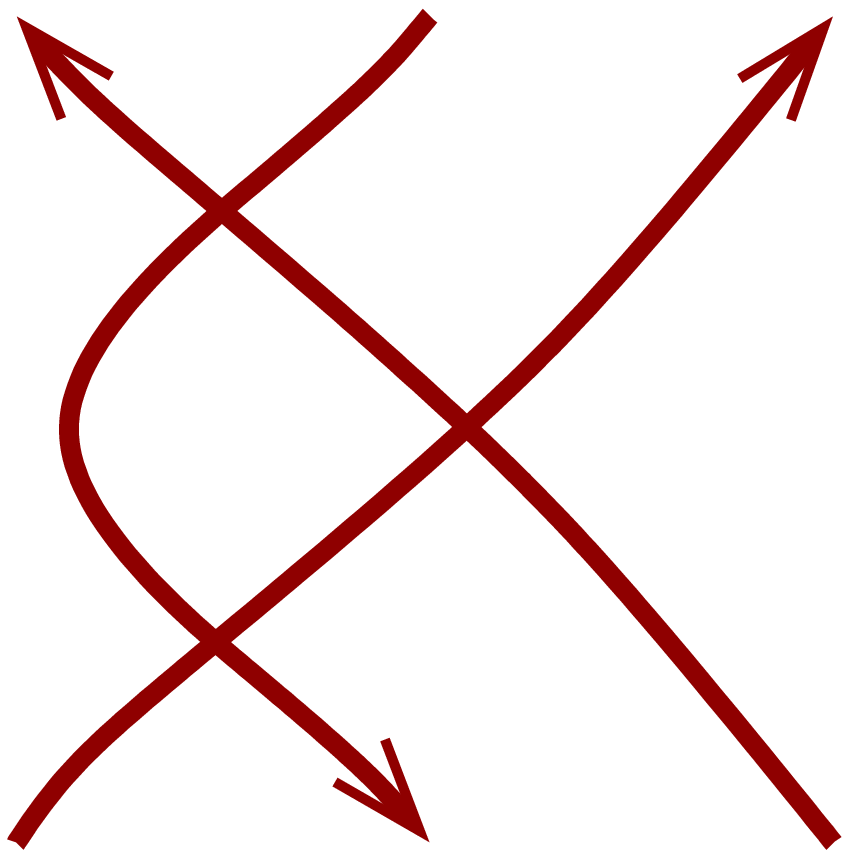}\
+\ [a,-3]\ \figins{-19}{0.60}{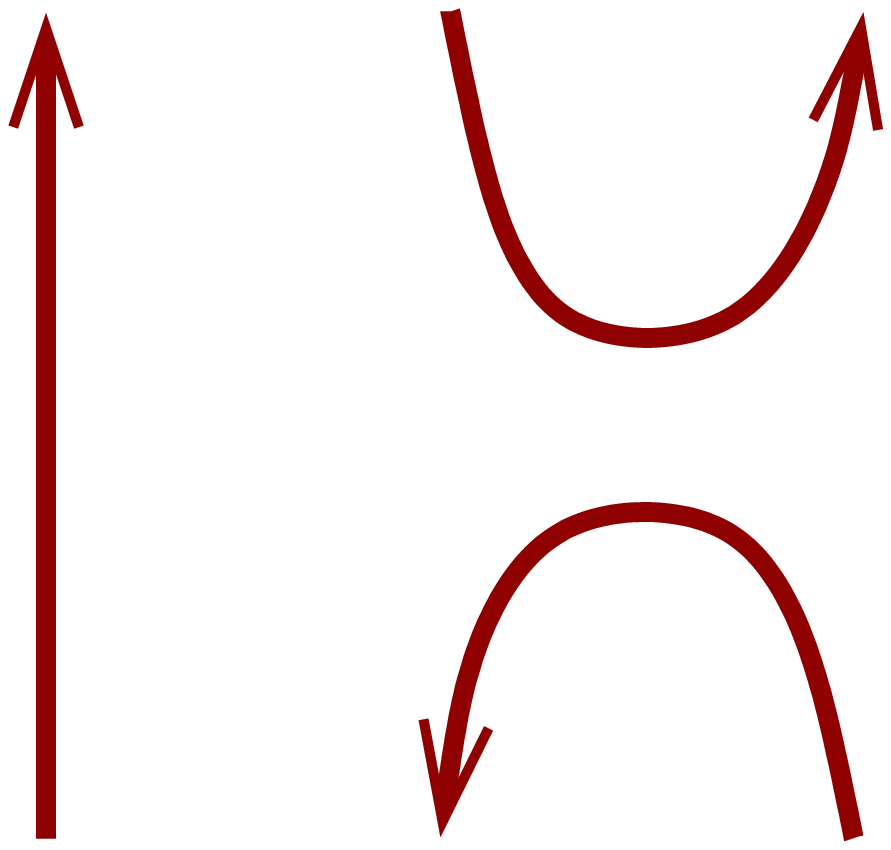}
\quad =\quad \figins{-19}{0.60}{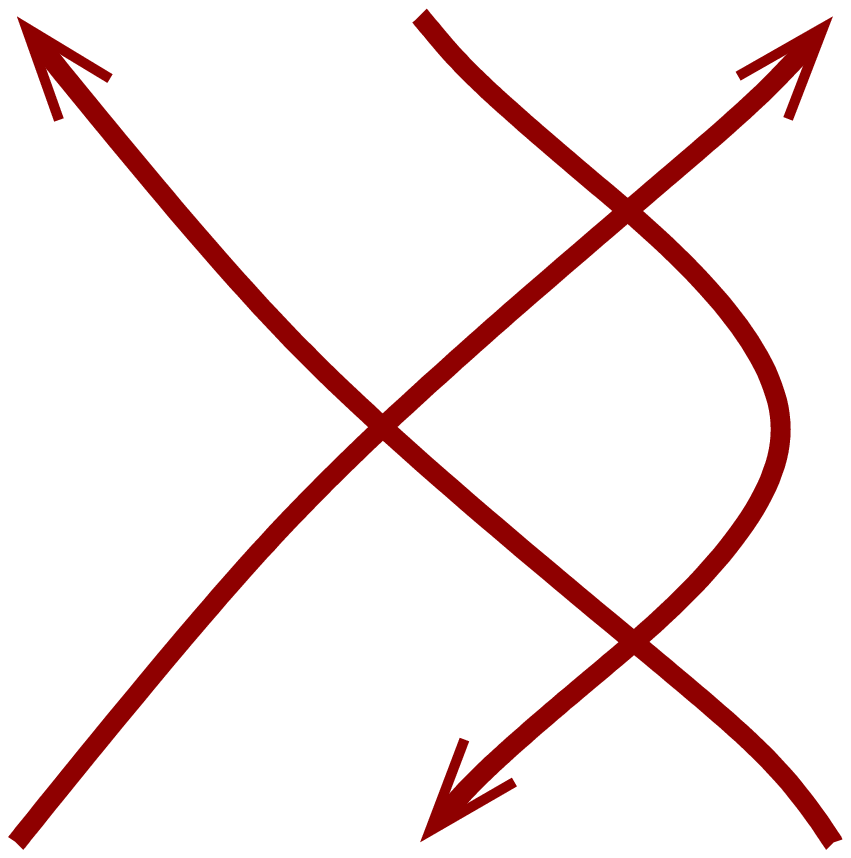}\
+\ [a,-3]\ \figins{-19}{0.60}{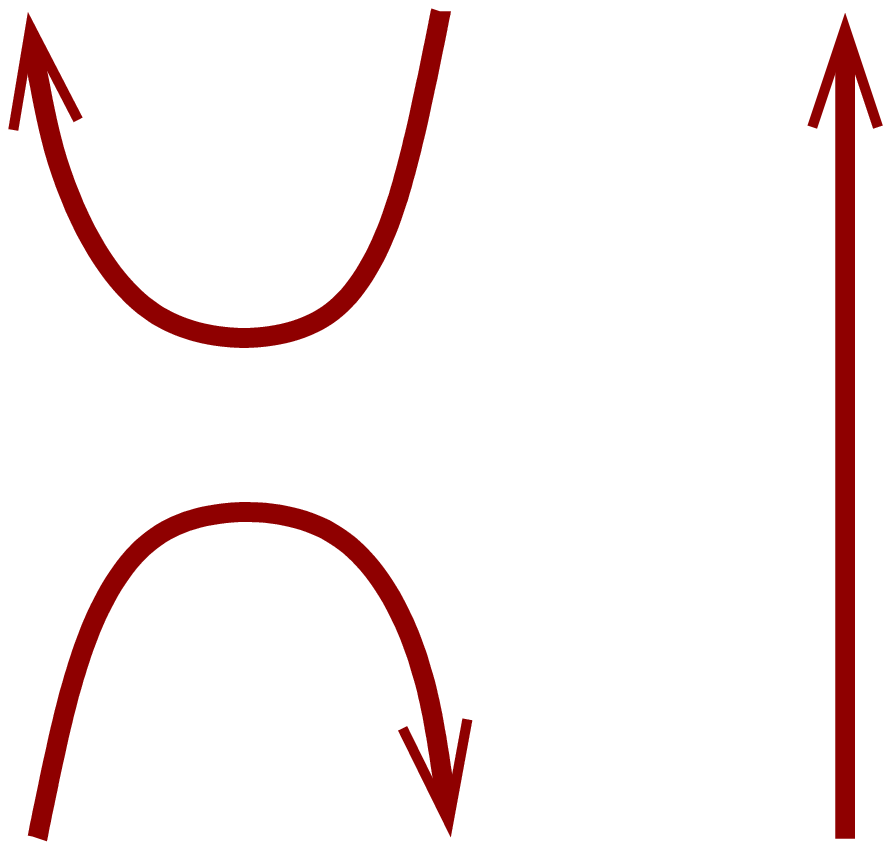}
\label{R3bMOY}
\end{equation}
is a consequence of the previous relations.\\

Similarly to Lemma~\ref{KViso} we have (see~\cite{kauff-vogel})
\begin{lem}
The algebras $\skein^{\tau}_n(a,q)$ and $\skein_n(a,q)$ are isomorphic.
\end{lem}

\medskip

We say an intersection $\ell$ of a diagram $D\in \skein_n(a,q)$ 
with a horizontal line is \emph{generic} if it does not cross any 
singularity of the height function.
Let $\und{\ell}$ be a sequence of $+$'s and $-$'s
in 1-1 correspondence with the orientation of the arcs of $D$ in the neighborhood of the points in $\ell$,
where $+$ (resp. $-$) corresponds to an upward (resp. downward) orientation. 
For example, for the generic intersection
\begin{equation*}
\figins{-25}{0.60}{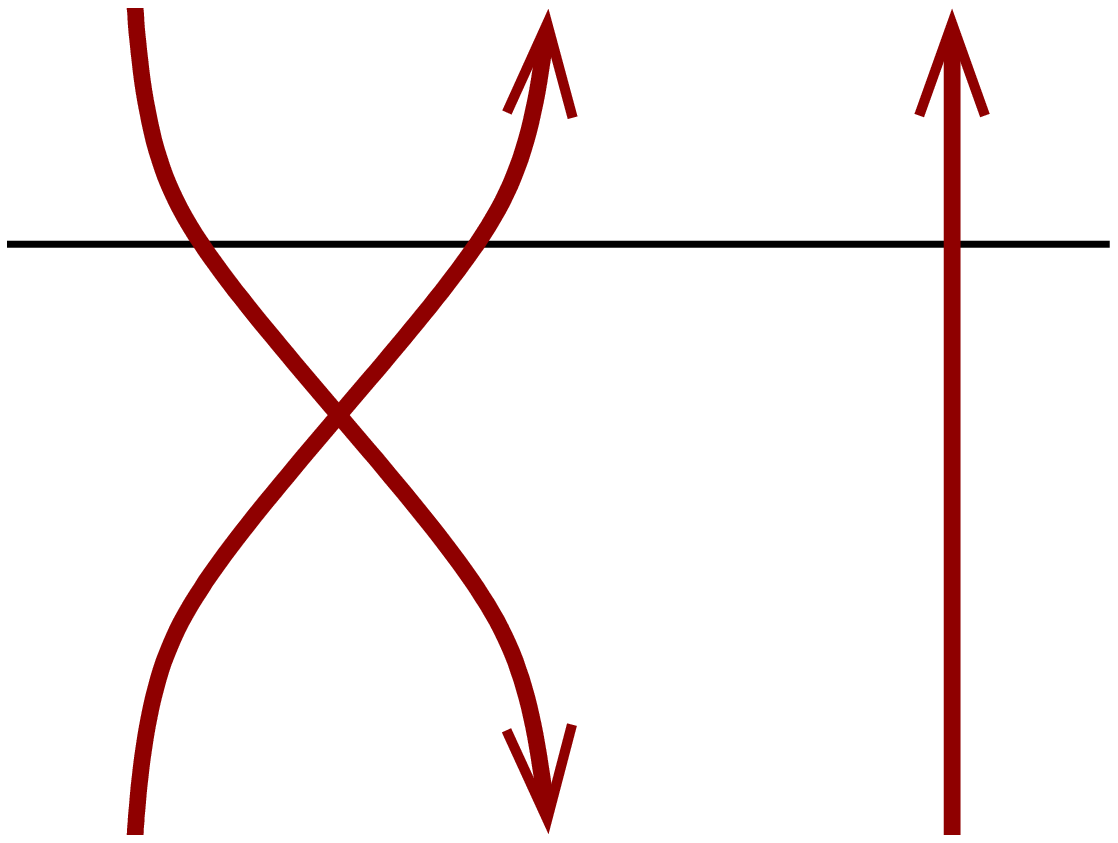}
\end{equation*}
we have $\und{\ell}=(-++)$.
Let $n_+(\ell)$ (resp. $n_-(\ell)$) be the number of $+$'s (resp. $-$'s) in $\und{\ell}$.
The quantity $n_+(\ell)-n_-(\ell)$ is the same for all 
generic intersections and therefore gives an invariant of the diagram $D$ itself which we denote
$n_\pm(D)$.
Given an element of $\skein_q(n,N)_{n_\pm}$ one can express it as a 
linear combination of elements such that $n_++n_-$ is at most equal to $n$. 
In the sequel we will also assume that the elements of $\skein_q(n,N)_{n_\pm}$ are 
of this form~\footnote{We thank A-L. Thiel for pointing this out.}. 
We have a direct sum decomposition of algebras
\begin{equation}\label{eq:skeingrad}
\skein_n(a,q) = \bigoplus\limits_{n_\pm =-n}^n \bigl(\skein_n(a,q)\bigr)_{n_\pm}
\end{equation}
where $(\skein_n(a,q))_{n_\pm}$ consists of all diagrams 
$D$ with $n_\pm(D)=n_\pm$.

\smallskip

For each  $\und{\ell}$ with $\vert\und{\ell}\vert=n$ there is an idempotent $e_{\und{\ell}}$ given
by $n$ parallel vertical strands with orientations matching $\und{\ell}$.
We have 
\begin{equation*}
1_{\skein_n(a,q)} = \sum\limits_{ \und{\ell}\in\{+,-\}^n} e_{\und{\ell}} .
\end{equation*}
We write $e_{(+)^n}$ in the special case of $n$ plus signs, $\und{\ell}=(+,\dotsc,+)$,
\begin{equation}
\label{eq:special-idemp-skein}
e_{(+)^n}\ =\ 
\labellist
\small
\hair 2pt
\pinlabel $\dotsc$ at 265 130
\endlabellist
\figins{-19}{0.60}{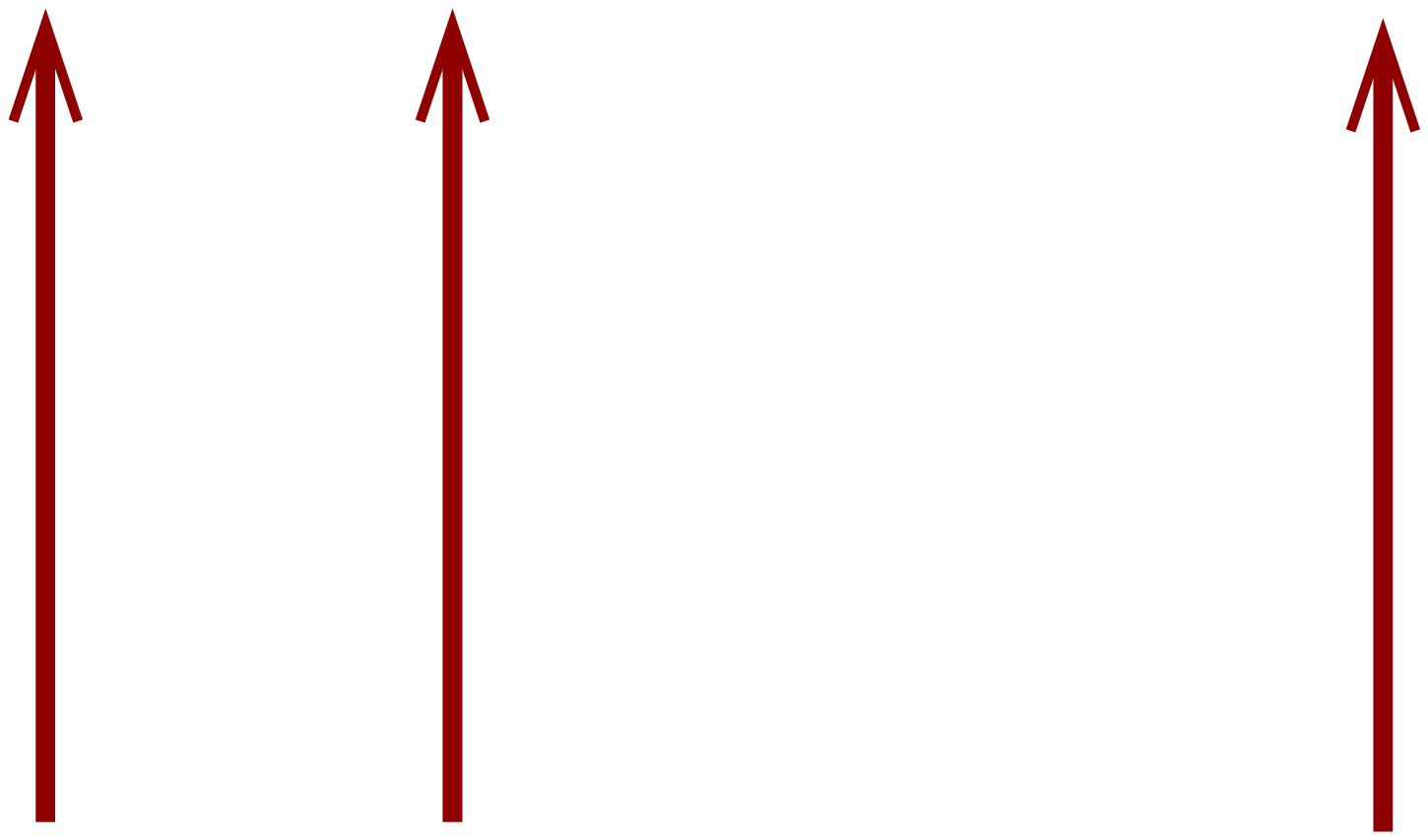}
\end{equation}

\subsection{The Iwahori-Hecke algebra $H_n(q)$}    %
\label{ssec:hecke}                                 %

The algebras $\bmw_n(a,q)$ and $\skein_n(a,q)$ introduced above share 
the Iwahori-Hecke algebra $H_n(q)$  as a quotient.
Recall that $H_n(q)$ is a $q$-deformation of the group algebra of 
the symmetric group on $n$ letters.
\begin{defn}
The Iwahori-Hecke algebra $H_n(q)$ is the unital associative 
$\bQ(q)$-algebra generated by the 
elements $T_i$, $i=1,\ldots, n-1$, subject to the relations
\begin{align*}
T_i^2 &= (q^2-1)T_i+q^2
\\
T_iT_j &= T_jT_i\qquad\text{if}\quad|i-j|>1
\\
T_iT_{i+1}T_i &=T_{i+1}T_iT_{i+1}.
\end{align*}
\end{defn}

\n For $q=1$ we recover the presentation of $\bQ[S_n]$ in terms of the simple 
transpositions $s_i$. For any element $s \in S_n$ we can define 
$T_{s}=T_{i_1}\cdots T_{i_k}$, choosing a reduced expression 
$s=s_{i_1}\cdots s_{i_k}$. The relations above guarantee that 
 all reduced expressions of $s$ give the same element $T_{s}$. 
The elements $T_{s}$, for $s \in S_n$, form a linear basis of $H_n(q)$.

There is a change of generators, 
which is convenient for us. Writing $b_i=q^{-1}(T_i+1)$ the relations above become 
\begin{align*}
b_i^2 &= (q+q^{-1})b_i
\\
b_ib_j &= b_jb_i \qquad\text{if}\quad|i-j|>1
\\
b_ib_{i+1}b_i+b_{i+1} &= b_{i+1}b_ib_{i+1}+b_i.
\end{align*}
\noindent These generators are the simplest elements of the  
{\em Kazhdan-Lusztig basis}~\cite{kazh-lusz}. Although the change of generators is simple, the 
whole change of linear bases is very complicated.

\medskip

The Iwahori-Hecke algebra $H_n(q)$ can be obtained as a quotient of both 
$\bmw_n(a,q)$ 
and $\skein_n(a,q)$ 
(see for instance~\cite{birmanw} and~\cite{kauff-vogel}). 
Let
 $\Bigl[\figins{-6}{0.25}{cupcap}\Bigr]$ and 
$\Bigl[\figins{-6}{0.25}{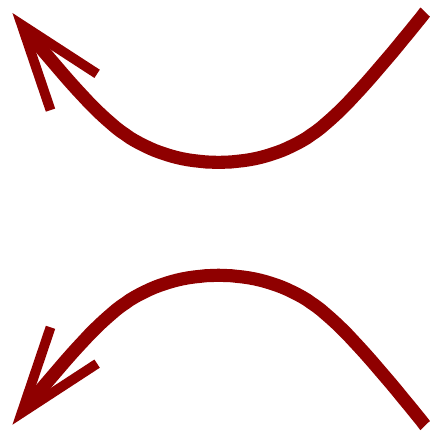},\,\figins{-6}{0.25}{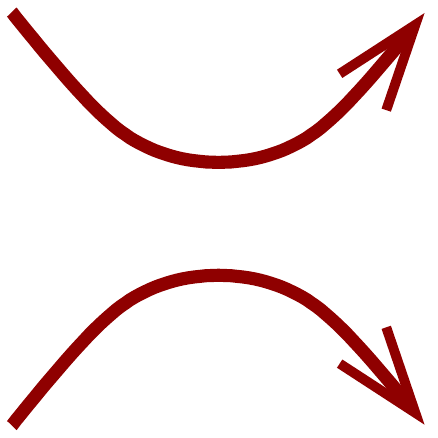},\,
\figins{-6}{0.25}{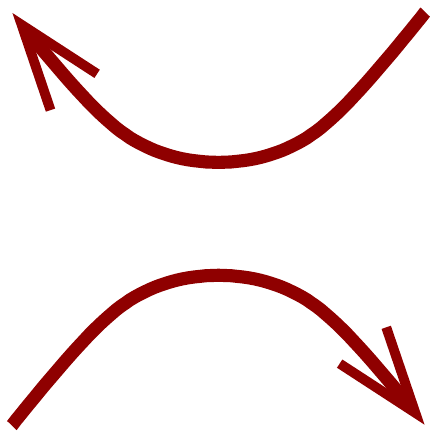},\,\figins{-6}{0.25}{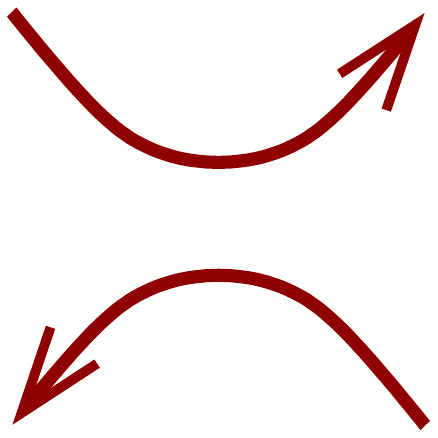}\Bigr]$
denote the two-sided ideals of $\bmw_n(a,q)$ and $\skein_n(a,q)$ respectively, generated by
the elements inside the brackets.

\begin{lem}
We have isomorphisms
\begin{equation*}
H_n(q)\ 
\cong\ \bmw_n(a,q) /_{ \Bigl[\figins{-6}{0.25}{cupcap}\Bigr]}\
\cong\ e_{(+)^n}\skein_n(a,q)e_{(+)^n} /
_{ \Bigl[\figins{-6}{0.25}{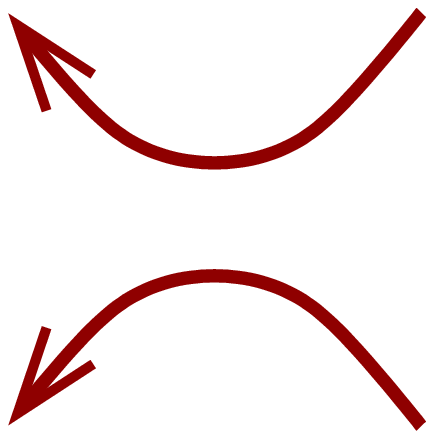},\,\figins{-6}{0.25}{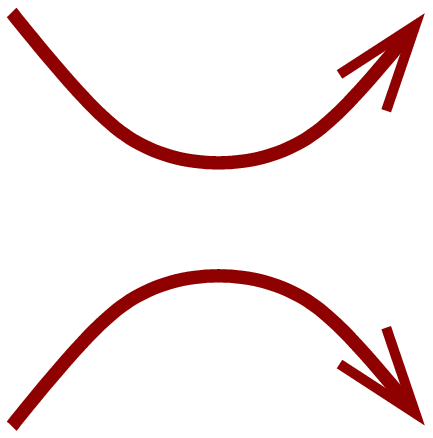},\,
\figins{-6}{0.25}{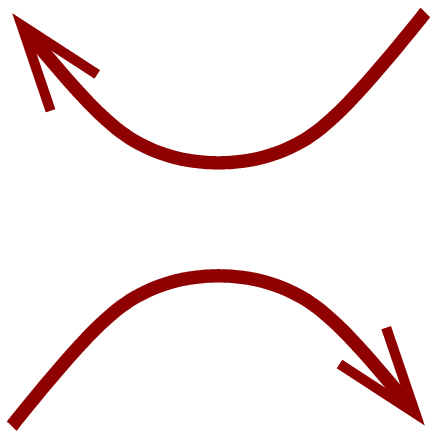},\,\figins{-6}{0.25}{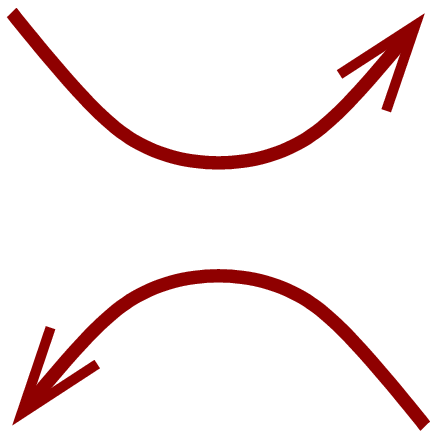}\Bigr]}
\end{equation*}
\end{lem}

\smallskip

In these quotients we identify $\figins{-6}{0.25}{4vert}$ and $\figins{-6}{0.25}{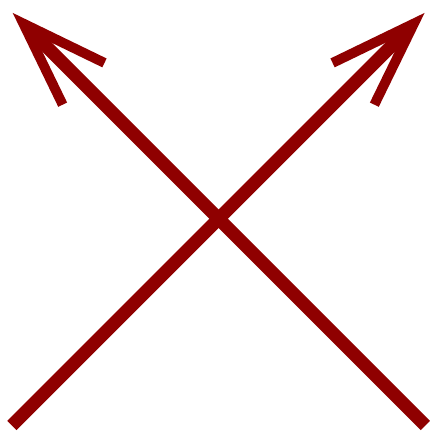}$
with the generator $b_i$ of $H_n(q)$.

\medskip

%
\section{Jaeger's model for the BMW algebras}\label{sec:jaeger}

\subsection{Jaeger's Theorem for the Kauffman polynomial}  %

Given any unoriented link $L$, Jaeger found a beautiful formula computing the 2-variable Kauffman polynomial 
$F$ of $L$ as a weighted sum of HOMFLY-PT polynomials $P$ of certain oriented links associated with $L$. 
We now explain how to obtain a family of oriented links diagrams 
to which we compute the HOMFLY-PT polynomial and recall Jaeger's formula.

\medskip

The HOMFLY-PT link polynomial $P=P(a,q)$ is the unique invariant of oriented  links satisfying the skein relation
\begin{equation*}
a P\biggl(\figins{-10}{0.35}{Xing-pu}\biggr) \ \ - \ \ 
a^{-1}P\biggl(\figins{-10}{0.35}{Xing-nu}\biggr)
=\ \ (q - q^{-1})P \biggl(
\figins{-10}{0.35}{upup}
\biggr)
\end{equation*}
and taking the value $[a]$ on the unknot,
\begin{equation*}
P\biggl(\figins{-8}{0.30}{loop-r}\biggr) 
=\  [a] 
\end{equation*}
(recall $[a]= \tfrac{a-a^{-1}}{q-q^{-1}}$ was defined in Subsection~\ref{ssec:bmw}).

The 2-variable Kauffman polynomial $F=F(a,q)$ is the unique invariant of framed unoriented 
links satisfying the relations 
\begin{gather*}
F\biggl(\figins{-10}{0.35}{ucurl-l}\biggr)\ \ 
= \ a^{-2}q\ \ 
F\biggl(\ \figins{-10}{0.35}{one}\ \biggr)
\displaybreak[0]\\[2ex]  
F\biggl(\figins{-10}{0.35}{Xing-p}\biggr) \ \ - \ \ 
F \biggl(\figins{-10}{0.35}{Xing-n}\biggr)\
=\ \ (q - q^{-1})\Biggl(\ 
F\biggl(\figins{-10}{0.35}{cupcap}\biggr)\  -  \
F\biggl(\figins{-10}{0.35}{cupcap-vert}\biggr)\
\Biggr)
\end{gather*}
and taking the value $\delta=[a^2,-1]+1$ on the unknot,
\begin{equation*}
F\biggl(\figins{-8}{0.30}{loop}\biggr) 
=\  \delta.
\end{equation*}

\medskip

We now recall the definition of the \emph{rotational number} of an oriented link diagram. 
Given an oriented link diagram $D$, smooth all its crossings as follows
\begin{equation*}
\figins{-10}{0.35}{Xing-pu} \ \ \longrightarrow \ \ 
\figins{-10}{0.35}{upup}    \ \ \longleftarrow  \ \ 
\figins{-10}{0.35}{Xing-nu}.
\end{equation*}

The result is a collection of oriented circles embedded in the plane. 
Define the rotational number $\rot(D)$ of $D$ to be the sum over all resulting circles 
of the contribution of each circle, where a circle contributes $-1$ if it is oriented clockwise, 
and $+1$ otherwise,
\begin{equation*}
\rot\biggl(
\figins{-8}{0.30}{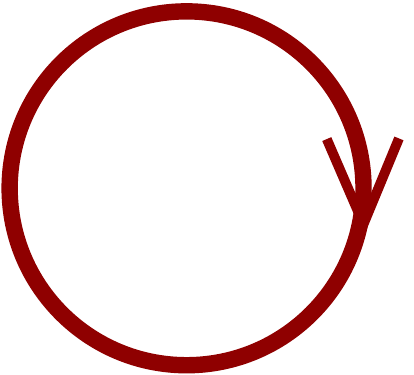}
\biggr)  =\  -1 
\mspace{80mu}
\rot\biggl(
\figins{-8}{0.30}{loop-r}
\biggr)  =\  +1 .
\end{equation*}

\medskip

Given an unoriented diagram $D$ of $L$, we resolve each of its crossing in six different ways,
\begin{equation*}
\figins{-10}{0.35}{Xing-p} \ \  \mapsto\ \ 
\ \ \figins{-10}{0.35}{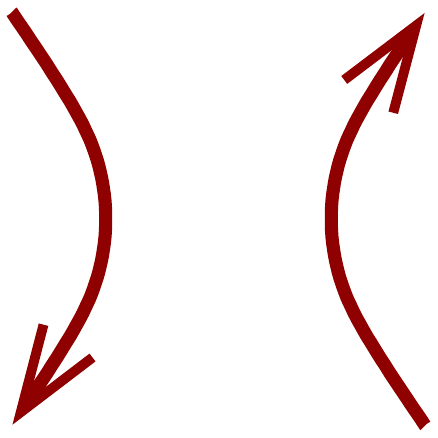}\ \  , \ \ 
\figins{-10}{0.35}{cupcap-lr}\ \  , \ \
\figins{-10}{0.35}{Xing-pu}\ \  , \ \
\figins{-10}{0.35}{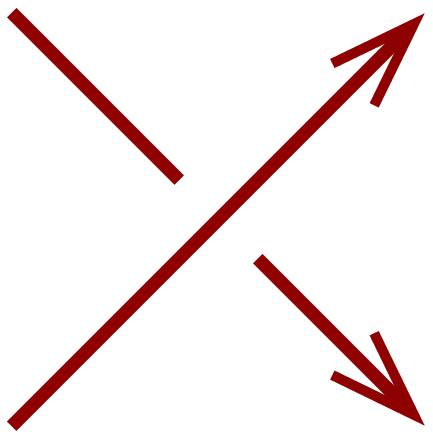}\ \  , \ \
\figins{-10}{0.35}{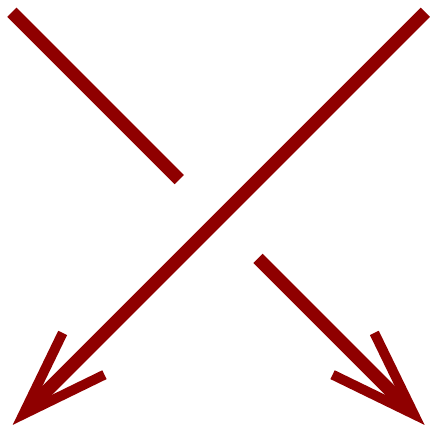}\ \  , \ \
\figins{-10}{0.35}{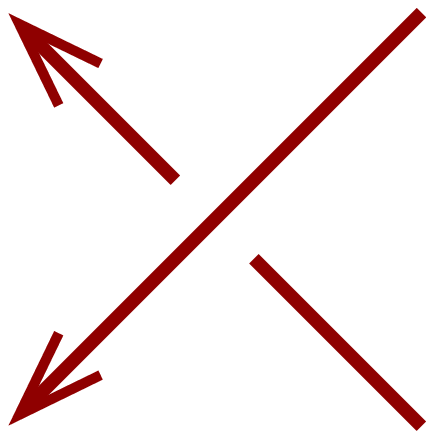}\ \ . \ \
\end{equation*}
Choosing a resolution for each crossing, is called a complete resolution.
A complete resolution resulting in a coherently oriented link diagram is called an oriented complete resolution of $D$.
Denote by $\mbox{res}(D)$ the set of all oriented complete resolutions. Notice that if there is no crossing there are 
two resolutions which consist in the two possible orientations of an unoriented circle.
We next define a weight $w$ associated to each oriented complete resolution. 
It is computed as a product of local weights associated to each crossing of $D$ and its oriented resolution. 
The local weights are
 \begin{equation*}
w\left( \figins{-10}{0.35}{Xing-p},\figins{-10}{0.35}{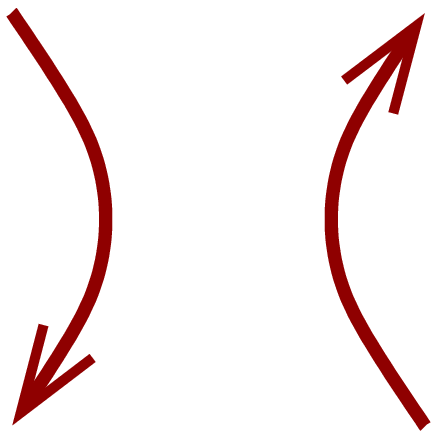}\\ \right)=q-q^{-1}\ \ , \ \
w\left( \figins{-10}{0.35}{Xing-p},\figins{-10}{0.35}{cupcap-lr}\\ \right)=q^{-1}-q\ \ , \ \
w\left( \figins{-10}{0.35}{Xing-p},\figins{-10}{0.35}{Xing-pu}\\ \right)=1\ \ , \ \
\end{equation*}
\begin{equation*}
w\left( \figins{-10}{0.35}{Xing-p},\figins{-10}{0.35}{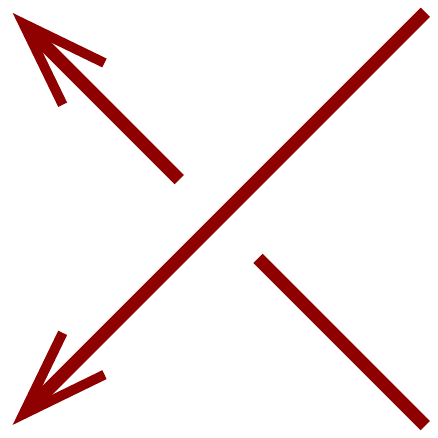}\\ \right)=1\ \ , \ \ 
w\left( \figins{-10}{0.35}{Xing-p},\figins{-10}{0.35}{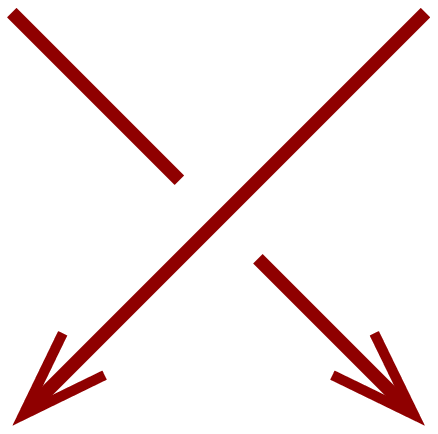}\\ \right)=1\ \ , \ \
w\left( \figins{-10}{0.35}{Xing-p},\figins{-10}{0.35}{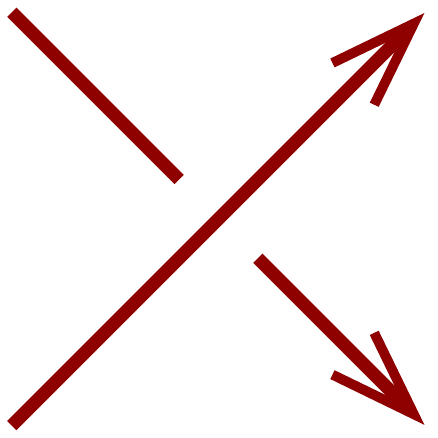}\\\right)=1\ \ , \ \
\end{equation*}

\medskip

Jaeger's formula~\cite{kauff-phys} is given in the following
\begin{thm}\label{thm:jaeg-link}
Let $D$ be an unoriented diagram of a link $L$.
The formula
\begin{equation}
\label{eq:jaeg-poly}
F(D)(a,q)=\sum_{\overrightarrow{D}\in{\res}(D)}{(a^{-1}q)}^{{\rot}(\overrightarrow{D})}w(\overrightarrow{D})P(\overrightarrow{D})
\end{equation}
is a HOMFLY-PT expansion of the Kauffman polynomial of $L$.
\end{thm}

The proof of this formula follows by direct computation checking that the right hand side of the equality is invariant 
under the second and the third Reidemeister moves, satisfies the 2-variable Kauffman skein relation and 
the change of framing relations, as well as the value on the unknot. 
We do not detail the proof here (see~\cite{kauff-phys}) 
because it will follow from our algebraic setting in the next section.

\subsection{Jaeger's BMW}             %
\label{ssec:jaegBMW}                  %

We now reformulate Jaeger's formula in 
terms of an algebra homomorphism between the BMW algebra and the HOMFLY-PT skein algebra. 
To this end we use the graphical calculus of Kauffman and Vogel described in Section~\ref{sec:algebras}. 
We give explicitly the algebra homomorphisms in terms of 4-valent graphs and derive 
a proof of Jaeger's theorem in this context (see Proposition~\ref{prop:jaegtang}). 
Such a reformulation of the Jaeger expansion for graph polynomials was also explored by Wu~\cite{Wu}.

\medskip

Given an $(n,n)$ 4-valent graph $\Gamma$, we can resolve each of its vertex in eight different ways,
\begin{equation*}
\figins{-10}{0.35}{4vert} \ \  \mapsto\ \ 
\ \ \figins{-10}{0.35}{cupcap-rl}\ \  , \ \ 
\figins{-10}{0.35}{downup}\ \  , \ \
\figins{-10}{0.35}{cupcap-lr}\ \  , \ \
\figins{-10}{0.35}{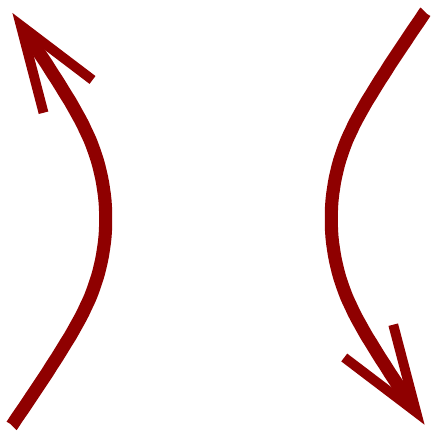}\ \  , \ \
\figins{-10}{0.35}{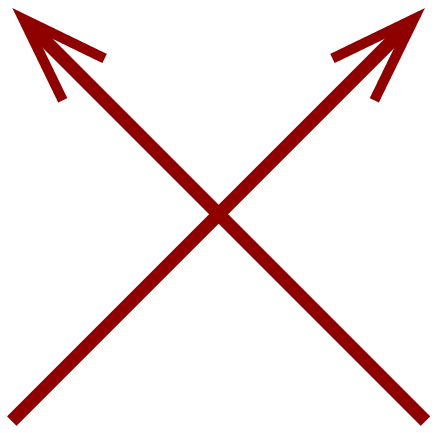}\ \  , \ \
\figins{-10}{0.35}{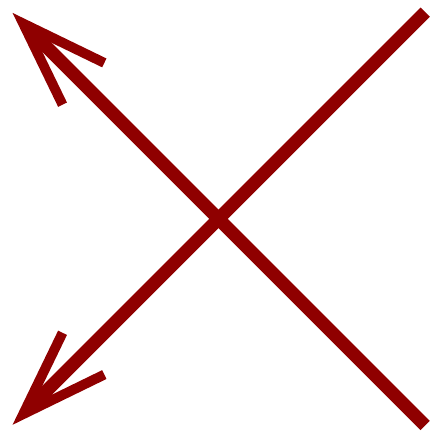}\ \ , \ \
\figins{-10}{0.35}{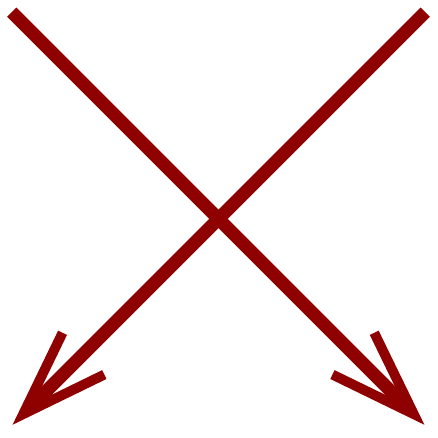}\ \  , \ \
\figins{-10}{0.35}{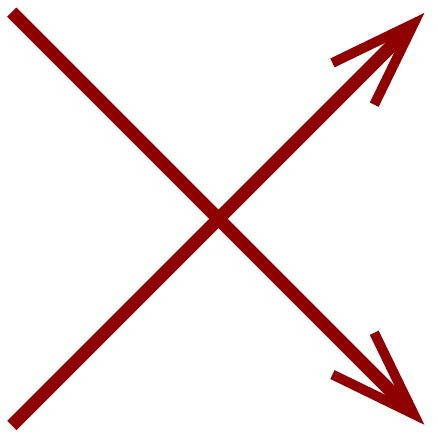}\ \ . \ \
\end{equation*}

The graph obtained by choosing a resolution for each vertex is called a \emph{complete resolution}.
A complete resolution resulting in a coherently oriented $(n,n)$  
graph is called an \emph{oriented complete resolution}.
Denote by $\mbox{res}(\Gamma)$ the set of all oriented complete resolutions of $\Gamma$. 
Notice that if there is no 4-valent vertex 
(i.e. $\Gamma$ consists of an embedding of $n$ arcs) there are $2^n$ 
resolutions consisting in choosing an orientation for each arc.

We now extend the concept of rotation number to 
4-valent oriented graphs of the type under consideration.
Given an oriented $(n,n)$ 4-valent graph $\Gamma$ we can apply the transformation
\begin{equation*}
\figins{-10}{0.35}{crossu} \ \ \longmapsto \ \ 
\figins{-10}{0.35}{upup}
\end{equation*}
to smooth all 4-valent vertices of $\Gamma$ and obtain a disjoint union of oriented circles and $n$ oriented arcs embedded in the plane. 
Define the rotational number $\rot(\Gamma)$ of $\Gamma$ to be the sum over all resulting circles  and arcs of 
the contribution of each circle and each arc, where a circle contributes $-1$ if it is oriented clockwise 
and $+1$ otherwise, and arcs contribute $\pm 1$  or zero according with the rules given below.
\begin{equation}\label{eq:rots}
\begin{split}
\rot\biggl(
\figins{-8}{0.30}{loop-r}
\biggr)  =\  +1 
\qquad
\rot\biggl(
\figins{-4}{0.2}{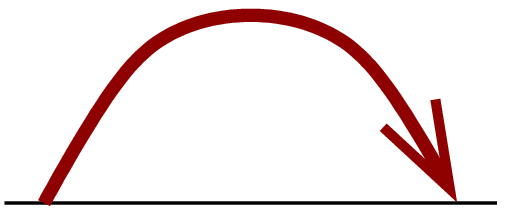}
\biggr)  =\  0 
\qquad
\rot\biggl(
\figins{-4}{0.2}{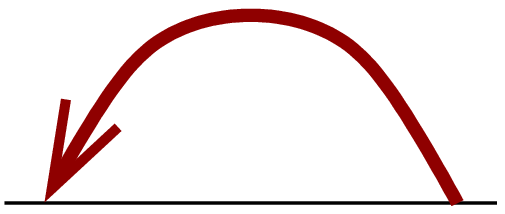}
\biggr)  =\  +1 
\\[2ex]
\rot\biggl(
\figins{-8}{0.30}{loop-l}
\biggr)  =\  -1 
\qquad
\rot\biggl(
\figins{-4}{0.2}{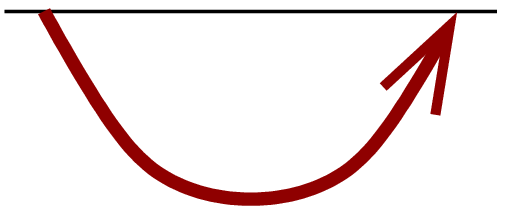}
\biggr)  =\  0 
\qquad
\rot\biggl(
\figins{-4}{0.2}{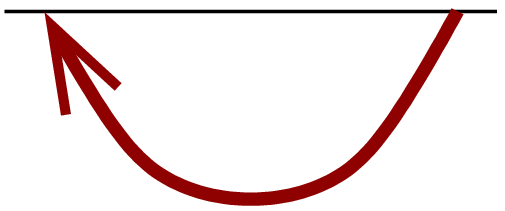}
\biggr)  =\  -1 
\end{split}\end{equation}
In addition the rotational number of a strand going up or down is zero. 
This set of rules allows extending the concept of rotation numbers to tangle diagrams. 

The rotational number is additive with respect to the multiplicative structure of $(n,n)$
 4-valent graph given by concatenation. 
For example,
\begin{align*}
\rot\biggl( \figins{-10}{0.35}{cupcap-ll} \circ  \figins{-10}{0.35}{cupcap-rr}\biggr)
&= \rot\left(\ \figins{-19}{0.60}{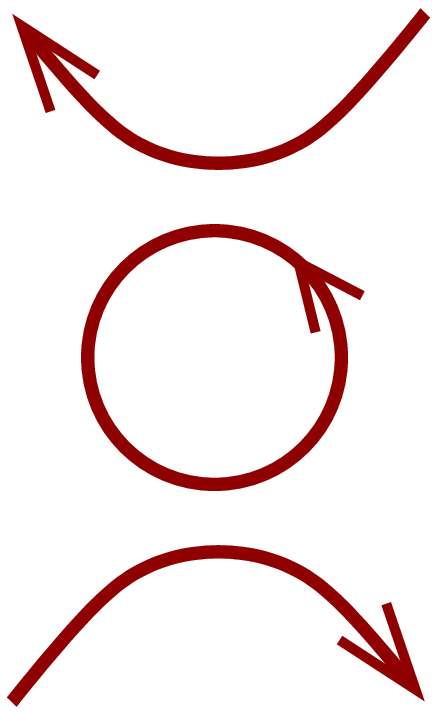}\ \right) 
=
\rot\biggl( \figins{-10}{0.35}{cupcap-ll}\biggr) +  \rot\biggl(\figins{-10}{0.35}{cupcap-rr}\biggr)
=\ 0 .
\end{align*}

The last concept needed in this section is the \emph{weight} $w$ 
associated to each oriented complete reso\-lution. 
This is an extension to oriented resolutions of the local weights of link diagrams
from the last subsection, which justifies the use of the same notation as before.
It is computed as a product of local weights 
associated to each 4-valent vertex of $\Gamma$ and an oriented resolution of it. 
The local weights are described below.
\begin{gather}
w\biggl( \figins{-10}{0.35}{4vert},\figins{-10}{0.35}{cupcap-rl} \biggr)=
w\biggl( \figins{-10}{0.35}{4vert},\figins{-10}{0.35}{downup} \biggr)=q^{-1}\ \ , \ \
\nonumber
\displaybreak[0]\\[2ex]\label{eq:weights-bmw}
w\biggl( \figins{-10}{0.35}{4vert},\figins{-10}{0.35}{cupcap-lr} \biggr)=
w\biggl( \figins{-10}{0.35}{4vert},\figins{-10}{0.35}{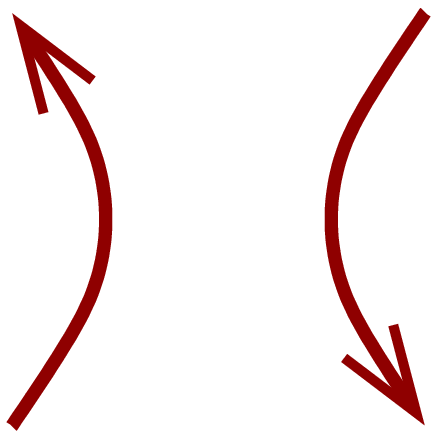} \biggr)=q\ \ , \ \ 
\displaybreak[0]\\[2ex]
w\biggl( \figins{-10}{0.35}{4vert},\figins{-10}{0.35}{crossu} \biggr)=
w\biggl( \figins{-10}{0.35}{4vert},\figins{-10}{0.35}{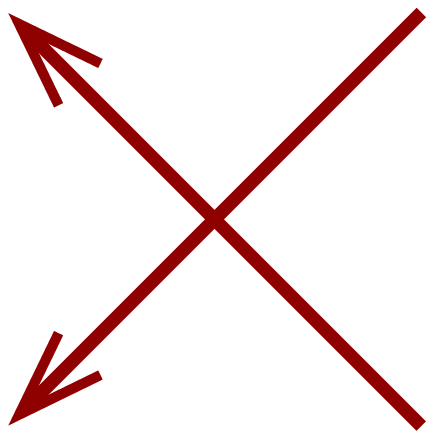} \biggr)=
w\biggl( \figins{-10}{0.35}{4vert},\figins{-10}{0.35}{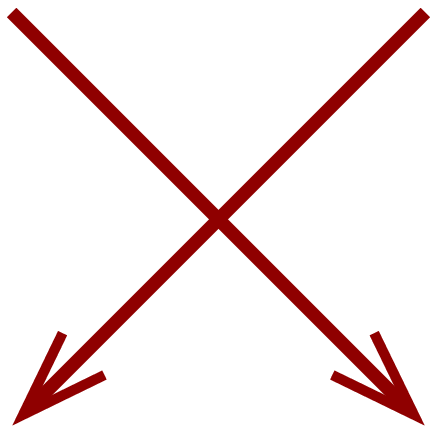} \biggr)= 
w\biggl( \figins{-10}{0.35}{4vert},\figins{-10}{0.35}{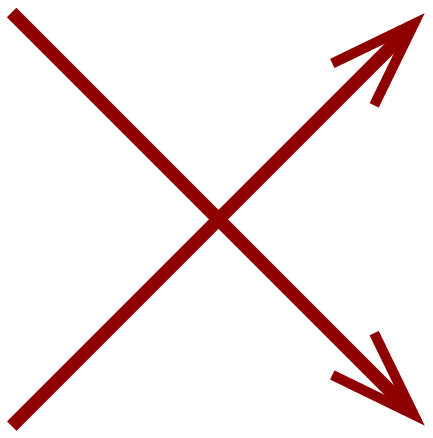} \biggr)=1\ \ . \ \
\nonumber
\end{gather}

\medskip

For any $(n,n)$  4-valent graph $\Gamma$, we define the \emph{Jaeger homomorphism} as
\begin{equation}\label{eq:jaeger-alg}
\psi(\Gamma)=\sum_{\overrightarrow{\Gamma}\in \res(\Gamma)}{(a^{-1}q)}^{\rot(\overrightarrow{\Gamma})}
w(\overrightarrow{\Gamma})\overrightarrow{\Gamma}.
\end{equation}

The content of the next theorem is to prove that the previous expression defines a well-defined 
injective morphism of algebras from $\bmw_n(a,q)$  to $\skein_n(a,q)$.

\begin{thm}
\label{thm:jaegmap}
The map $\psi$ from $\bmw_n(a,q)$  to $\skein_n(a,q)$ is a well-defined injective morphism of algebra.
\end{thm}
\begin{proof}
We first prove that $\psi$ is well-defined. 
Notice that the expression $\psi(\Gamma)$ is invariant on the isotopy class of $\Gamma$. 
This follows from the observations that the rotational factor and the local weights are invariant by planar isotopies.
Secondly we have to check that the relations in  Definition~\ref{kauffgraphrelation} are
in the kernel of $\psi$. 
In order to simplify the computations we consider some symmetries of the relations in Definition~\ref{kauffgraphrelation} 
as well as some symmetries of the local weights. 
All relations are invariant by 
reflections as well as by simultaneous changes of variables from $a$ to $a^{-1}$ and $q$ to $q^{-1}$. 
In addition, the 
rotational factor and the local weights are invariant by simultaneously applying a reflection of the 
plane and the previous changes of variables. 
This implies that for any $(n,n)$ 4-valent graph $\Gamma$, $\psi(\Gamma)$ is 
invariant by simultaneously applying a reflection of the plane and the previous changes of variables.
Hence in order to check that the relations in Definition~\ref{kauffgraphrelation} are sent to zero by $\psi$ we can restrict our 
verifications to some cases depending on fixing the orientations of the boundary of the graphs involved. 
In addition notice also that by the locality of the relations in $\skein_n(a,q)$, two elements of $\skein_n(a,q)$ 
which do not have 
the same orientations on the endpoints are linearly independent.

We check the first relation which is
\begin{equation*}
\figins{-13}{0.45}{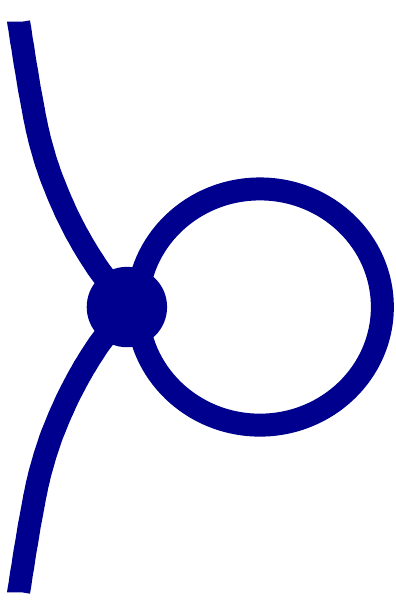}\
= \ 
[a](q^2a^{-1}+q^{-2}a)\ \ 
\figins{-13}{0.45}{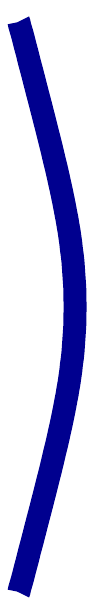}\ .
\end{equation*}
The image of the left hand side by $\psi$  contains $6$ terms, in three of them the orientation is upwards.
By the previous considerations, we can restrict attention to the case for instance where the orientation is upwards. 
The three graphs involved are
\begin{equation}\label{eq:threegraphs}
q^{-1}\figins{-13}{0.45}{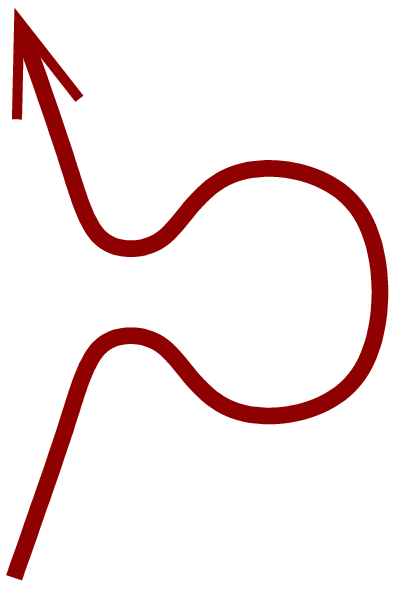}
\ \ +\ \ 
q^2a^{-1}\figins{-13}{0.45}{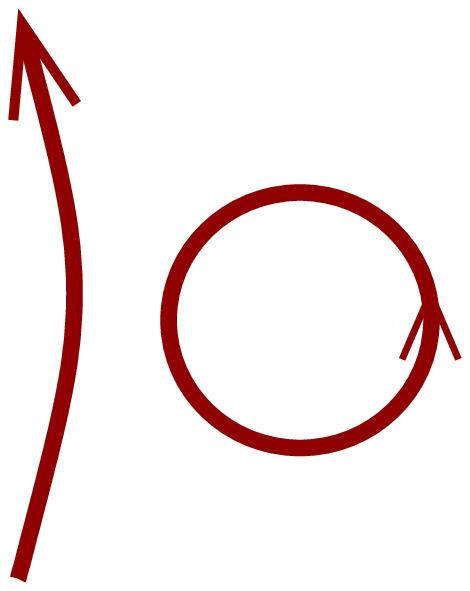}
\ \ +\ \ 
aq^{-1}\figins{-13}{0.45}{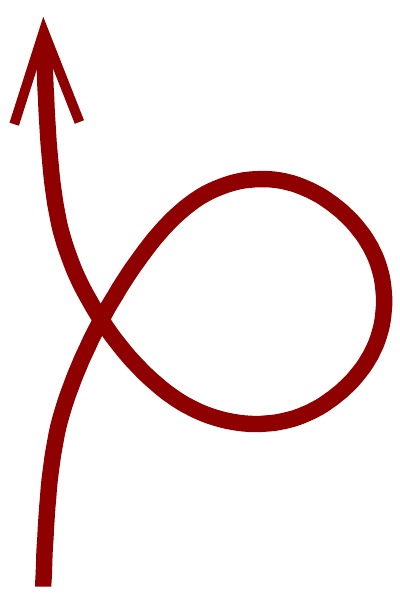}\ .
\end{equation}

It is a matter of a short computation (using relations (\ref{R1MOY}) and (\ref{SPLITMOY})) 
to check that~\eqref{eq:threegraphs} is equal to
\begin{equation*}
[a](q^2a^{-1}+q^{-2}a)\ \ 
\figins{-13}{0.45}{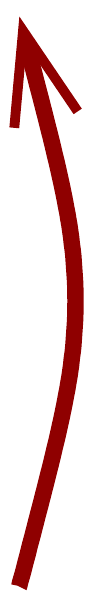}
\end{equation*}
which is exactly the part of the image by $\psi$ of the right hand side with orientation upward. 
This concludes the proof of the fact that the first relation is annihilated by $\psi$.\\

For the second relation, there are three different cases to consider up to symmetry, which are the following ones:
\begin{equation*}
\figins{-21}{0.65}{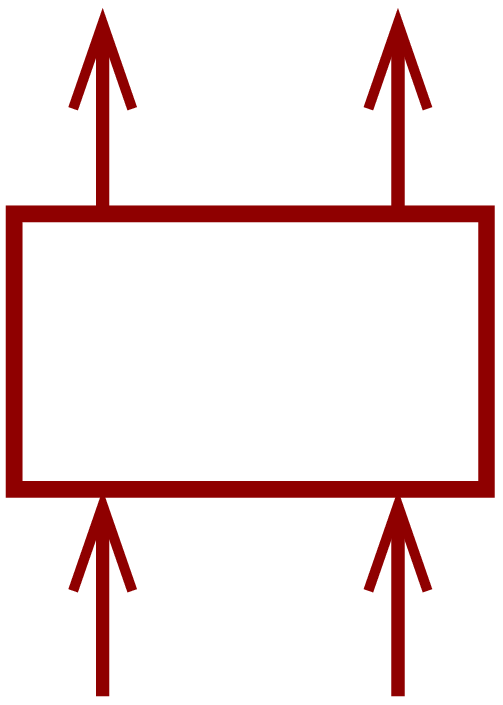}\ ,
\qquad
\figins{-21}{0.65}{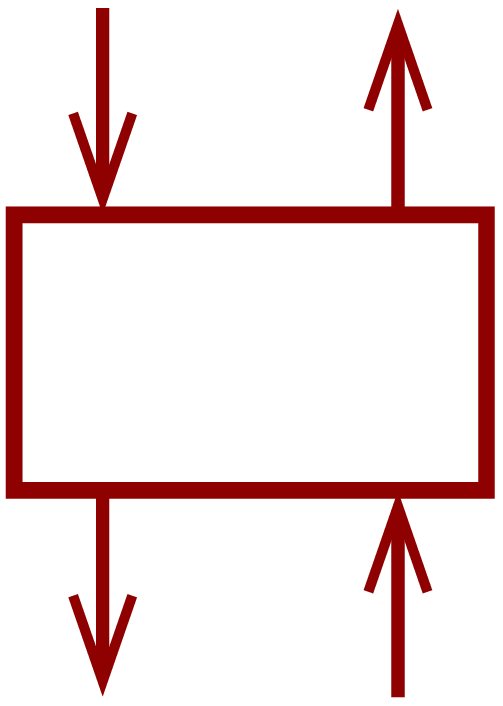}\quad\text{and}\quad
\figins{-21}{0.65}{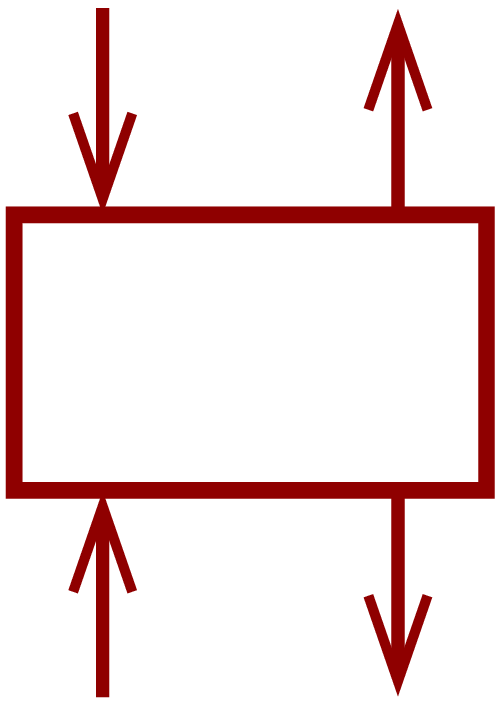}\ .
\end{equation*}

For each case one computes the contribution of each of the three terms involved in the relation.
For the first case, there are only two terms that contribute, and one obtains immediately the relation (\ref{R2aMOY}) of the 
HOMFLY-PT skein algebra.
For the second case the contributions of each term are
\begin{align*}
\figins{-19}{0.60}{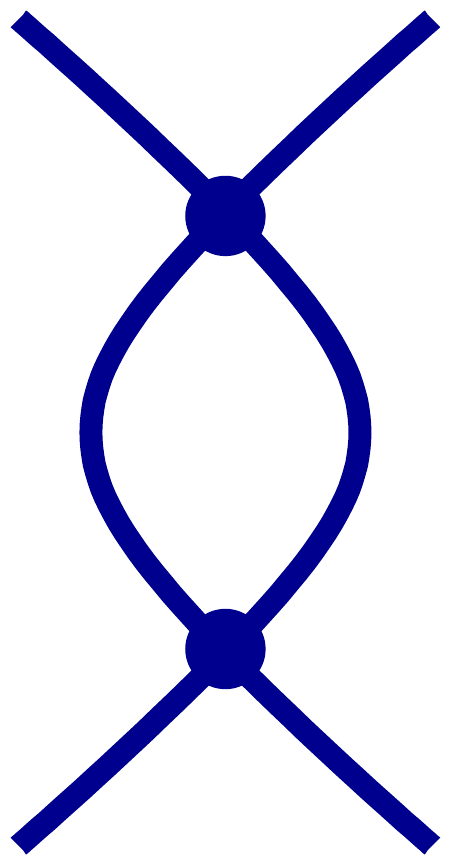} &\colon\quad
q^{-2}\figins{-19}{0.60}{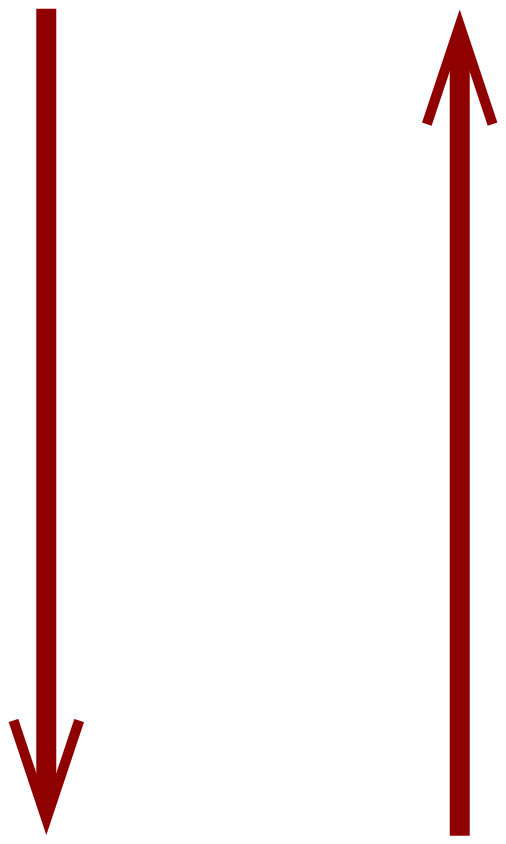}
+a^{-1}q\figins{-19}{0.60}{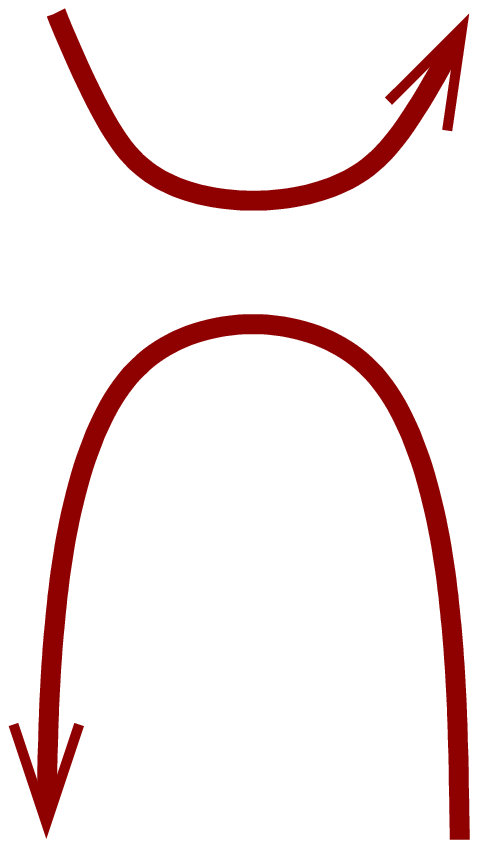}
+a^{-1}q\figins{-19}{0.60}{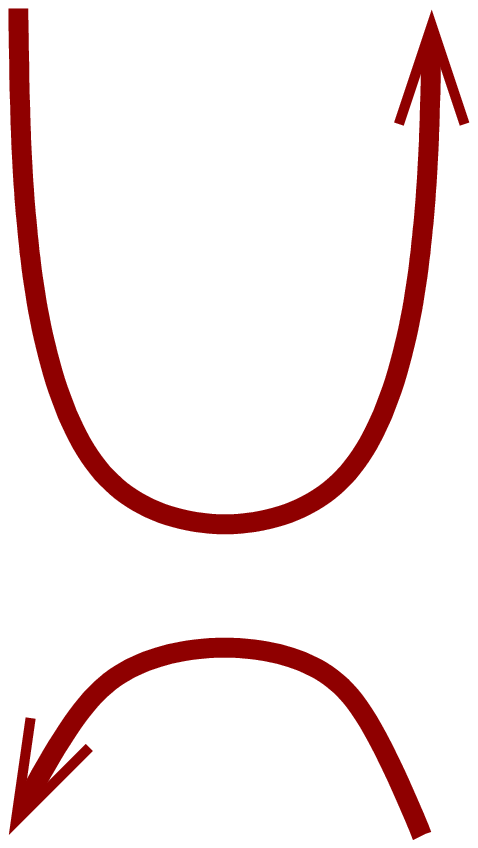}
+\ \figins{-19}{0.60}{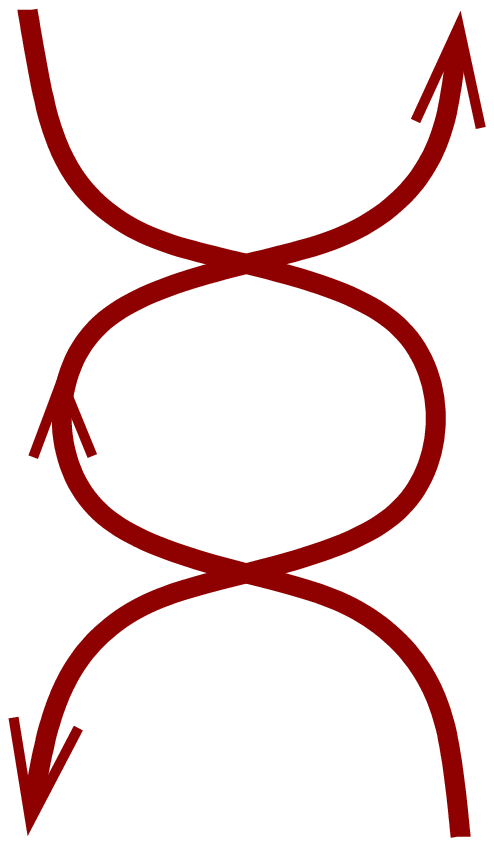}\ 
+a^{-2}q^4\figins{-19}{0.60}{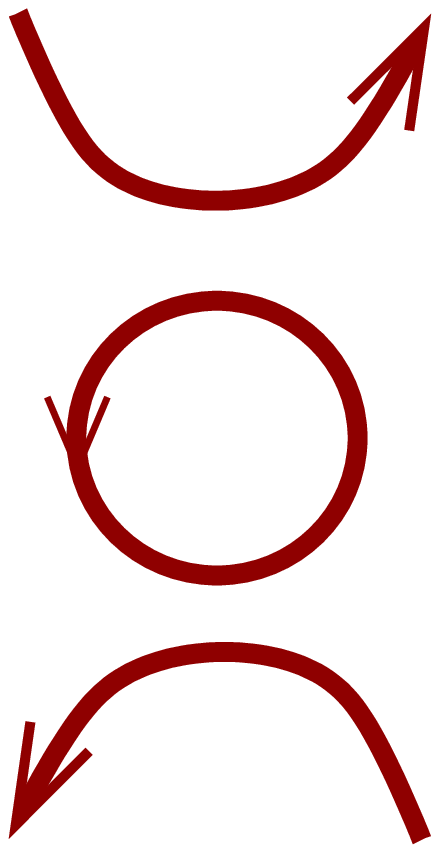}
\displaybreak[0]\\[2ex]
(q+q^{-1})\figins{-10}{0.35}{4vert} &\colon\quad
(q+q^{-1})a^{-1}q^2\ \figins{-10}{0.35}{cupcap-lr} 
+ q^{-1}(q+q^{-1})\ \figins{-10}{0.35}{downup} 
\displaybreak[0]\\[2ex]
\bigl([a^2,-3]+1\bigr)\figins{-10}{0.35}{cupcap} &\colon\quad 
\bigl([a^2,-3]+1\bigr)a^{-1}q\ \figins{-10}{0.35}{cupcap-lr} 
\end{align*}

It is easy to check that
\begin{equation*}
\figins{-19}{0.60}{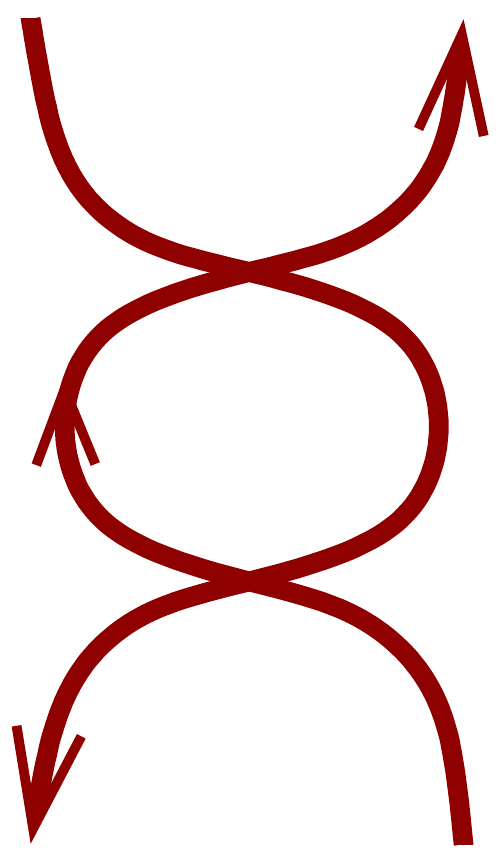}\ 
+a^{-2}q^4\figins{-19}{0.60}{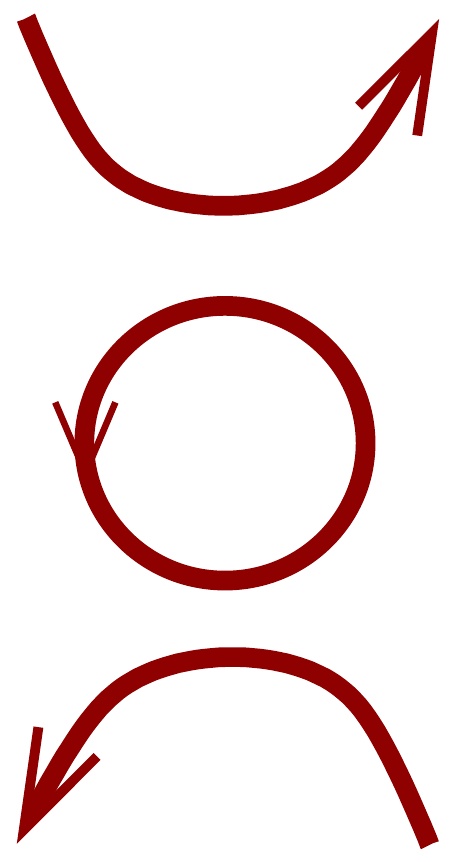}\ =\
\bigl(a^{-1}q^3 + [a^2,-3]a^{-1}q\bigr) 
   \figins{-19}{0.60}{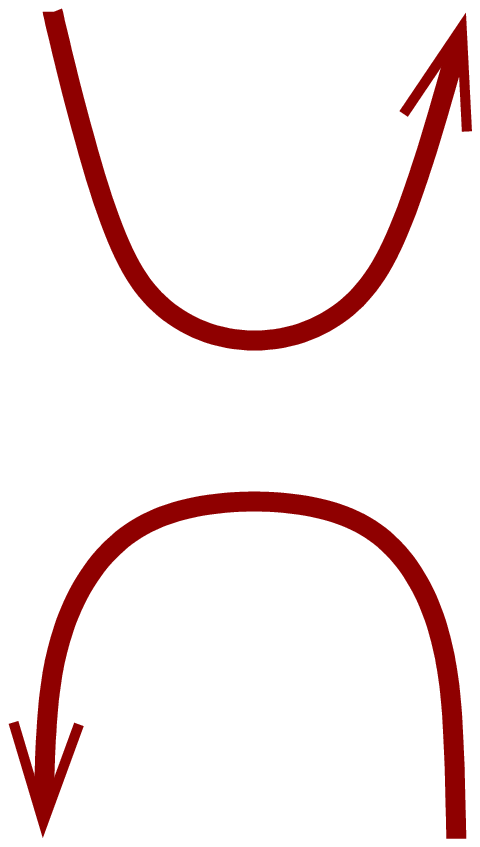}\
+\ \figins{-19}{0.60}{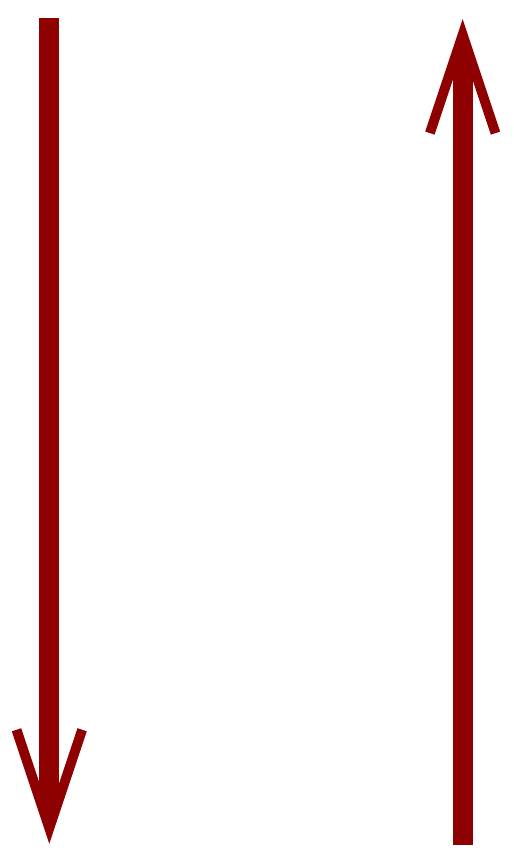} 
\end{equation*}
using the relations (\ref{R2bMOY}) and (\ref{SPLITMOY}) in the HOMFLY-PT skein algebra.
The third case goes the same way, using only the relation~(\ref{R1MOY}) of the HOMFLY-PT skein algebra.\\

For the last relation, there are also up to symmetry three different cases to consider which are
\begin{equation*}
\figins{-21}{0.65}{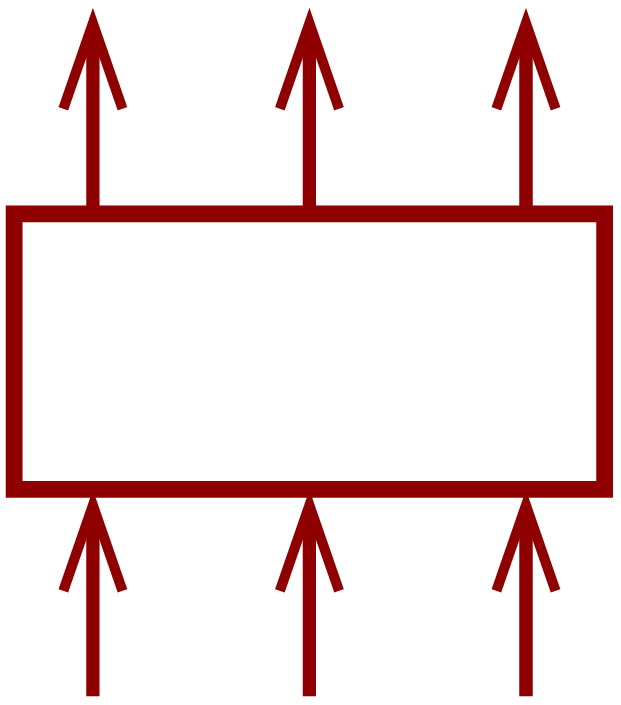}\ ,
\qquad
\figins{-21}{0.65}{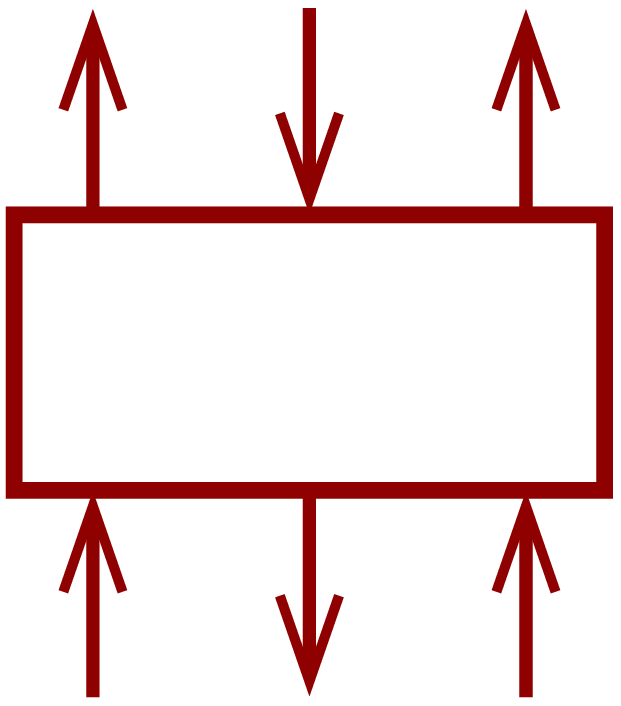}\quad\text{and}\quad
\figins{-21}{0.65}{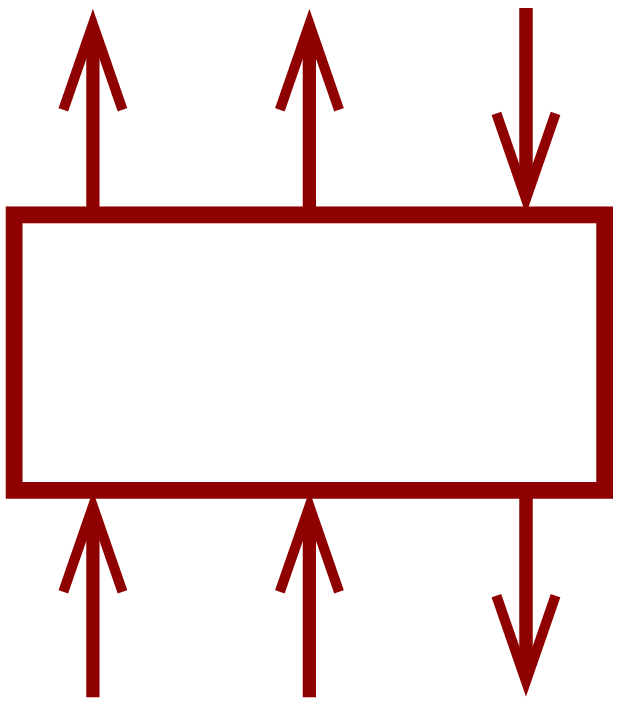}\ .
\end{equation*}
The first case is exactly given by the relation~(\ref{R3MOY}) of the HOMFLY-PT skein algebra.

We detail now the second case. The contribution of each term in the equality are the following ones:
\begin{align*}
\figins{-19}{0.60}{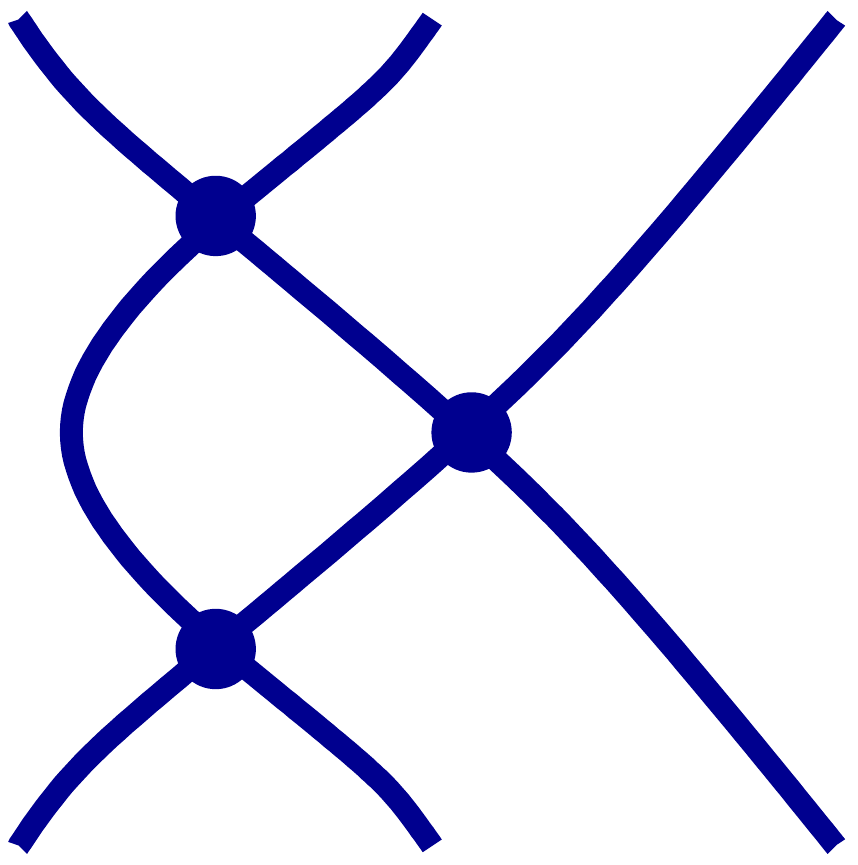} &\colon\quad
      q\ \ \figins{-19}{0.60}{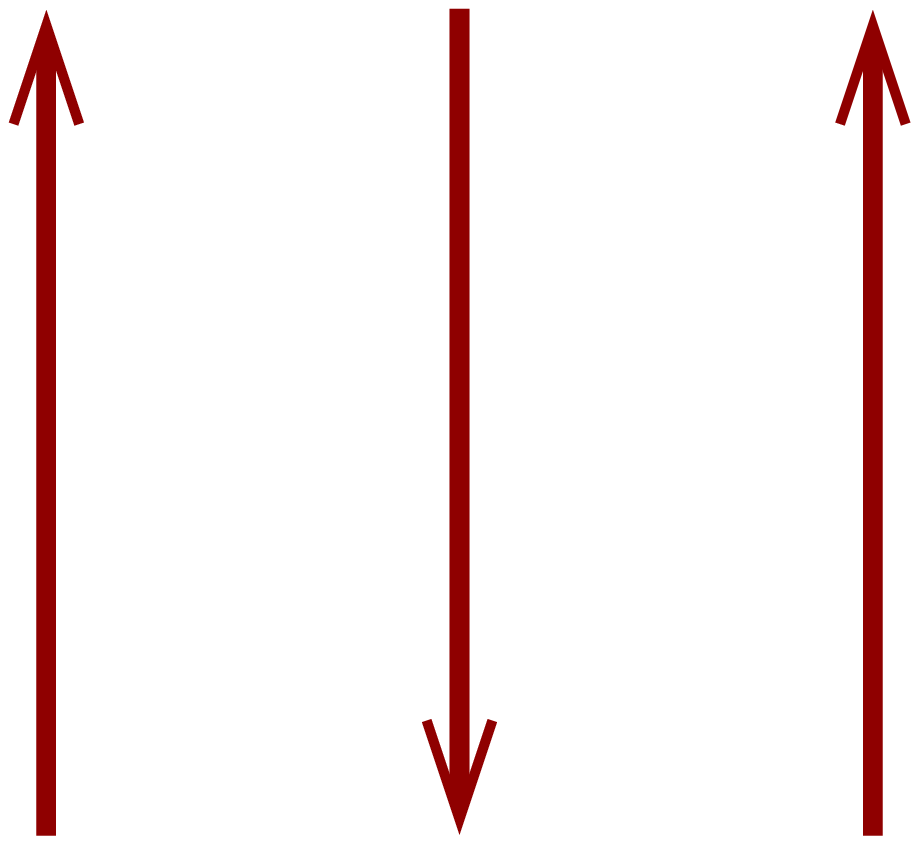}
  +aq^{-2}\ \figins{-19}{0.60}{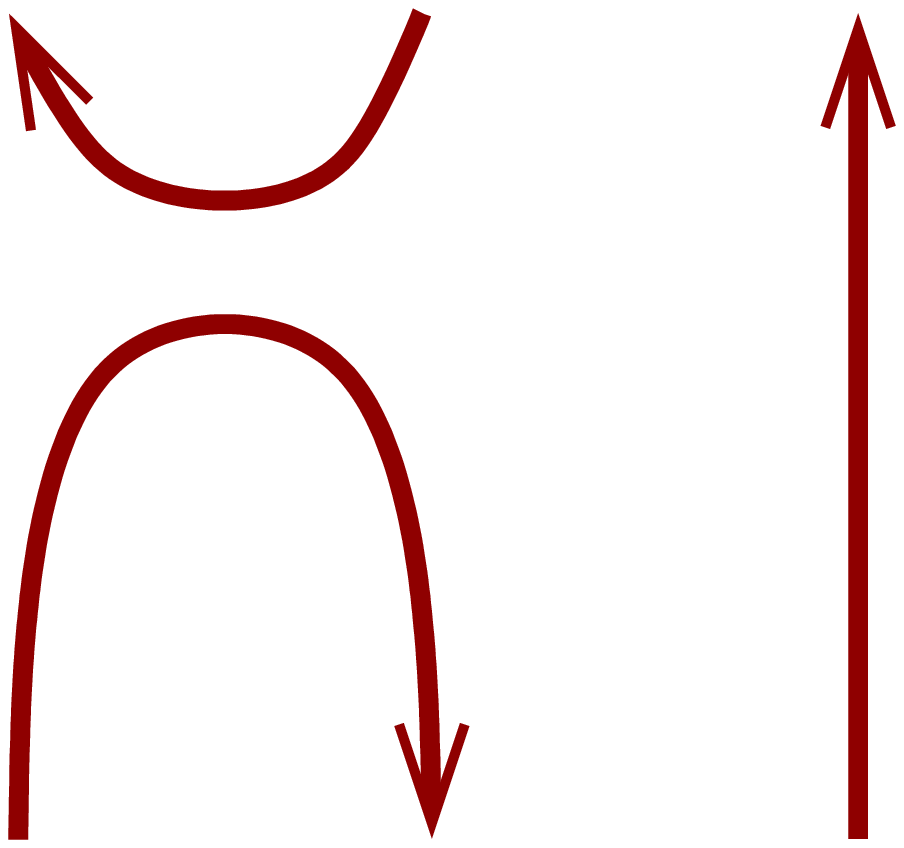}
  +aq^{-2}\ \figins{-19}{0.60}{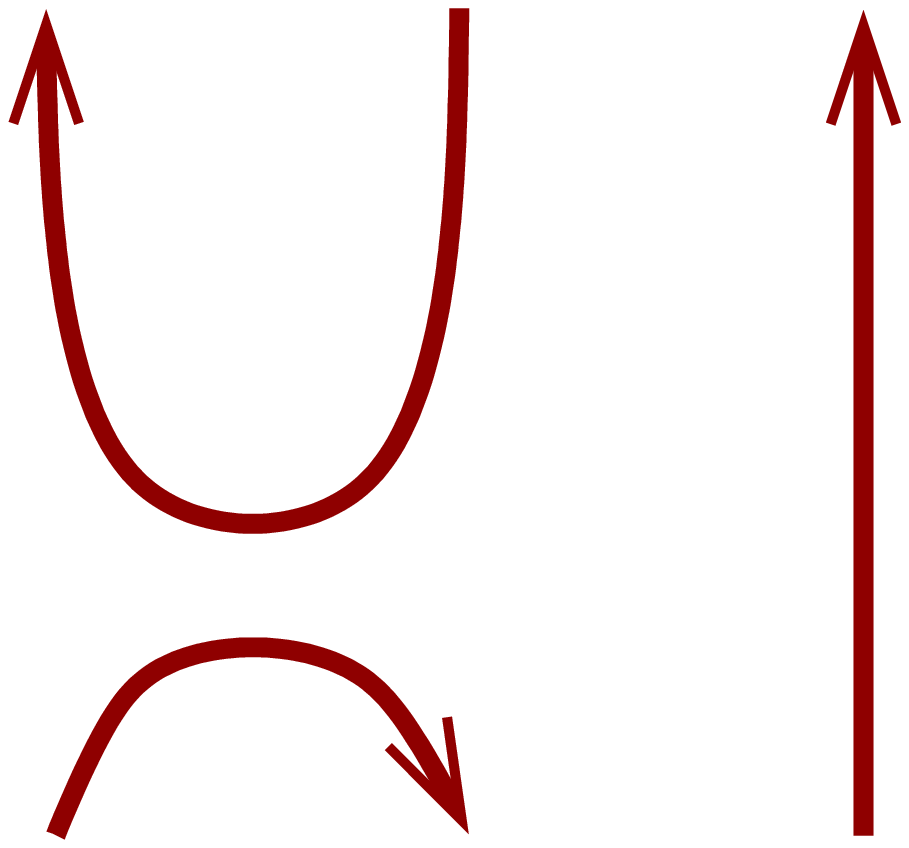}
+a^2q^{-5}\ \figins{-19}{0.60}{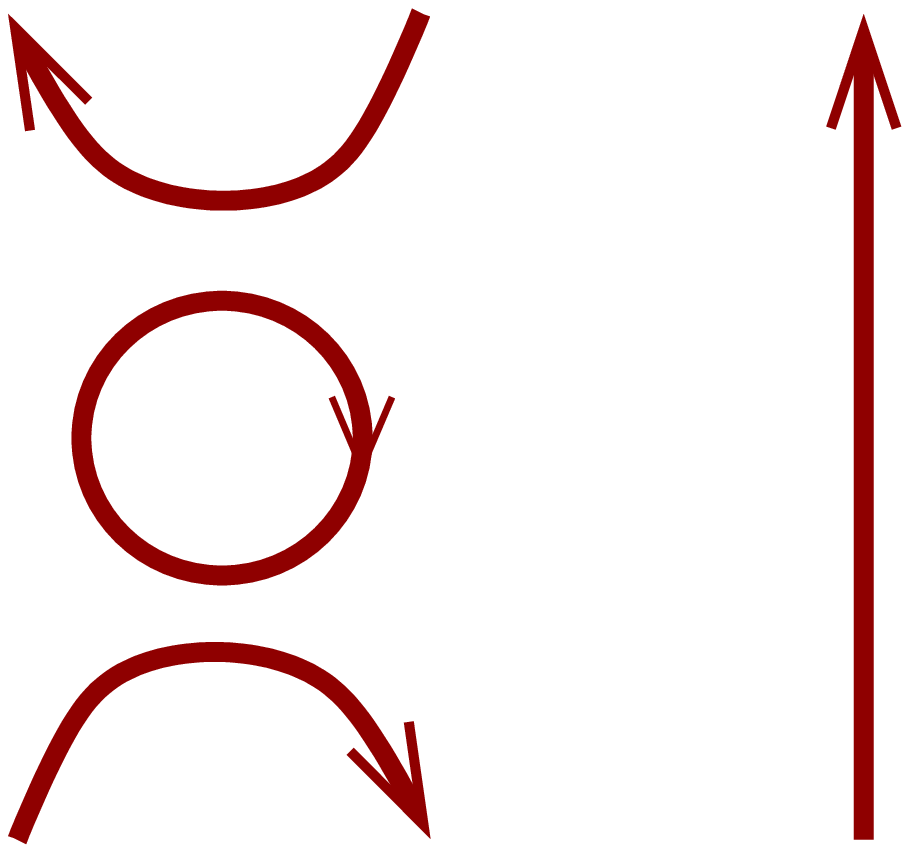}
\\[2ex] &\mspace{-20mu} 
+a^{-1}q^4\ \figins{-19}{0.60}{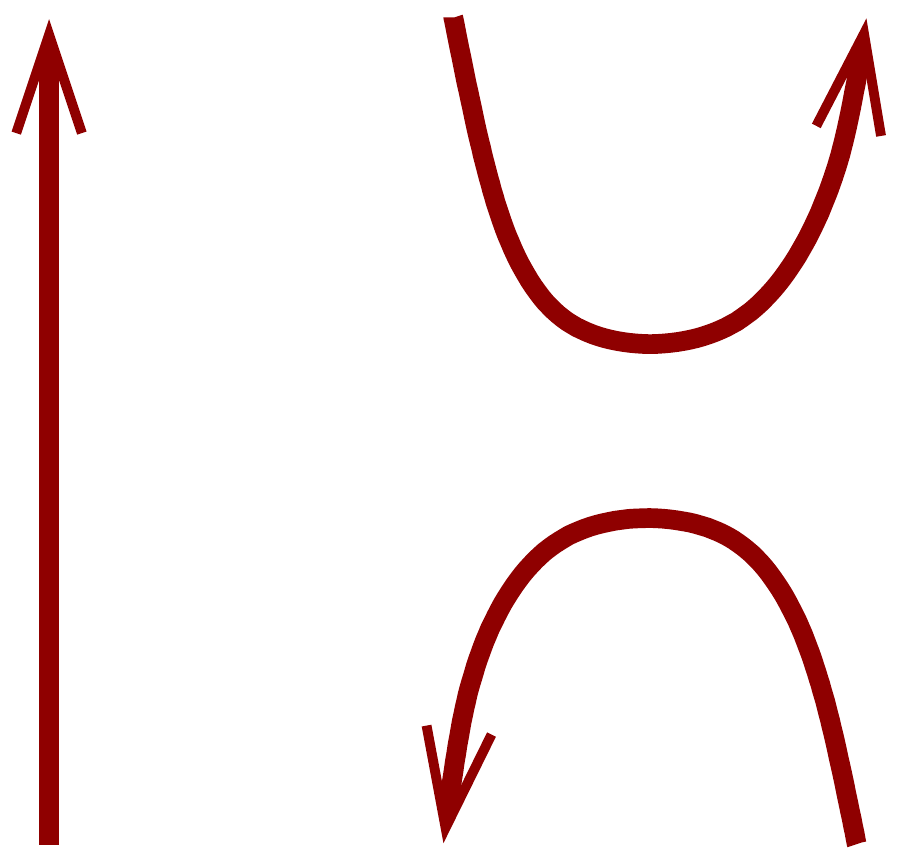}
       +q\ \figins{-19}{0.60}{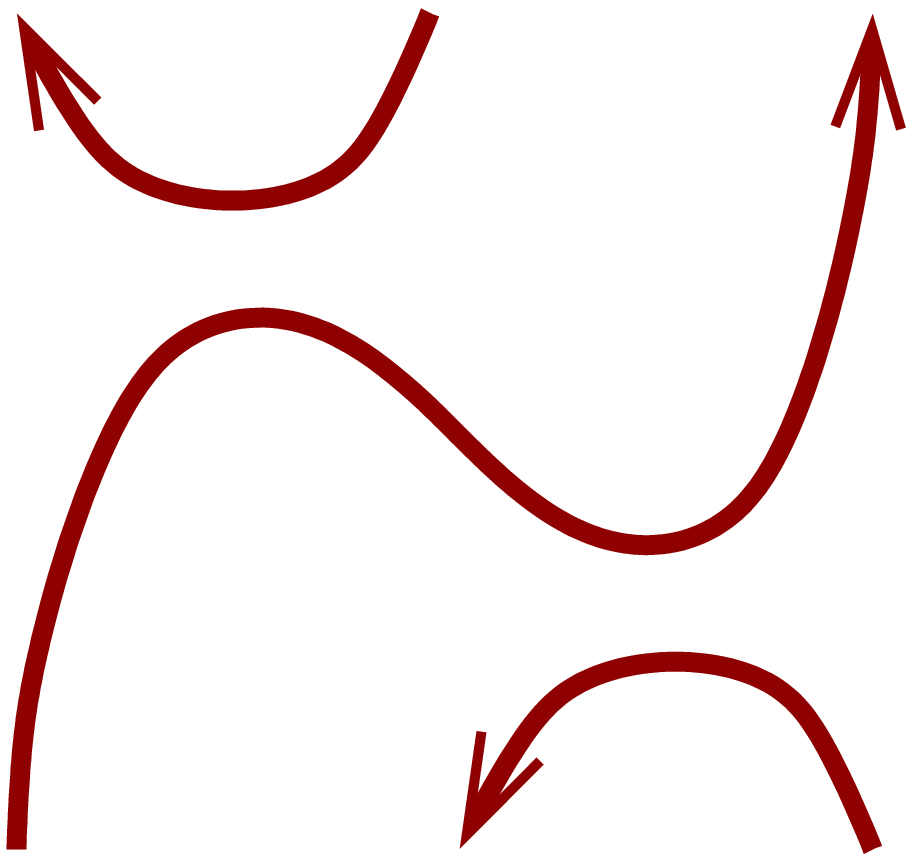}
       +q\ \figins{-19}{0.60}{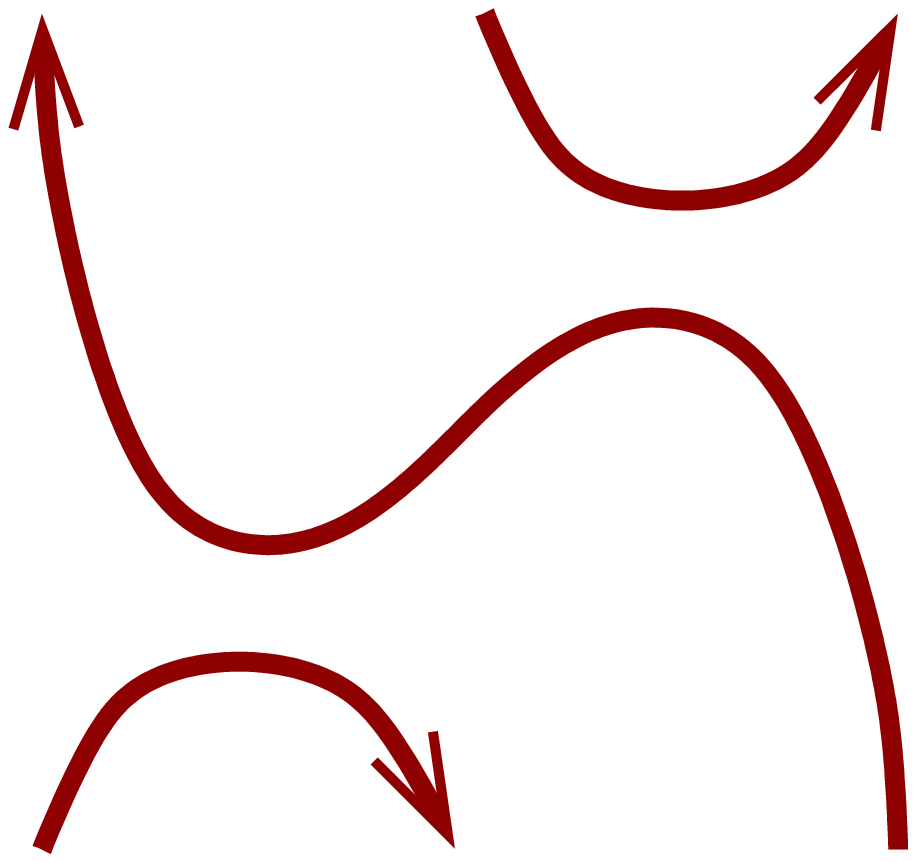}
  +aq^{-2}\ \figins{-19}{0.60}{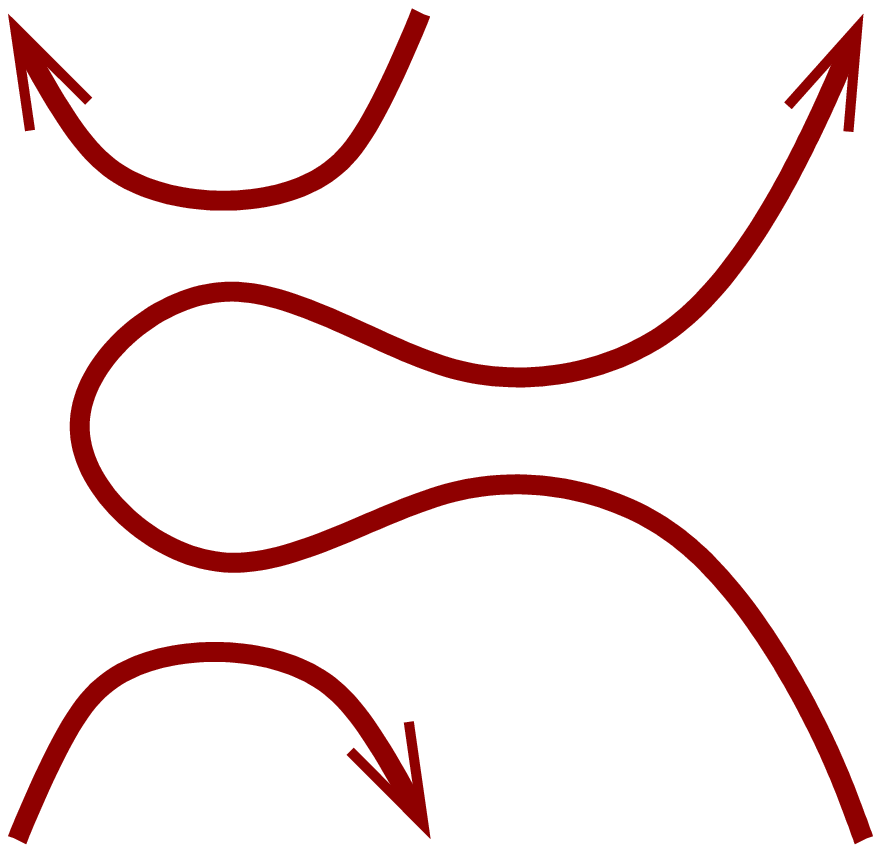}
        +\ \figins{-19}{0.60}{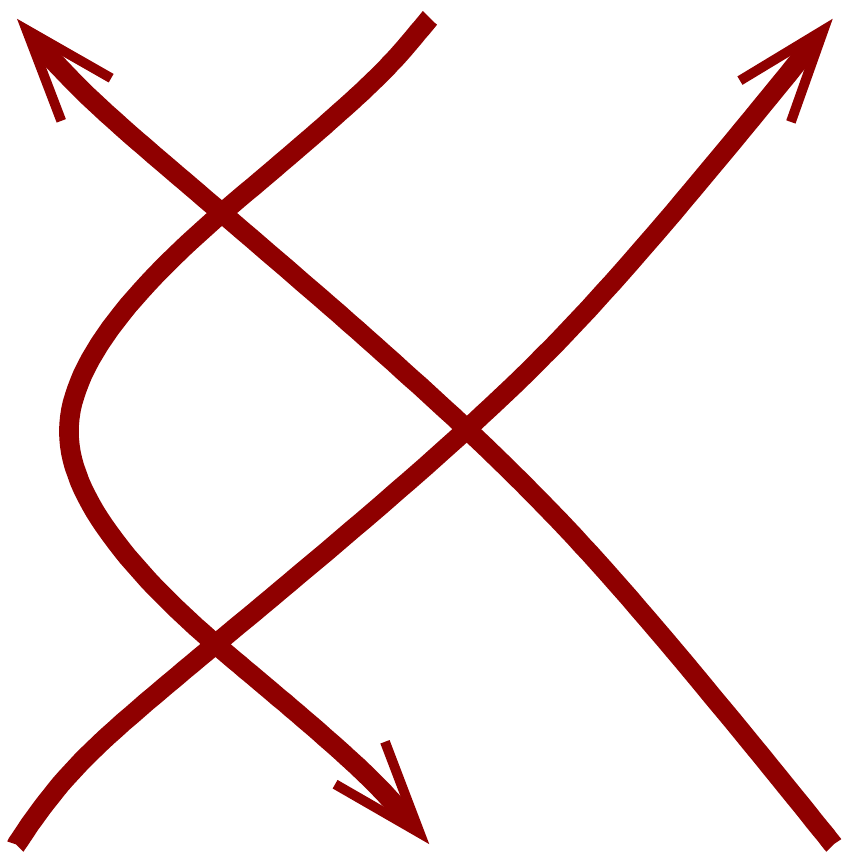} 
\displaybreak[0]\\[2ex]
\figins{-19}{0.60}{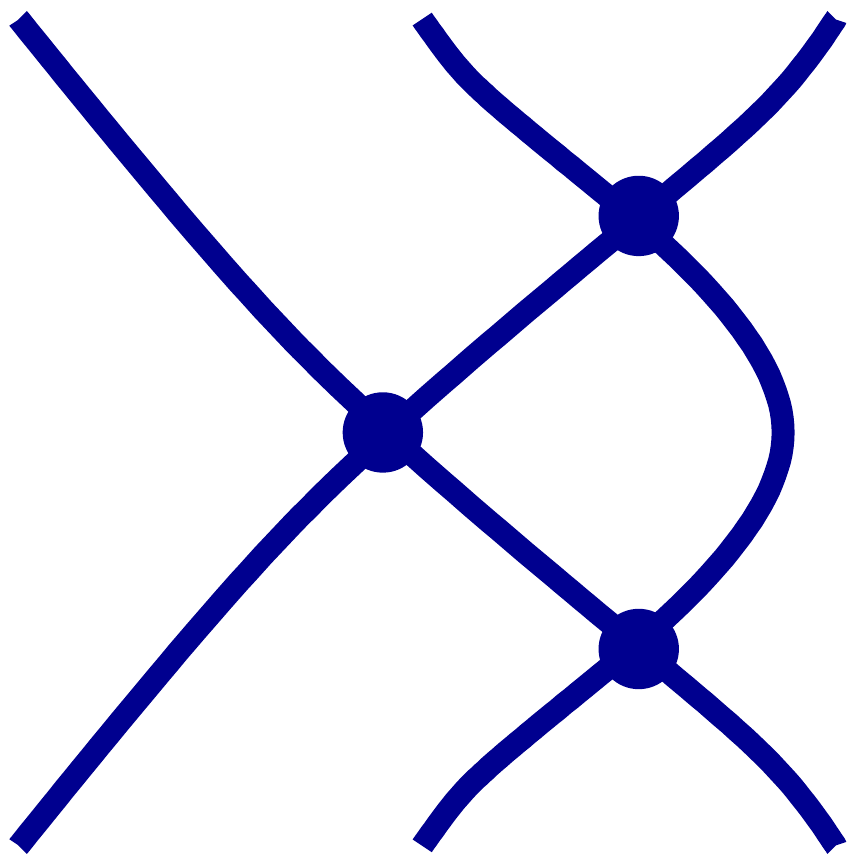} &\colon\quad
  q^{-1}\ \ \figins{-19}{0.60}{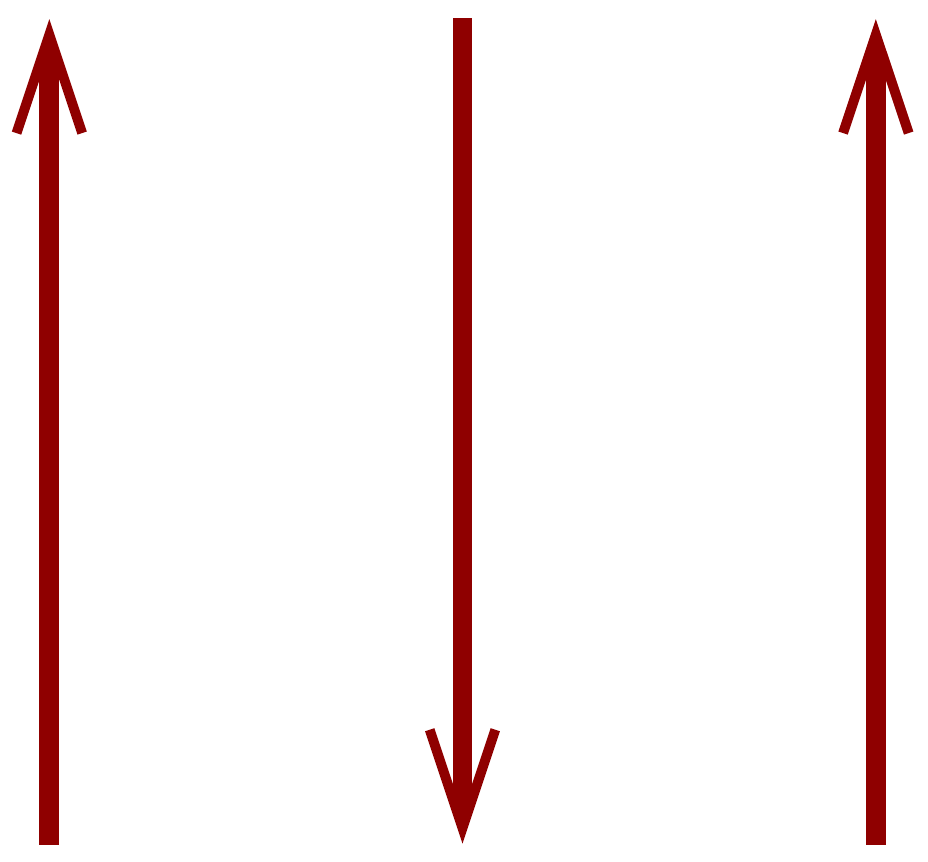}
 +a^{-1}q^2\ \figins{-19}{0.60}{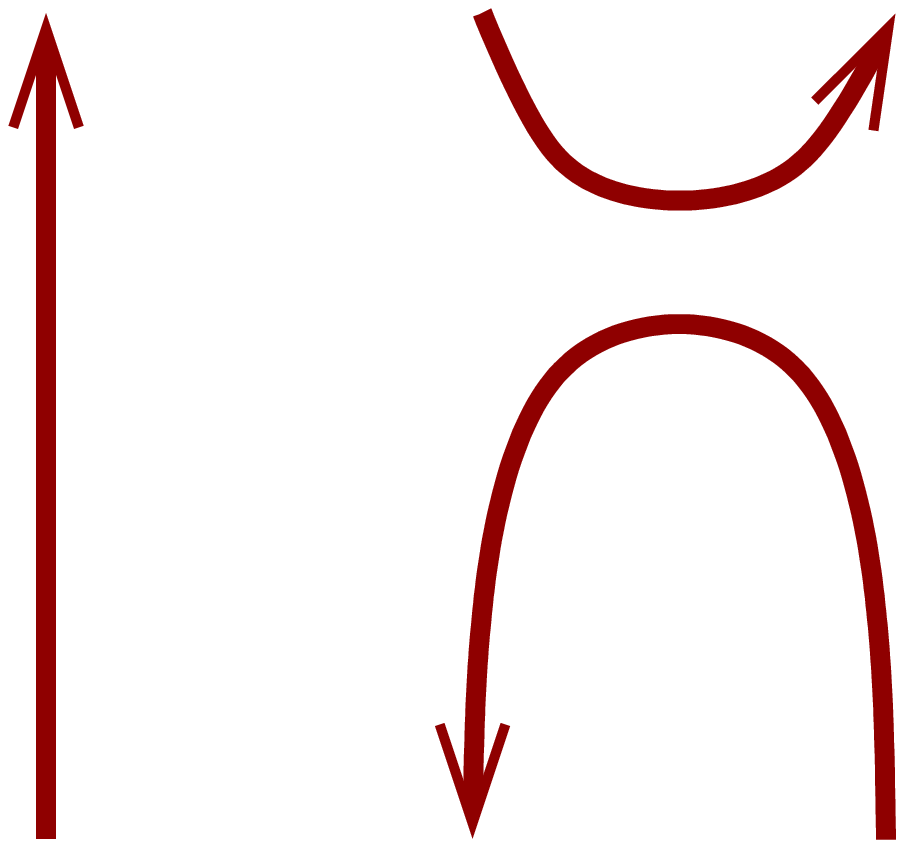}
 +a^{-1}q^2\ \figins{-19}{0.60}{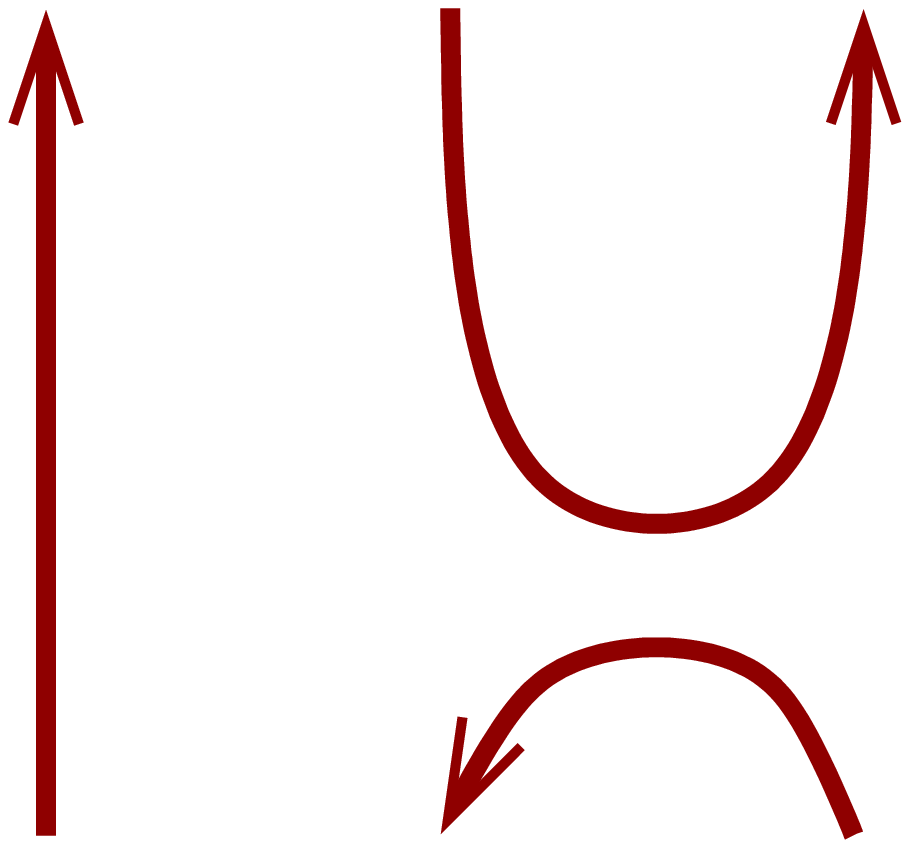}
 +a^{-2}q^5\ \figins{-19}{0.60}{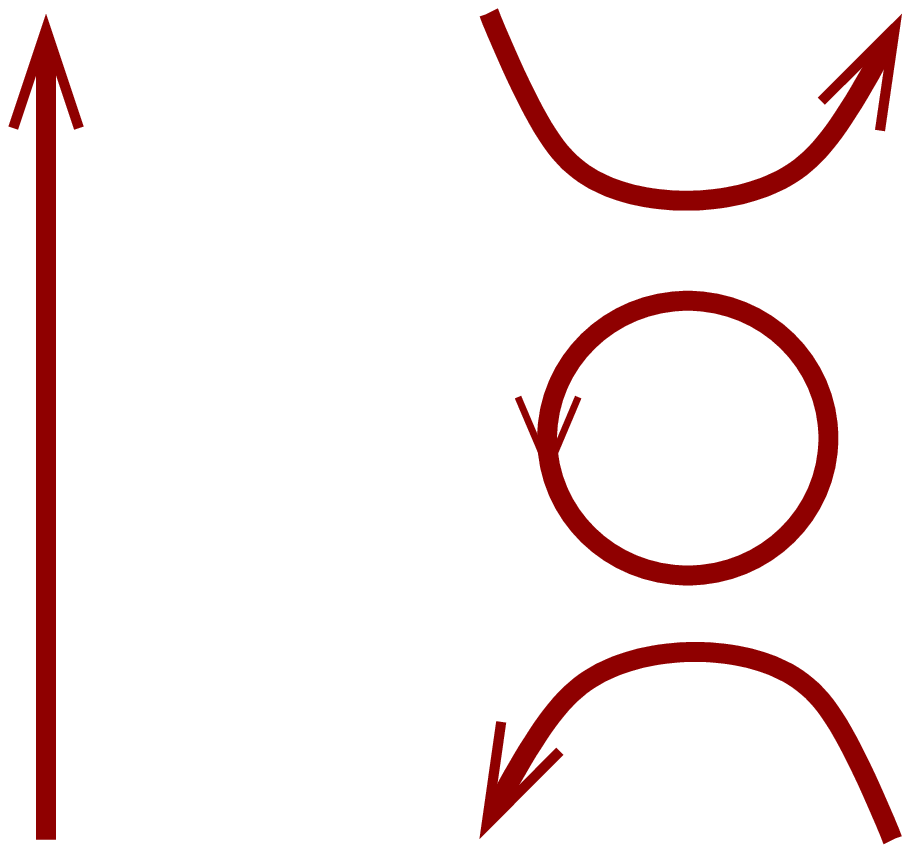}
\\[2ex] & \mspace{-45mu}
  +aq^{-4}\ \figins{-19}{0.60}{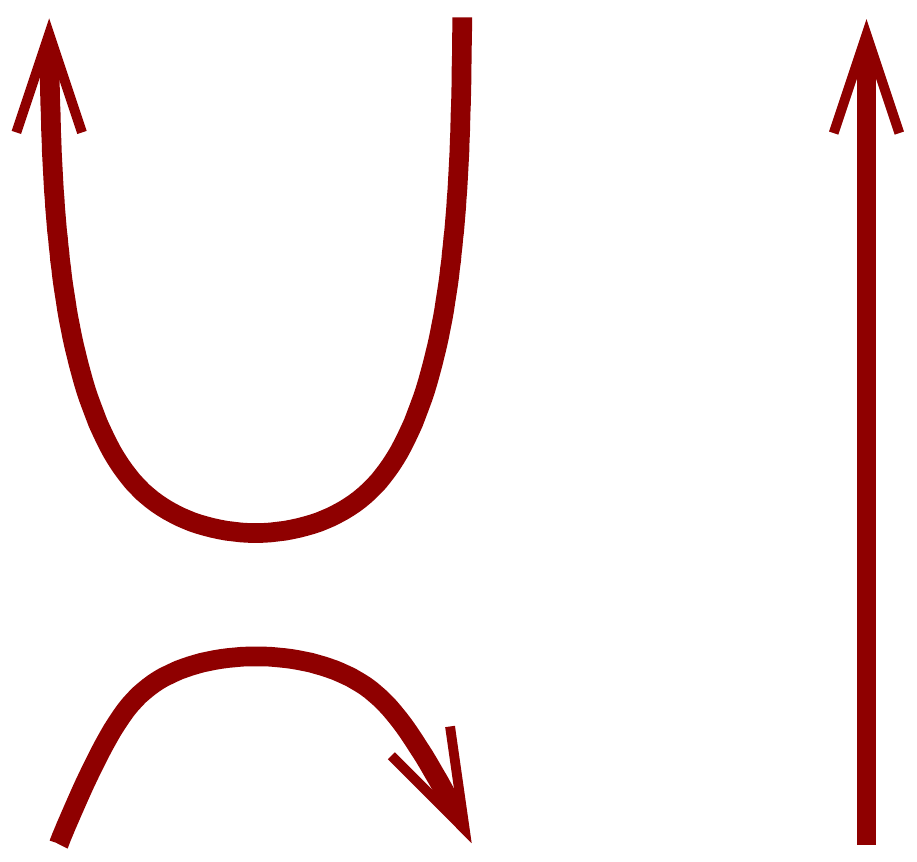}
   +q^{-1}\ \figins{-19}{0.60}{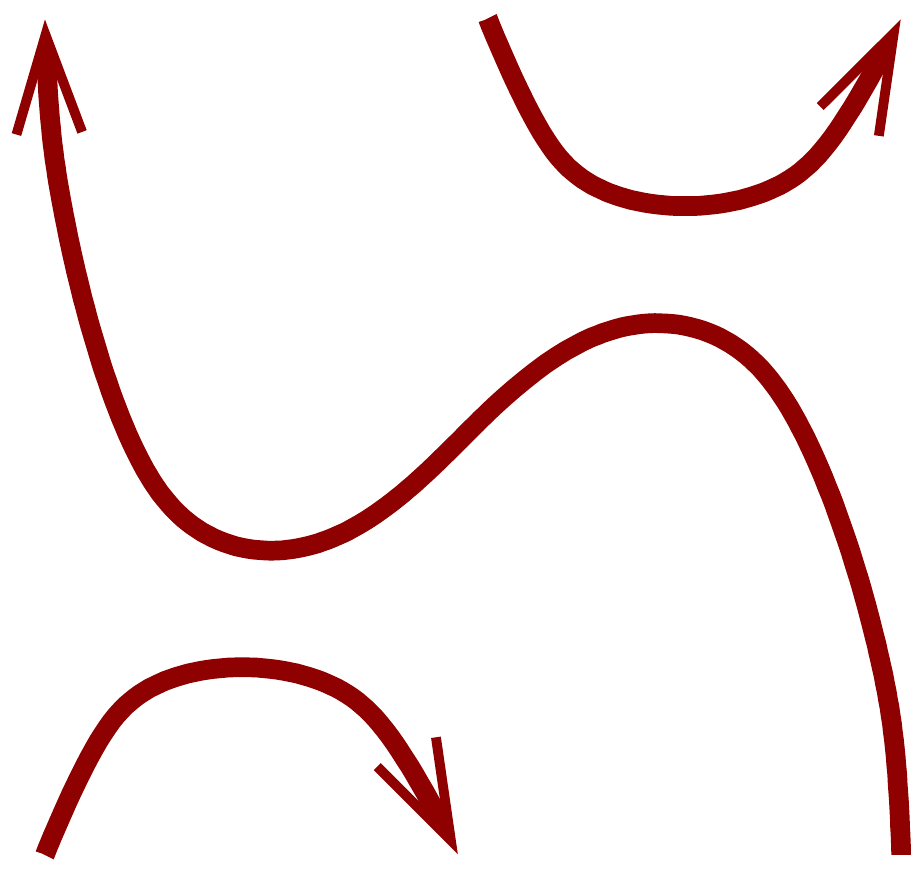}
   +q^{-1}\ \figins{-19}{0.60}{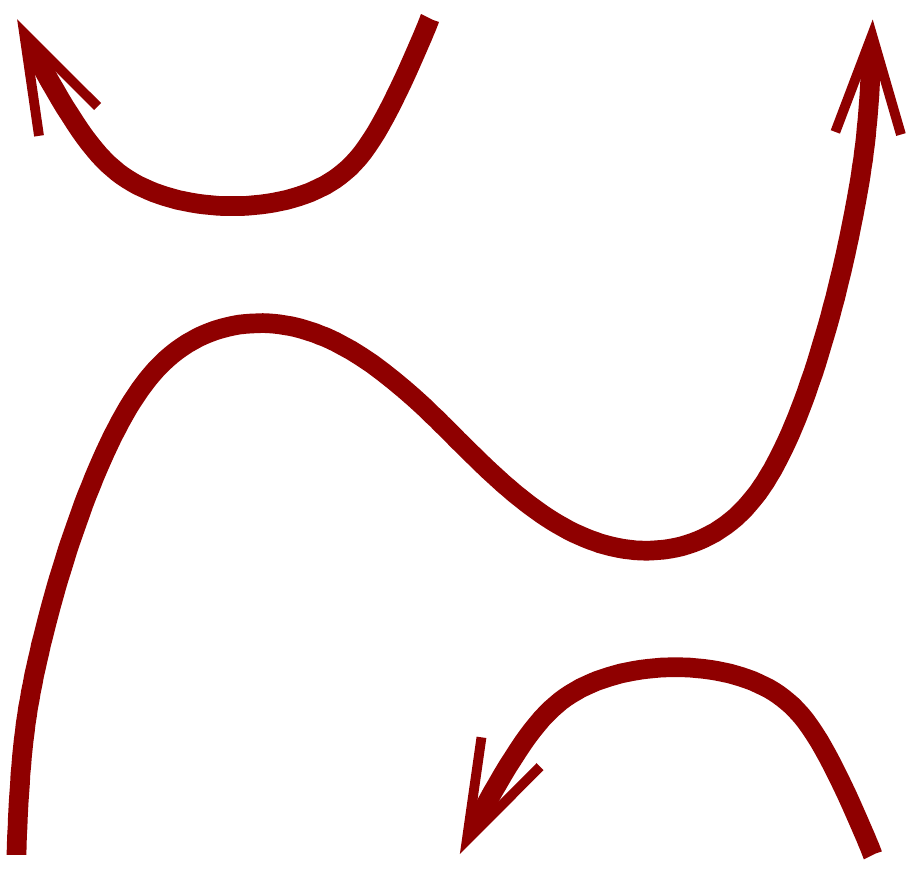}
 +a^{-1}q^2\ \figins{-19}{0.60}{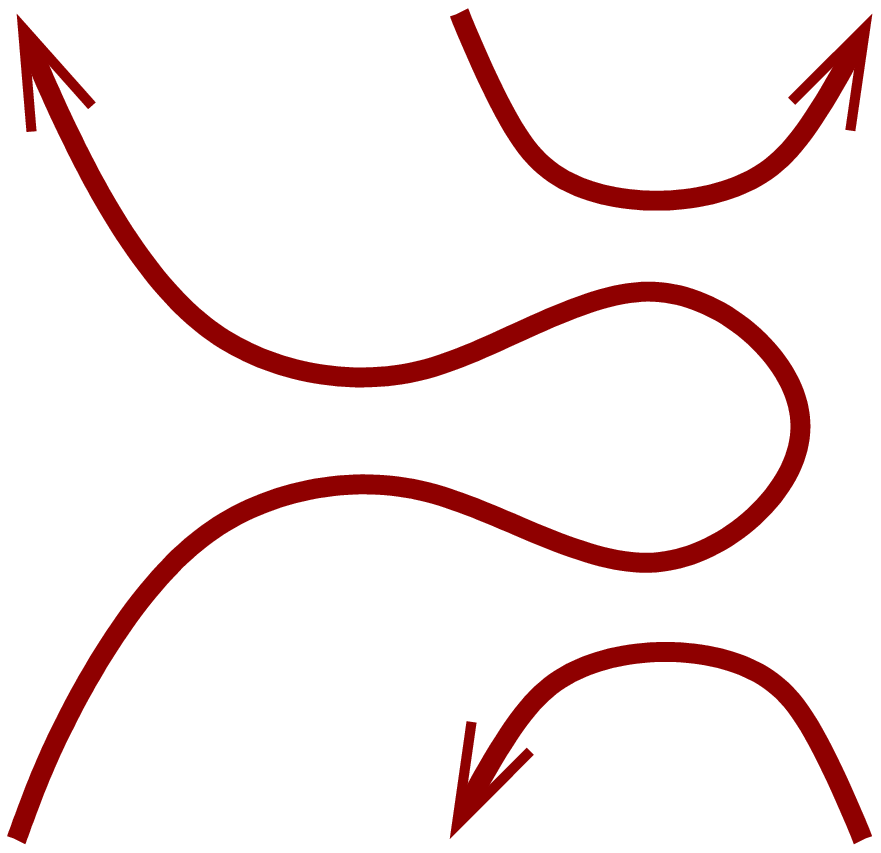}
        +\ \figins{-19}{0.60}{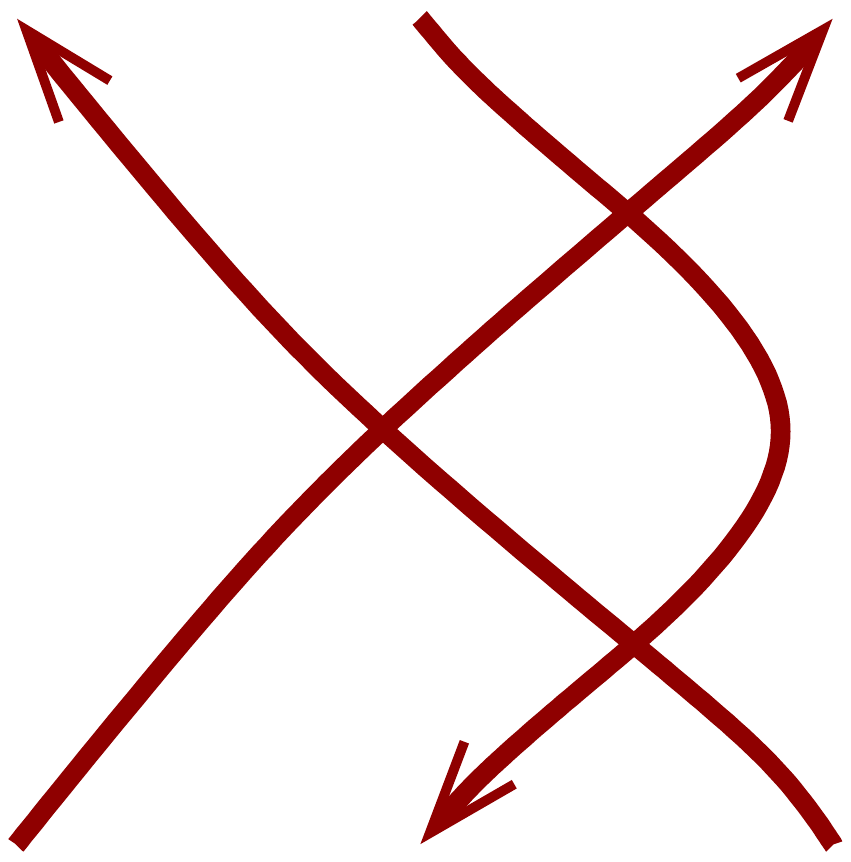}
 \displaybreak[0]\\[2ex]
\figins{-19}{0.60}{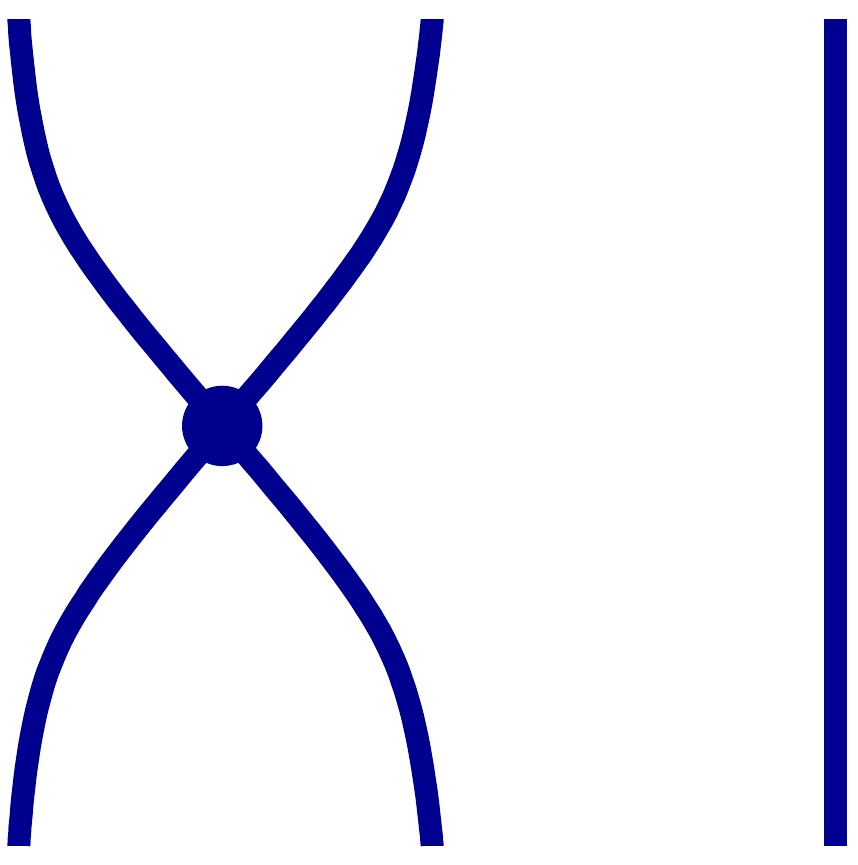}\ &\colon\quad
 aq^{-2}\ \figins{-19}{0.60}{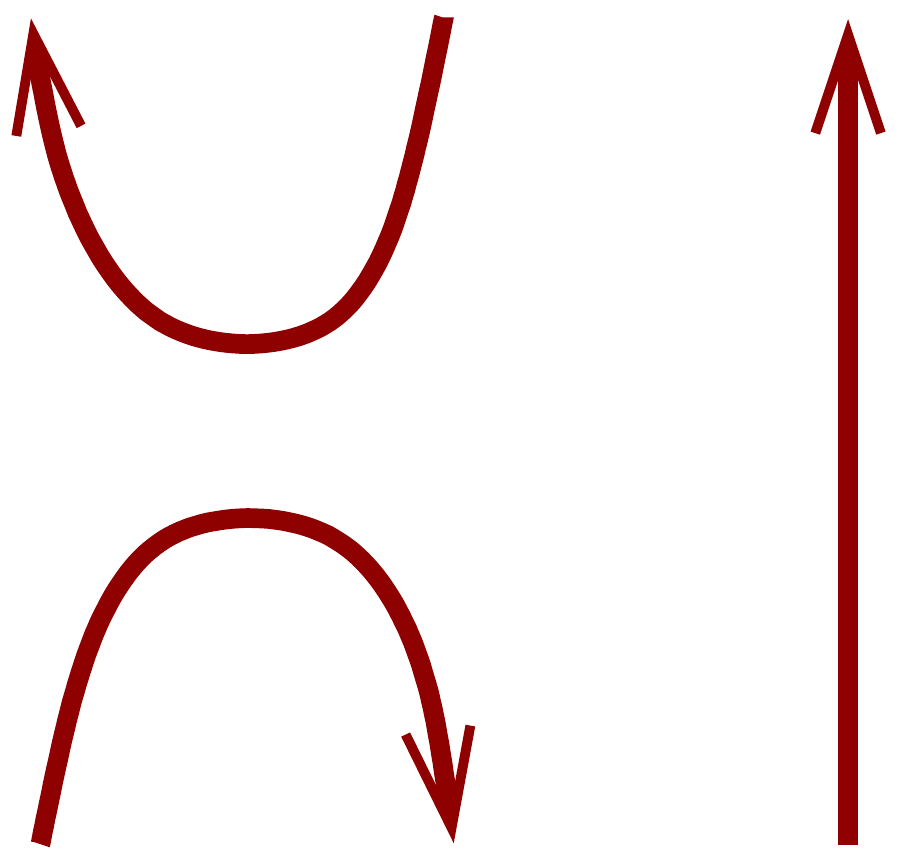}
+   q\ \ \figins{-19}{0.60}{updownupthin}
\displaybreak[0]\\[2ex]
\figins{-19}{0.60}{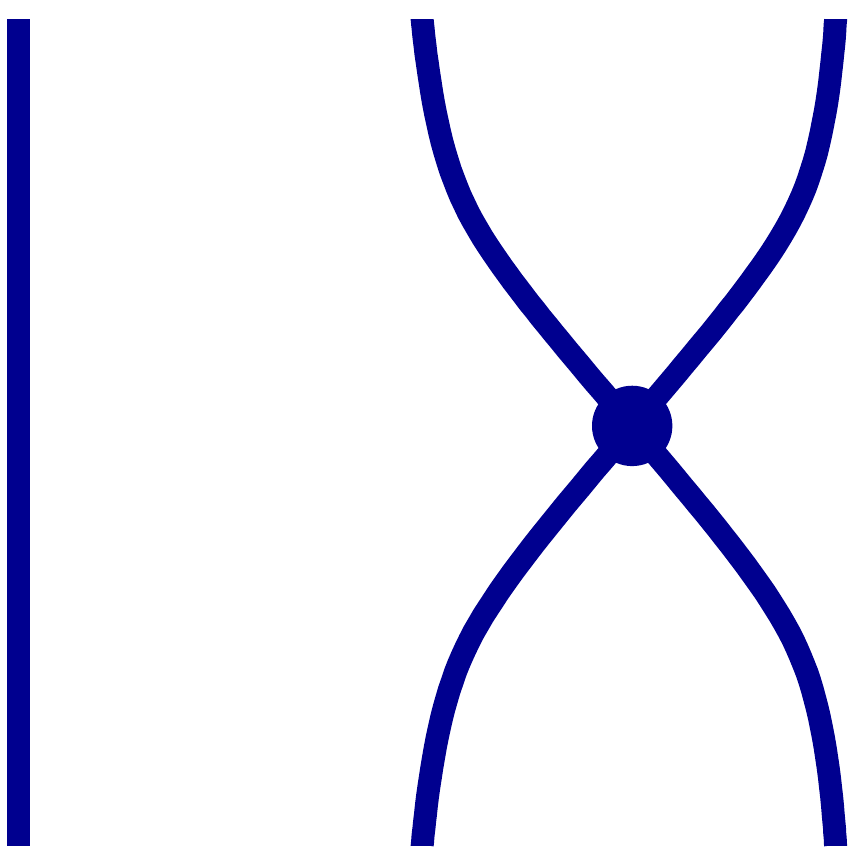}\ &\colon\quad
 a^{-1}q^2\ \figins{-19}{0.60}{capcup6}
+ q^{-1}\ \ \figins{-19}{0.60}{updownupthin}
\displaybreak[0]\\[2ex]
\figins{-19}{0.60}{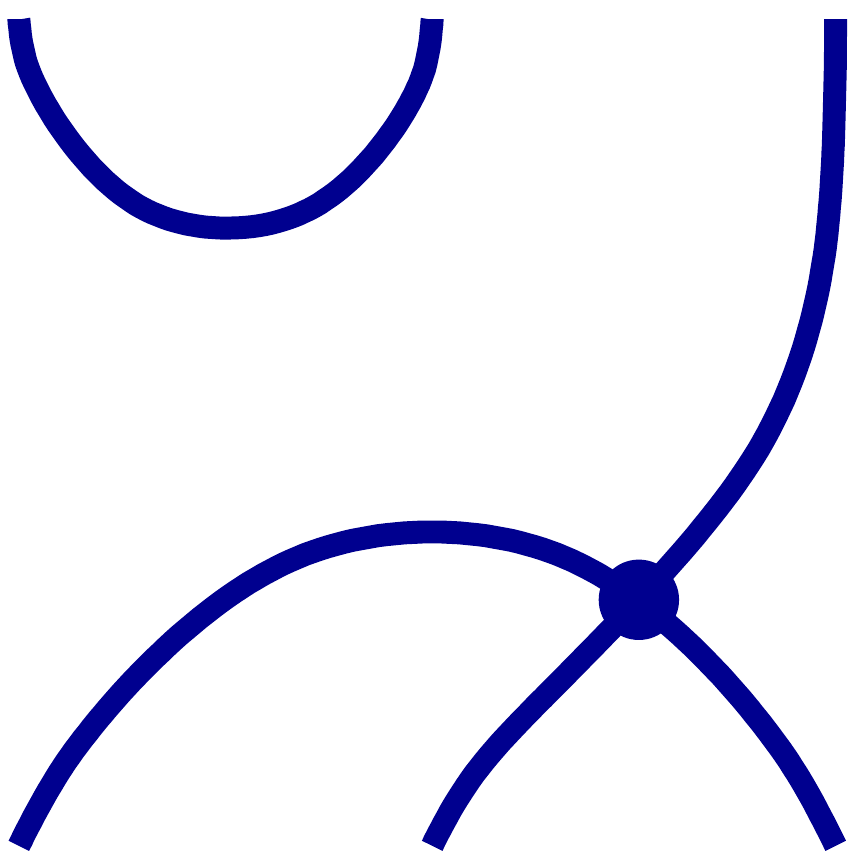} &\colon\quad
 aq^{-2}\ \figins{-19}{0.60}{capcup9}
     +q\ \figins{-19}{0.60}{zigzag1}
 \displaybreak[0]\\[2ex]
\figins{-19}{0.60}{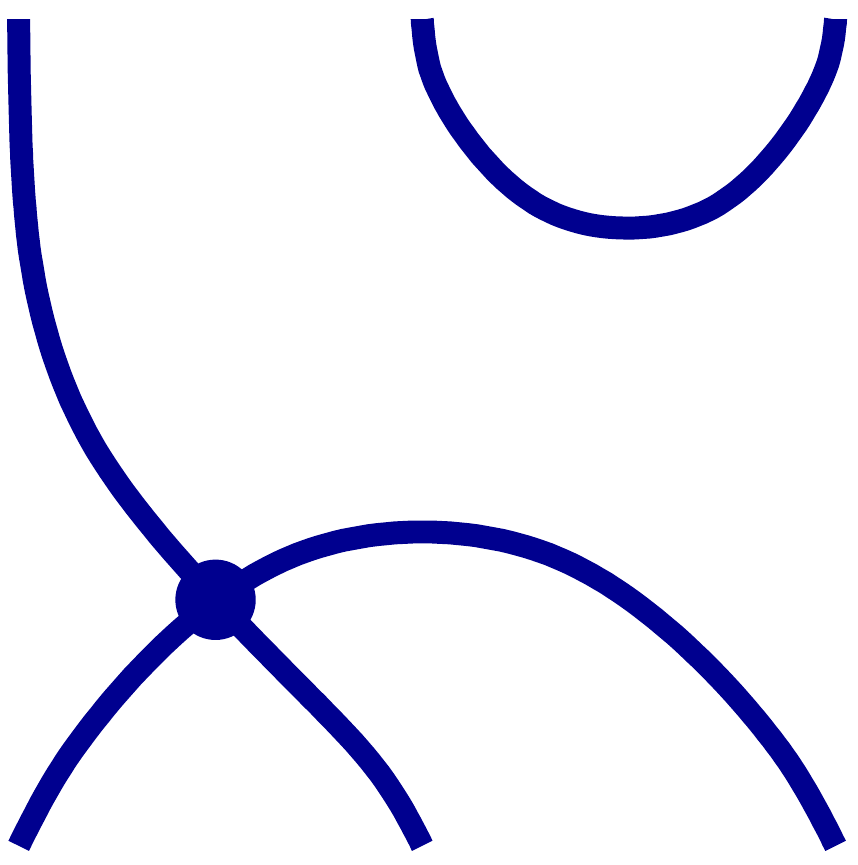} &\colon\quad
 a^{-1}q^2\ \figins{-19}{0.60}{capcup6}
   +q^{-1}\ \figins{-19}{0.60}{zigzag2}
\displaybreak[0]\\[2ex]
\figins{-19}{0.60}{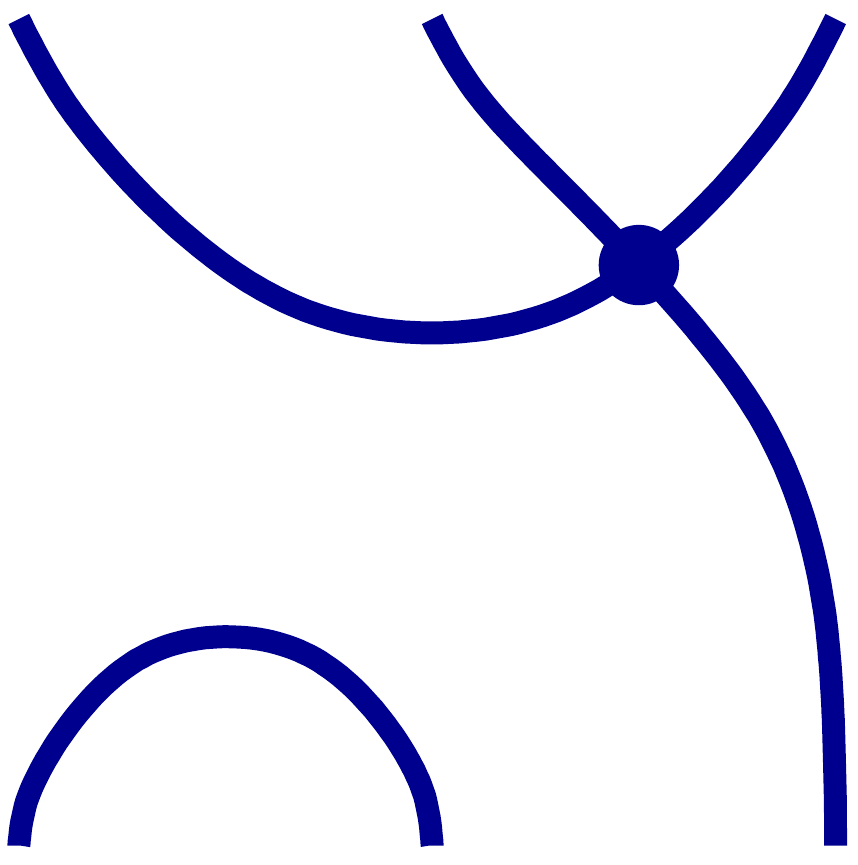} &\colon\quad 
 aq^{-2}\ \figins{-19}{0.60}{capcup9}
     +q\ \figins{-19}{0.60}{zigzag2}
\displaybreak[0]\\[2ex]
\figins{-19}{0.60}{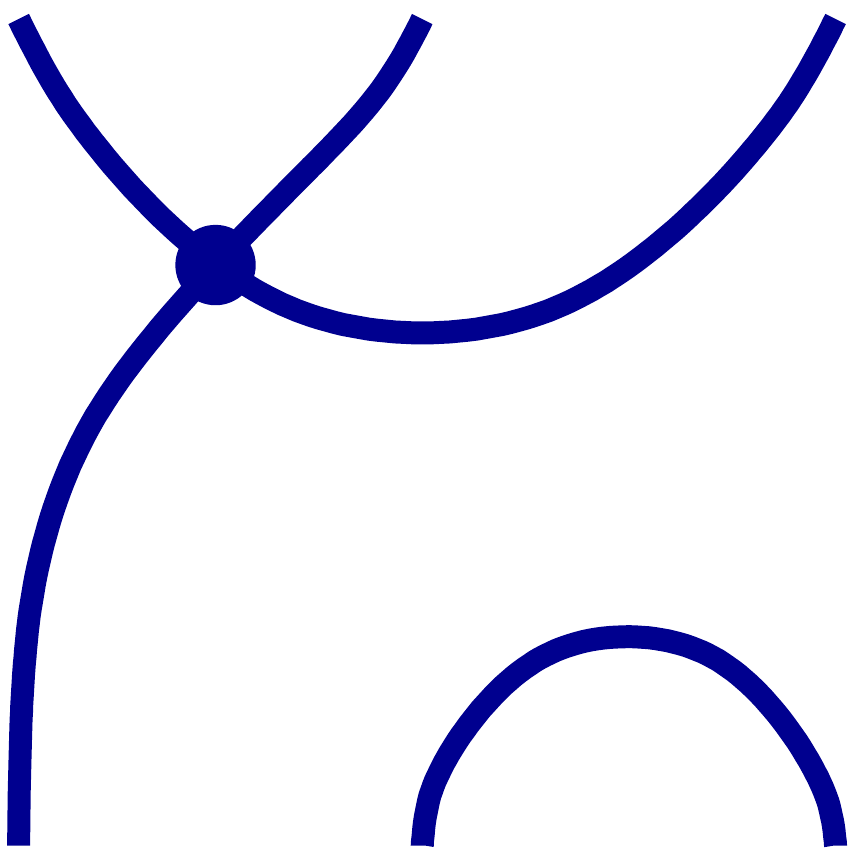} &\colon\
  a^{-1}q^2\ \figins{-19}{0.60}{capcup6}
   +q^{-1}\ \figins{-19}{0.60}{zigzag1}
\displaybreak[0]\\[2ex]
[a^2,-4]\ \figins{-19}{0.60}{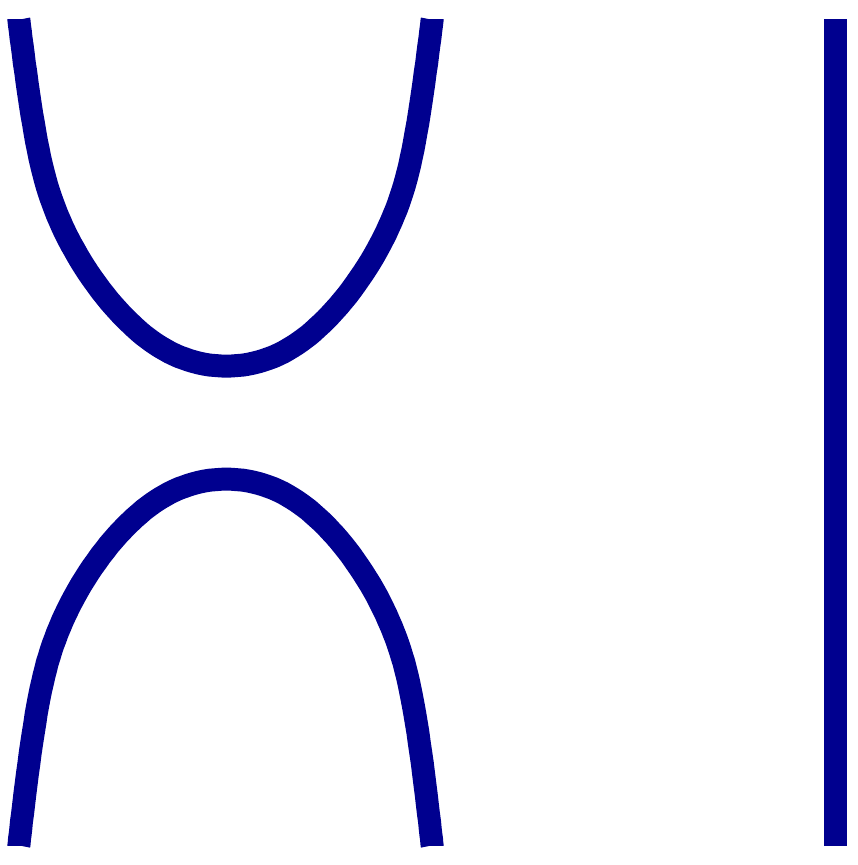}  &\colon\quad
aq^{-1}[a^2,-4]\ \ \figins{-19}{0.60}{capcup9}
\displaybreak[0]\\[2ex]
[a^2,-4]\ \figins{-19}{0.60}{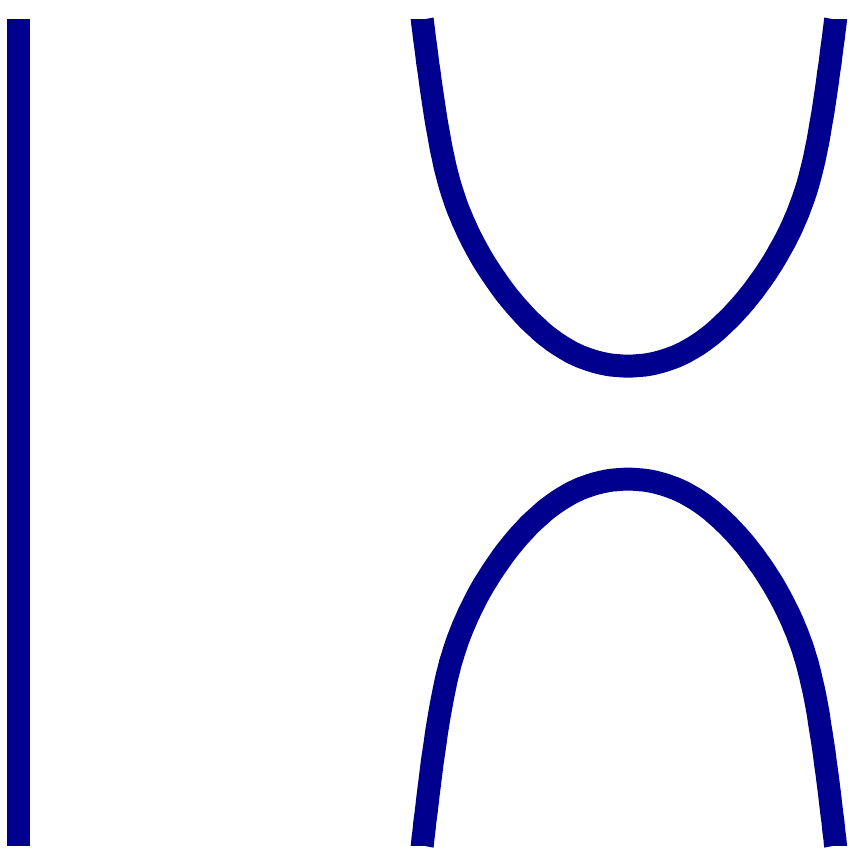} &\colon\quad 
a^{-1}q[a^2,-4] \ \figins{-19}{0.60}{capcup6}
\end{align*}
Among them, many cancel directly and we are left to check that
\begin{gather*}
\mspace{-80mu}
\figins{-19}{0.60}{reid3-ludu}\
+\ a^2q^{-5}\ \figins{-19}{0.60}{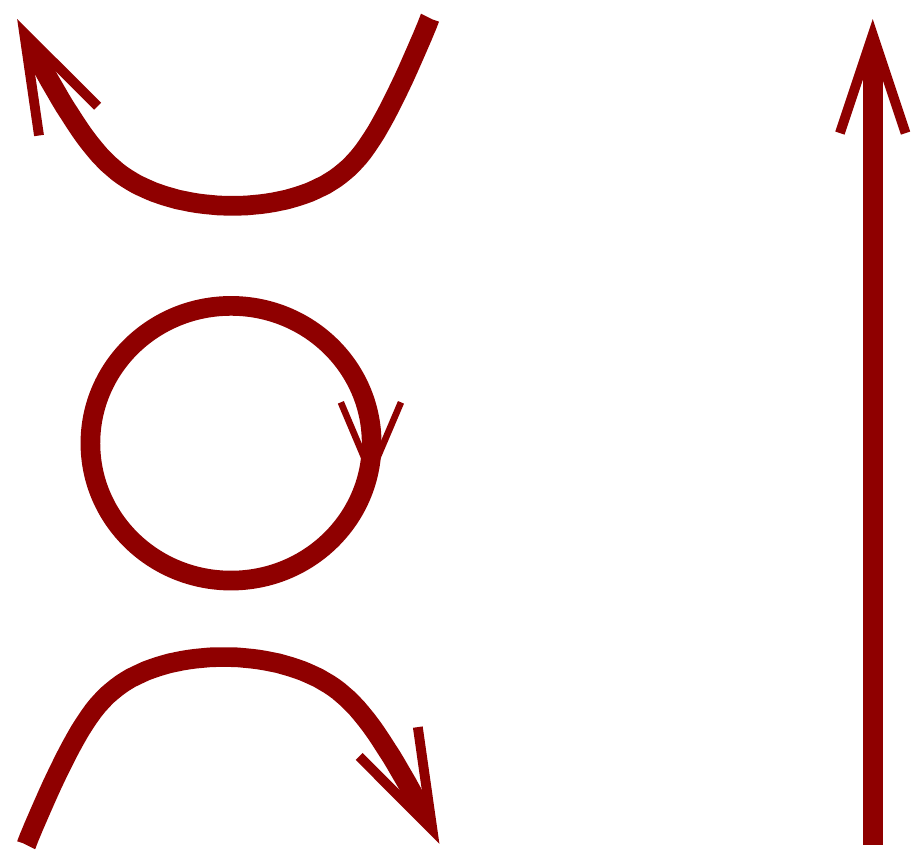}
+\ a^{-1}q^4\ \figins{-19}{0.60}{capcup6}
+\ a^{-1}q[a^2,-4]\ \figins{-19}{0.60}{capcup6}
\\[2ex]
\mspace{80mu}
=\ \figins{-19}{0.60}{reid3-rudu}\
+\ a^{-2}q^{5}\ \figins{-19}{0.60}{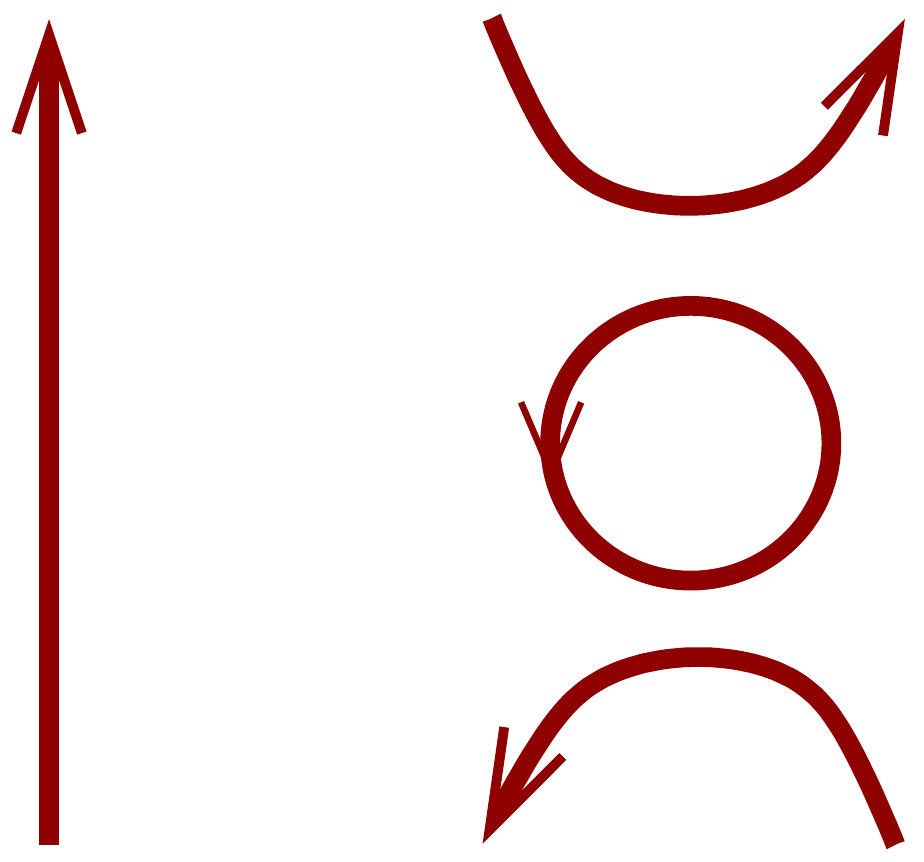}
+\ aq^{-4}\ \figins{-19}{0.60}{capcup9}
+\ aq^{-1}[a^2,-4]\ \figins{-19}{0.60}{capcup9}
\end{gather*}
The equality follows from additional relation (\ref{R3bMOY}) in HOMFLY-PT skein algebra.
The third case goes the same way using only the relation (\ref{R1MOY}) of the HOMFLY-PT skein algebra.
Hence $\psi$ is well defined.
The fact that $\psi$ is an algebra homomorphism follows from the fact that the rotational number 
is additive with respect to the multiplicative structure of $\skein_n(a,q)$ and the weight is multiplicative with 
respect to the multiplicative structure of 
$\skein_n(a,q)$.

We are left to prove $\psi$ is injective. Consider first the morphism of algebra $\hat{\psi}$ which is the same 
as $\psi$ but defined on the 
free algebras. In other words $\hat{\psi}$ goes from the $R$-algebra $F_n$ generated by isotopy classes of 
unoriented $(n,n)$ 4-valent graphs to the $R$-algebra $G_n$ generated by isotopy classes of 
oriented $(n,n)$ 4-valent graphs.  
We  now show that $\hat{\psi}$ is injective and descends to an injective morphism between 
the quotients $\bmw_n(a,q)$ and $\skein_n(a,q)$.
To this end we assume that it is injective on the subalgebra of $F_n$ generated
by the diagrams in $F_n$ containing less than $m$ vertices and
consider a linear combination 
$\sum_{i=1 }^k a_i \Gamma_i$ of elements of $F_n$, 
with each term having $m$ $4$-valent vertices.
Then for each $\Gamma_i$ consider the projection $\hat{\Gamma_i}$ of $\hat{\psi}(\Gamma_i)$ to the subvector 
space of $G_n$ generated by $(n,n)$ 4-valent graphs  with exactly $m$ vertices. Notice that $\hat{\Gamma_i}$ 
consists of a weight direct sum of graphs which are obtained by choosing an orientation for each arc in $\Gamma_i$, 
each of these graphs being linearly independent of the others.
In addition for any given choice of an orientation of the endpoints of the $\Gamma_i$'s, the graphs are linearly independent. 
It follows that if 
$\sum_{i=1 }^k a_i \Gamma_i=0$ then $a_i=0$ for all $i=1,\ldots,k$, hence $\hat{\psi}$ sends linearly independent
elements of $F_n$ to linearly independent elements of $G_n$ and hence is injective. 
Consider now the morphism from $\hat{\psi}(F_n)$ to $\bmw_n(a,q)$ using the projection $\pi$ of $F_n$ to $\bmw_n(a,q)$. 
The previous proof of the fact that $\psi$ is well defined implies that the kernel of $\pi$ is exactly 
the intersection of the relations in $\skein_n(a,q)$ with $\hat{\psi}(F_n)$. 
The image through $\hat{\psi}$ of 
the relations defining $\bmw_n(a,q)$ are exactly the relations of $\skein_n(a,q)$. 
This concludes the proof of the injectivity of $\psi$.
\end{proof}
%
%
%
%
In the sequel we will make use of the explicit form of Jaeger's homomorphism,
given on the generators of $\bmw_n(a,q)$ by
\begin{align*}
\figins{-10}{0.35}{cupcap}\ 
&\overset{\psi}{\longmapsto}\ \
\figins{-10}{0.35}{cupcap-ll}\ \ +\ \
\figins{-10}{0.35}{cupcap-rr}\ \ +\ \  a^{-1}q\
\figins{-10}{0.35}{cupcap-lr}\ \ +\ \  a q^{-1}\
\figins{-10}{0.35}{cupcap-rl}
\\[2ex]
\figins{-10}{0.35}{4vert}\
&\overset{\psi}{\longmapsto}\ \
q^{-1}\biggl(a q^{-1}\
\figins{-10}{0.35}{cupcap-rl}\ \ +\ \
\figins{-10}{0.35}{downup}\ \ \biggr)
+ q\biggl( a^{-1}q
\figins{-10}{0.35}{cupcap-lr}\ \ +\ \
\figins{-10}{0.35}{updown}\ \ \biggr)
\\[2ex] 
& \mspace{60mu} + \ \
\figins{-10}{0.35}{crossu}\ \ + \ \
\figins{-10}{0.35}{crossl}\ \ + \ \
\figins{-10}{0.35}{crossd}\ \ + \ \
\figins{-10}{0.35}{crossr}
\\[2ex]
\figins{-10}{0.35}{one}\quad\ 
&\overset{\psi}{\longmapsto}\ \ \
\figins{-10}{0.35}{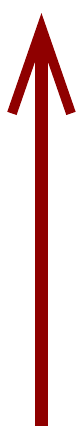}\ \ \ + \ \ \ 
\figins{-10}{0.35}{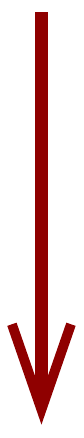}
\end{align*}
Notice that all the coefficients of the expansion of the generators of $\bmw_n(a,q)$ are in
$\bN(a,q)$. From the description of the HOMFLY-PT skein algebra in Subsection~\ref{ssec:skein} 
we conclude that, under $\psi$, any diagram in $\bmw_n(a,q)$ 
can be written as a linear combination of diagrams in $\skein_n(a,q)$ with coefficients in $\bN(a,q)$.

\medskip


Jaeger's homomorphism from Equation~\eqref{eq:jaeger-alg} induces the Jaeger's formula for polynomials of 
Equation~\eqref{eq:jaeg-poly} in the following sense.
The maps $f_{\bmw}\colon\bmw_n^\tau(a,q)\to\bmw_n(a,q)$ given by
\begin{equation*}
\figins{-10}{0.35}{Xing-p}\ \mapsto\
q\ \figins{-10}{0.35}{cupcap}\  -\  
\figins{-10}{0.35}{4vert}\ +\
q^{-1}\ \figins{-10}{0.35}{cupcap-vert}
\end{equation*}
and 
$f_{\skein}\colon\skein_n^\tau(a,q)\to\skein_n(a,q)$ given by
\begin{equation*}
\figins{-10}{0.35}{Xing-pu}\ \mapsto\  
a^{-1}q\ \figins{-10}{0.35}{upup}\ -\ 
a^{-1}\figins{-10}{0.35}{crossu} 
\qquad\text{and}\qquad
\figins{-10}{0.35}{Xing-nu}\ \mapsto\
aq^{-1}\ \figins{-10}{0.35}{upup}\ -\ 
a\figins{-10}{0.35}{crossu}
\end{equation*}
can be used to obtain a version of Jaeger's homomorphism in terms of the tangle 
algebras $\bmw_n^\tau(a,q)$ and $\skein_n^\tau(a,q)$ using the following procedure.
Inverting $f_{\skein}$,
\begin{equation*}
\figins{-10}{0.35}{crossu} \ \xmapsto{f_{\skein}^{-1}}\
q\ \figins{-10}{0.35}{upup}\ -\ 
a\figins{-10}{0.35}{Xing-pu}\ =\
q^{-1}\ \figins{-10}{0.35}{upup}\ -\ 
a^{-1}\figins{-10}{0.35}{Xing-nu}
\end{equation*}
we define $\psi^\tau:=f_{\skein}^{-1}\psi f_{\bmw}\colon\bmw_n^\tau(a,q)\to\skein_n^\tau(a,q)$.

\medskip

The relations imposed on the BMW and HOMFLY-PT skein algebras imply that 
closed diagrams in $\bmw^\tau_n(a,q)$ and $\skein_n^\tau(a,q)$ reduce to polynomials, 
which coincide with the Kauffman or HOMFLY-PT polynomials respectively.

\begin{prop}\label{prop:jaegtang}
For a closed tangle diagram $D$ we have that $\psi^\tau(D)$ coincides with the Jaeger 
expansion for link polynomials in Theorem~\ref{thm:jaeg-link}. 
\end{prop}

%
\section{Setting $a=q^N$: embeddings and projections}\label{sec:emb}

\subsection{The $q$-Schur algebra $S_q(n,d)$} %
\label{ssec:schur}                            %

Another algebra that enters the play is the $q$-Schur algebra $S_q(n,d)$. 
In this subsection we briefly review $S_q(n,d)$ following the exposition in~\cite{MSV} 
(see~\cite{MSV} and the references therein for more details). 
The $q$-Schur algebra appears naturally in the context of (polynomial) representations of ${U}_q(\mathfrak{gl}_n)$,
which is the starting point of this subsection.

Let 
$\epsilon_i=(0,\ldots,1,\ldots,0)\in \bZ^n$, with $1$ being on the $i$th 
coordinate for $i=1,\ldots,n$.
Let also
$\alpha_i=\epsilon_i-\epsilon_{i+1}\in\bZ^{n}$ 
and $(\epsilon_i,\epsilon_j)=\delta_{i,j}$ be the Euclidean inner product on $\bZ^n$. 
\begin{defn}
The quantum general linear algebra ${U}_q(\mathfrak{gl}_n)$ is the 
associative unital $\bQ(q)$-algebra generated by $K_i,K_i^{-1}$, for $i=1,\ldots, n$, 
and $E_{\pm i}$, for $i=1,\ldots, n-1$, subject to the relations
\begin{gather*}
K_iK_j=K_jK_i\quad K_iK_i^{-1}=K_i^{-1}K_i=1
\\
E_iE_{-j} - E_{-j}E_i = \delta_{i,j}\dfrac{K_iK_{i+1}^{-1}-K_i^{-1}K_{i+1}}{q-q^{-1}}
\\
K_iE_{\pm j}=q^{\pm (\epsilon_i,\alpha_j)}E_{\pm j}K_i
\\
E_{\pm i}^2E_{\pm j}-(q+q^{-1})E_{\pm i}E_{\pm j}E_{\pm i}+E_{\pm j}E_{\pm i}^2=0
\qquad\text{if}\quad |i-j|=1
\\
E_{\pm i}E_{\pm j}-E_{\pm j}E_{\pm i}=0\qquad\text{if}\quad |i-j|<1.
\end{gather*} 
\end{defn}

Let $V$ be the natural $n$ dimensional representation of ${U}_q(\mathfrak{gl}_n)$ and $d$
a non-negative integer.
There is a natural action of ${U}_q(\mathfrak{gl}_n)$ on $V^{\otimes d}$ with weights being the elements in
\begin{equation*}
\Lambda(n,d) =
\{ \lambda\in\bN^n\colon\sum_i\lambda_i = d \}
\end{equation*}
and highest weights the elements in
\begin{equation*}
\Lambda^+(n,d) =
\{ \lambda\in\Lambda(n,d)\colon\lambda_1 \geq \lambda_2\geq\dotsc \geq \lambda_n \},
\end{equation*}
the latter corresponding exactly to the irreducibles $V_\lambda$ that show up in the decomposition of $V^{\otimes d}$.
We denote by $\psi_{n,d}$ the representation above.

We can then define $S_q(n,d)$ as follows:
\begin{defn}
The $q$-Schur algebra $S_q(n,d)$ is the image of the representation $\psi_{n,d}$,
$$S_q(n,d) = \psi_{n,d}({U}_q(\mathfrak{gl}_n)).$$ 
\end{defn}

\smallskip

It is well-known that there is an action of the Iwahori-Hecke algebra $H_d(q)$ 
on $V^{\otimes d}$ commuting with the action of ${U}_q(\mathfrak{gl}_n)$.
As a matter of fact, we have
\begin{equation*}
S_q(n,d) \cong \End_{H_d(q)}(V^{\otimes d}) , 
\end{equation*}
and this may also be used to define the $q$-Schur algebra as a centralizer algebra.

For each $\lambda\in\Lambda^+(n,d)$, the 
${U}_q(\mathfrak{gl}_n)$-action on $V_{\lambda}$ factors through 
the projection $\psi_{n,d}\colon {U}_q(\mathfrak{gl}_n)\to S_q(n,d)$. 
This way we obtain all irreducible representations of $S_q(n,d)$. Note that 
this also implies that all representations of $S_q(n,d)$ have a 
weight decomposition. 
It is well known that 
$$S_q(n,d)\cong \prod_{\lambda\in\Lambda^+(n,d)}\End(V_{\lambda}),$$
and therefore $S_q(n,d)$ is a finite-dimensional split semi-simple 
unital algebra. 

Since we are only interested in weight representations we can restrict our attention
to the Beilinson-Lusztig-MacPherson idempotented version  
of ${U}_q(\mathfrak{gl}_n)$. 
It can be obtained 
from ${U}_q(\mathfrak{gl}_n)$ by adjoining orthogonal idempotents 
$1_{\lambda}$, for $\lambda\in\bZ^{n}$. We have the extra relations
\begin{align*}
1_{\lambda}1_{\nu} &= \delta_{\lambda,\nu}1_{\nu} 
\\
E_{\pm i}1_{\lambda} &= 1_{\lambda\pm{\alpha_i}}E_{\pm i}
\\
K_i1_{\lambda} &= q^{\lambda_i}1_{\lambda}.
\end{align*}
Note that $\dot{U}(\mathfrak{gl}_n)\cong \underset{\lambda,\mu\in\bZ^{n}}{\oplus} 
1_{\lambda}{U}_q(\mathfrak{gl}_n)1_{\mu}$  
is not unital because $1=\sum\limits_{\lambda\in\bZ^{n}}1_{\lambda}$ is  
an infinite sum. 
In this setting the $q$-Schur algebra occurs naturally as a quotient
of idempotented ${U}_q(\mathfrak{gl}_n)$, which happens to be very easy to describe.
Since $V^{\otimes d}$ is a weight representation, 
$\psi_{n,d}$ factors through $\dot{U}_q(\mathfrak{gl}_n)$ and we have 
$$S_q(n,d)\cong \dot{S}_q(n,d):=\psi_{n,d}(\dot{U}_q(\mathfrak{gl}_n)).$$
The kernel of $\psi_{n,d}$ is of course the ideal generated by all idempotents 
$1_{\lambda}$ such that $\lambda\not\in\Lambda(n,d)$. Thus we arrive at 
the following finite presentation of $S_q(n,d)$:

\begin{defn} 
$S_q(n,d)$ is the associative unital 
$\bQ(q)$-algebra generated by $1_{\lambda}$, for $\lambda\in\Lambda(n,d)$, 
and $E_{\pm i}$, for $i=1,\ldots,n-1$, subject to the relations
\begin{align*}
1_{\lambda}1_{\mu} &= \delta_{\lambda,\mu}1_{\lambda} 
\\[0.5ex]
\sum_{\lambda\in\Lambda(n,d)}1_{\lambda} &= 1
\\[0.5ex]
E_{\pm i}1_{\lambda} &= 1_{\lambda\pm\alpha_i}E_{\pm i}
\\[0.5ex]
E_iE_{-j}-E_{-j}E_i &= \delta_{ij}\sum\limits_{\lambda\in\Lambda(n,d)}
[\overline{\lambda}_i]1_{\lambda}
\end{align*}
where $\overline{\lambda}_i=\lambda_i - \lambda_{i+1}$. 
We use the convention that $1_{\mu}X1_{\nu}=0$, if $\mu$ 
or $\nu$ is not contained in $\Lambda(n,d)$. 
\end{defn}

\medskip

The Iwahori-Hecke algebra can be obtained as a quotient of $S_q(n,d)$ for $d\leq n$.
Let $(1)^d$ denote the weight $(1,\dotsm,1,0\dotsm,0)$ with $d$ ones followed by $n-d$ zeros.
\begin{prop}[Doty, Giaquinto~\cite{dotyg}]
For every $d\leq n$, the map $H_d(q)\to 1_{(1)^d}S_q(n,d)1_{(1)^d}$ given by 
$b_i\mapsto 1_{(1)^d}E_iE_{-i}1_{(1)^d}$ is an isomorphism.
\end{prop}

\medskip
Recall the \emph{quantum factorial} and \emph{quantum binomial} which are defined by
\begin{equation*}
[\kappa]! = [\kappa][\kappa -1]\dotsc [2][1]
\mspace{30mu}\text{and}\mspace{30mu}
\qbin{\kappa}{\kappa'} = \dfrac{[\kappa]!}{[\kappa-\kappa']![\kappa']!}
\end{equation*}
respectively, for $\kappa\geq\kappa'\geq 0$.
To establish a connection between the BMW algebra and the $q$-Schur algebra we need 
the \emph{divided powers} which are defined as 
\begin{equation*}
E_{\pm i}^{(\kappa)} := \dfrac{E_{\pm i}^{\kappa}}{[\kappa]!}.
\end{equation*}

\begin{lem}
The divided powers satisfy the relations
\begin{align*}
E_{\pm i}^{(\kappa)}1_\lambda 
&=
1_{\lambda \pm \kappa i_X}E_{\pm i}^{(\kappa)}1_\lambda
\\
E_{\pm i}^{(\kappa)}E_{\pm i}^{(\ell)}1_\lambda 
&=
\qbin{\kappa+\ell}{\kappa} E_{\pm i}^{(\kappa+\ell)}1_\lambda
\\
E_i^{(\kappa)}E_{-i}^{(\ell)} 1_\lambda 
&=
\sum\limits_{t=0}^{\min(\kappa,\ell)}\qbin{\kappa-\ell+\overline{\lambda}_i}{t} E_{-i}^{(\ell-t)}E_{i}^{(\kappa-t)} 1_\lambda 
\\
E_{-i}^{(\ell)}E_{i}^{(\kappa)} 1_\lambda 
&=
\sum\limits_{t=0}^{\min(\kappa,\ell)}\qbin{-\kappa+\ell-\overline{\lambda}_i}{t} E_{i}^{(\kappa-t)}E_{-i}^{(\ell-t)} 1_\lambda 
\end{align*}
and
\begin{align*}
E_{i}^{(\kappa)}E_{j}^{(\ell)} 1_\lambda 
&=
E_{j}^{(\ell)}E_{i}^{(\kappa)} 1_\lambda 
\rlap{\qquad\quad$\vert i - j\vert \neq 0,1$}
\\
E_{\pm i}^{(\kappa)}E_{\mp j}^{(\ell)} 1_\lambda 
&=
E_{\mp j}^{(\ell)}E_{\pm i}^{(\kappa)} 1_\lambda 
\rlap{\qquad\quad$i \neq j$}
\end{align*}
\end{lem}

We denote by $S_q(n,d)_\bZ$ the $\bZ[q,q^{-1}]$-subalgebra of $S_q(n,d)$
spanned by products of elements in the set $\{E_{\pm i}^{(\kappa)}1_\lambda\}$
(see~\cite[Thm. 2.3]{dotyg}).

\medskip

We are also interested in algebras which are certain direct sums of $q$-Schur algebras $S_q(n,d)$, 
for various specific values of $d$.

\begin{defn}
For $\Delta$ a finite subset of $\bN$ define the set of $n$-levels of $\Delta$ as
\begin{equation*}
\lev_n(\Delta)  = \biggl\{ \sum_i\mu_i ,\ \mu\in \Delta^n \biggr\} .
\end{equation*}
\end{defn}
This concept allows the introduction of the special direct sum $q$-Schur algebras.
\begin{defn}
We define the $\Delta$-$q$-Schur algebra as 
\begin{equation*}
S_q(n,\Delta) := 
\bigoplus\limits_{d\in\lev_n(\Delta)}
S_q(n,d)_\bZ .
\end{equation*}
\end{defn}

The identity in $S_q(n,\Delta)$ is $\sum_{\mu\in\Delta^n}1_\mu$
and the idempotents 
\begin{equation}
\label{eq:idemp-S}
e_d := \sum_{\lambda\in\Lambda(\Delta,d)}1_\lambda, 
\end{equation}
where $\Lambda(\Delta,d) := \Delta^n\cap\Lambda(n,d)$,
have the property that
\begin{equation*}
S_q(n,d)_\bZ = e_d S_q(n,\Delta) e_d .
\end{equation*}

\medskip

Recall that the irreducible modules $V_{\lambda}(d)$, 
for $\lambda\in\Lambda^+(n,d)$ , called \emph{Weyl modules}, can be 
constructed as subquotients of $S_q(n,d)$. Let $<$ denote  
the lexicographic order on $\Lambda(n,d)$.  
\begin{lem} 
\label{lem:weyl}
For any $\lambda\in\Lambda^+(n,d)$, we have 
$$V_{\lambda}(d) \cong S_q(n,d)1_{\lambda}/[\mu>\lambda].$$
Here $[\mu>\lambda]$ is the ideal generated by all elements of the form 
$1_{\mu}X1_{\lambda}$, for any $X\in S_q(n,d)$ and $\mu>\lambda$.
\end{lem}

The $S_q(n,\Delta)$-module $V$ is irreducible if and only if
$e_dV$ is an irreducible $S_q(n,d)$-module for exactly one $d\in\lev_n(\Delta)$,
while $e_{d'}V =0$ for the remaining elements $d'$ of $\lev_n(\Delta)$.
Therefore
the irreducibles of $S_q(n,\Delta)$ are exactly the $V_\lambda(d)$, 
where $d$ runs over $\lev_n(\Delta)$, 
and the highest weights of $S_q(n,\Delta)$ are precisely the elements in 
$\bigcup\limits_{d\in\lev_n(\Delta)}\Lambda^+(n,d)$.

\medskip


\subsection{$\MOY$ algebras and ramifications}   %
\label{ssec:moy}                                 %

In this subsection we describe a graphical calculus introduced in~\cite{moy}
by H.~Murakami, T.~Ohtsuki and S.~Yamada to obtain a 
state-sum-formula for the quantum $\mathfrak{sl}_N$ link polynomial. 
We relate this graphical calculus with the Skein and $q$-Schur algebras described before.
The calculus in~\cite{moy} can be roughly defined as a graphic description of  
the algebra of intertwiners between tensor products 
of higher fundamental representations of ${U}_q(\mathfrak{sl}_N)$.

An element $\mu=(\mu_1,\dotsc,\mu_k)\in\bN^k$ is a composition of $d$ 
if $\sum_{i}\mu_i =d$, it is denoted $\mu\vDash d$.
Let $W$ be the $N$ dimensional fundamental representation
of ${U}_q(\mathfrak{sl}_N)$ and consider 
\begin{equation*}
\bigoplus\limits_{\mu\vDash d} \bigwedge\limits^\mu W
\end{equation*}
where
\begin{equation*}
\bigwedge\limits^\mu W := \bigwedge\limits^{\mu_1}W \otimes \bigwedge\limits^{\mu_2}W \otimes\dotsc\otimes\bigwedge\limits^{\mu_n}W.
\end{equation*} 
For any $a,b \in \{1,\dotsc,N\}$ we have the intertwiners
\begin{equation*}
\Ydown_{a+b}^{a,b}\colon\wedge^{a+b}W\to \wedge^aW\otimes\wedge^bW 
\qquad\text{and}\qquad
\Yup_{a,b}^{a+b}\colon\wedge^{a}W\otimes\wedge^bW\to\wedge^{a+b}W.
\end{equation*}
Since any general intertwiner can be obtained as a composition of the various intertwiner maps above
we see that it can be described by compositions of the diagrams
\vspace*{2ex}
\begin{equation}
\label{eq:intertwiners}
\labellist
\tiny\hair 2pt
\pinlabel $a+b$ at  58 -15
\pinlabel $a$   at   0 155
\pinlabel $b$   at 115 157
\endlabellist
\figins{-18}{0.55}{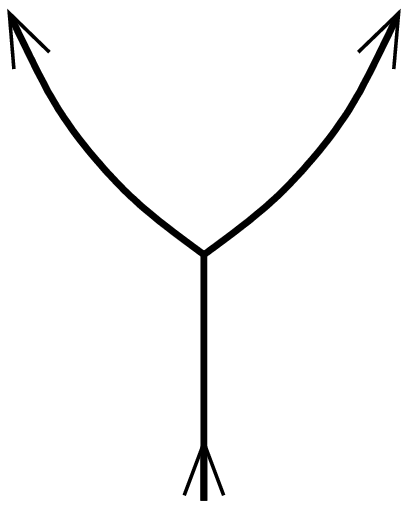}
\qquad\text{and}\qquad
\labellist
\tiny\hair 2pt
\pinlabel $a+b$ at  61 158
\pinlabel $a$   at   5 -15
\pinlabel $b$   at 120 -17
\endlabellist
\figins{-18}{0.55}{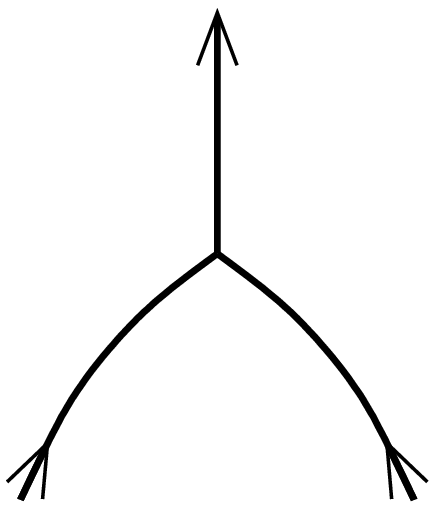}.
\end{equation}\vspace*{2ex}

\n The diagrams are read from bottom to top. 
For any general intertwiner we obtain a trivalent graph whose edges are coloured from $\{1,\dotsm,N\}$.
The product of two intertwiner maps corresponds to composition of diagrams, i.e. the product $a.b$ corresponds to
stacking the diagram associated to $a$ on top of the diagram associated to $b$.
We write $\MOY_n$ for the collection of all these graphs containing less than $n$ ingoing and less than $n$ 
outgoing strands and such that sum of the colours of the edges at the bottom (or the top) is equal to $d$. 
The $\MOY$ algebra $\MOY_q(n,d,N)$ is the associative, unital algebra over $\bZ[q,q^{-1}]$ generated by all the
$\MOY_n$ diagrams modulo some relations, which can be found in~\cite{moy}. 
Notice that the product $a\cdot b$ is zero if the labels along the edges where the diagrams are to be glued do not match. 
As a matter of fact the relations given in~\cite{moy} do not form a complete set 
(a complete set of relations is conjectured in~\cite{morrison} and proved in~\cite{CKM}).
The authors of~\cite{moy} defined an evaluation of closed diagrams and derived 
only the relations which are enough to prove
invariance of the corresponding link polynomial.

The following lemma relates the $q$-Schur and the $\MOY$ algebras.
\begin{lem}
We have a homomorphism of algebras $f\colon S_q(n,d) \to \MOY_q(n,d,N)$ given by

\begin{align*}
1_{\lambda}\quad &\mapsto\qquad
\labellist
\tiny\hair 2pt
\pinlabel $\lambda_1$ at   5 235    
\pinlabel $\lai$      at 190 235
\pinlabel $\laii$     at 295 235
\pinlabel $\lambda_n$ at 485 235 
\pinlabel $\dotsc$    at 100 120 
\pinlabel $\dotsc$    at 400 120  
\endlabellist
\figins{-25}{0.7}{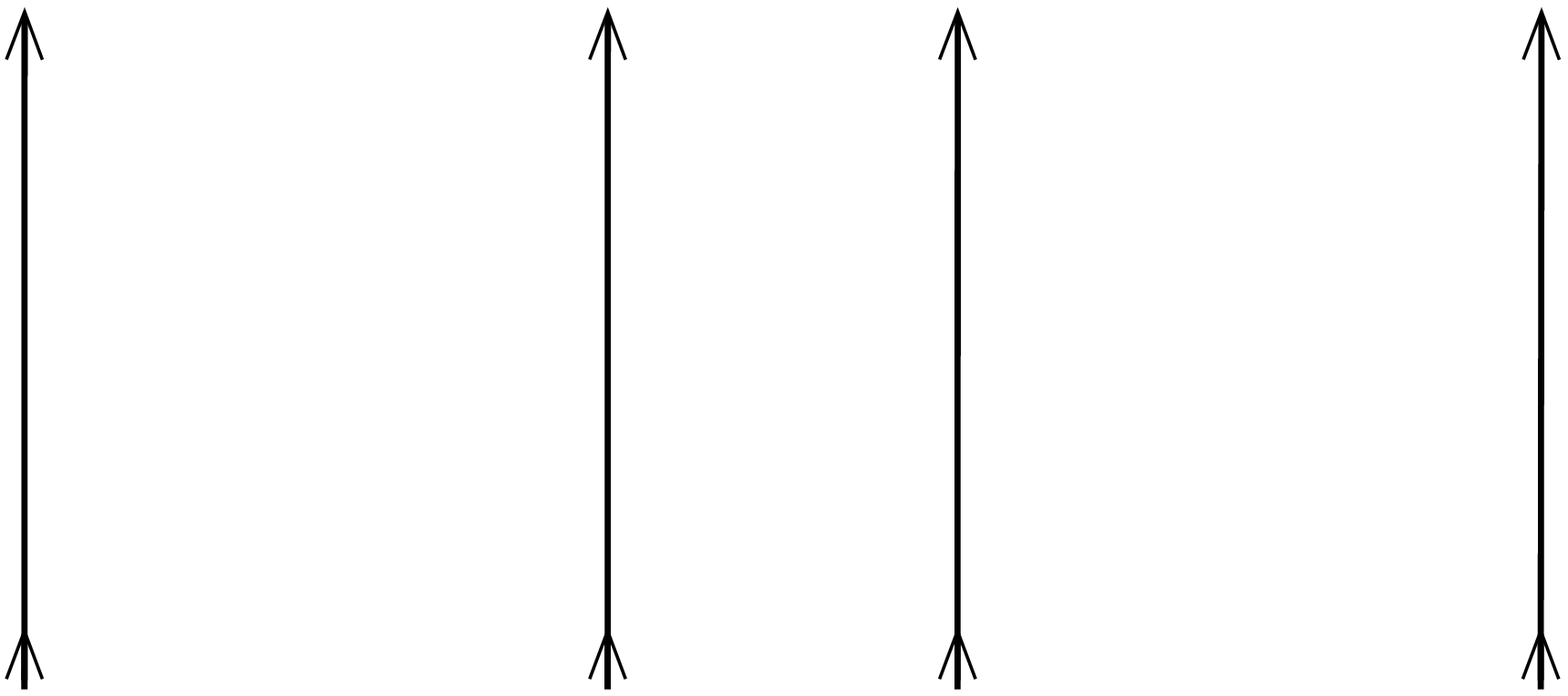}
\displaybreak[0] \\[3ex]
E_{+i}1_\lambda\ \
&\mapsto
\qquad
\labellist
\tiny\hair 2pt
 \pinlabel $\lambda_1$ at   5 235    
\pinlabel $\lai+1$      at 170 235
\pinlabel $\laii-1$     at 320 235
\pinlabel $\lambda_n$ at 485 235 
\pinlabel $1$ at 243 74
\pinlabel $\lai$ at 170 15
\pinlabel $\laii$ at 350 15 
\pinlabel $\dotsc$ at 100 120 
\pinlabel $\dotsc$ at 400 120  
\endlabellist
\figins{-25}{0.7}{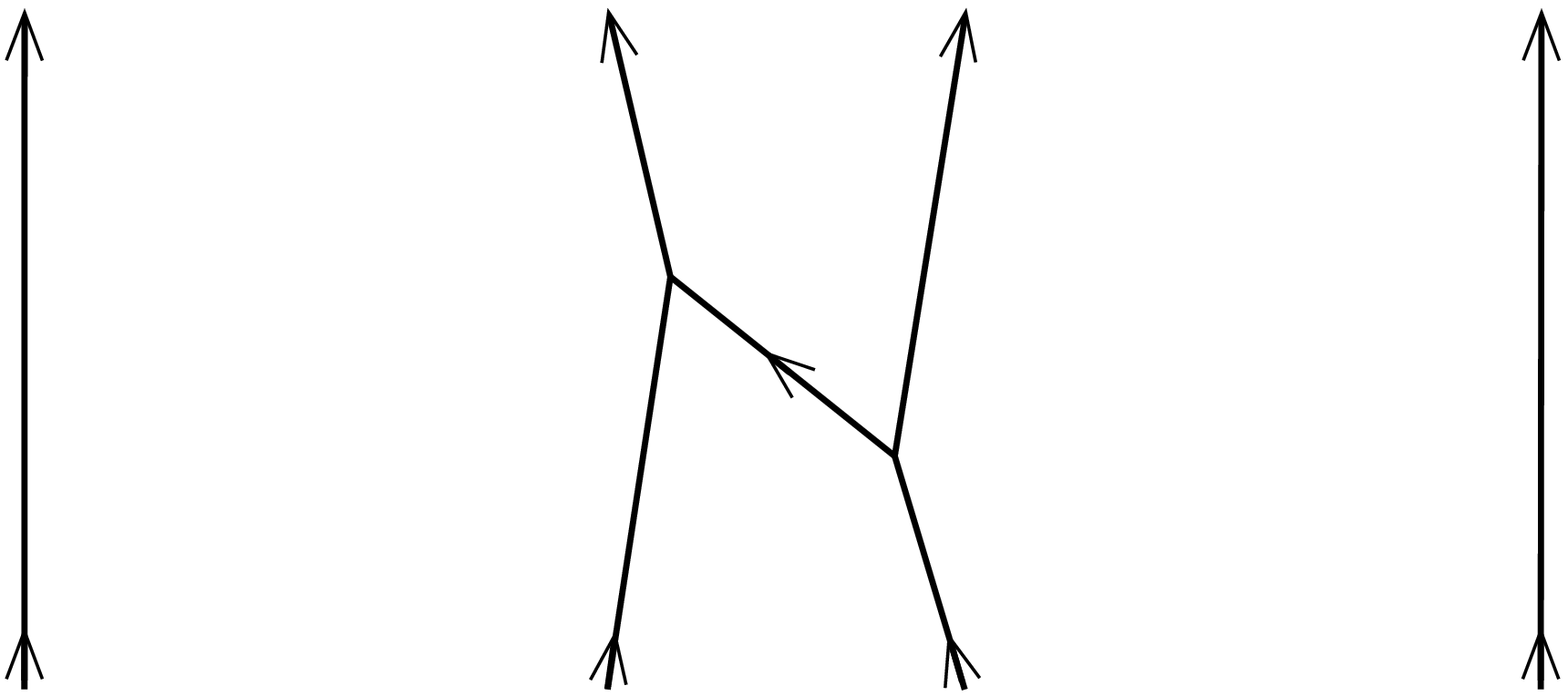}
\displaybreak[0] \\[3ex]
E_{-i}1_\lambda \ \
&\mapsto
\qquad
\labellist
\tiny\hair 2pt
\pinlabel $\lambda_1$ at   5 235    
\pinlabel $\lai-1$      at 170 235
\pinlabel $\laii+1$     at 320 235
\pinlabel $\lambda_n$ at 485 235 
\pinlabel $1$ at 243 74
\pinlabel $\lai$ at 170 15
\pinlabel $\laii$ at 350 15 
\pinlabel $\dotsc$ at 100 120 
\pinlabel $\dotsc$ at 400 120  
\endlabellist
\figins{-25}{0.7}{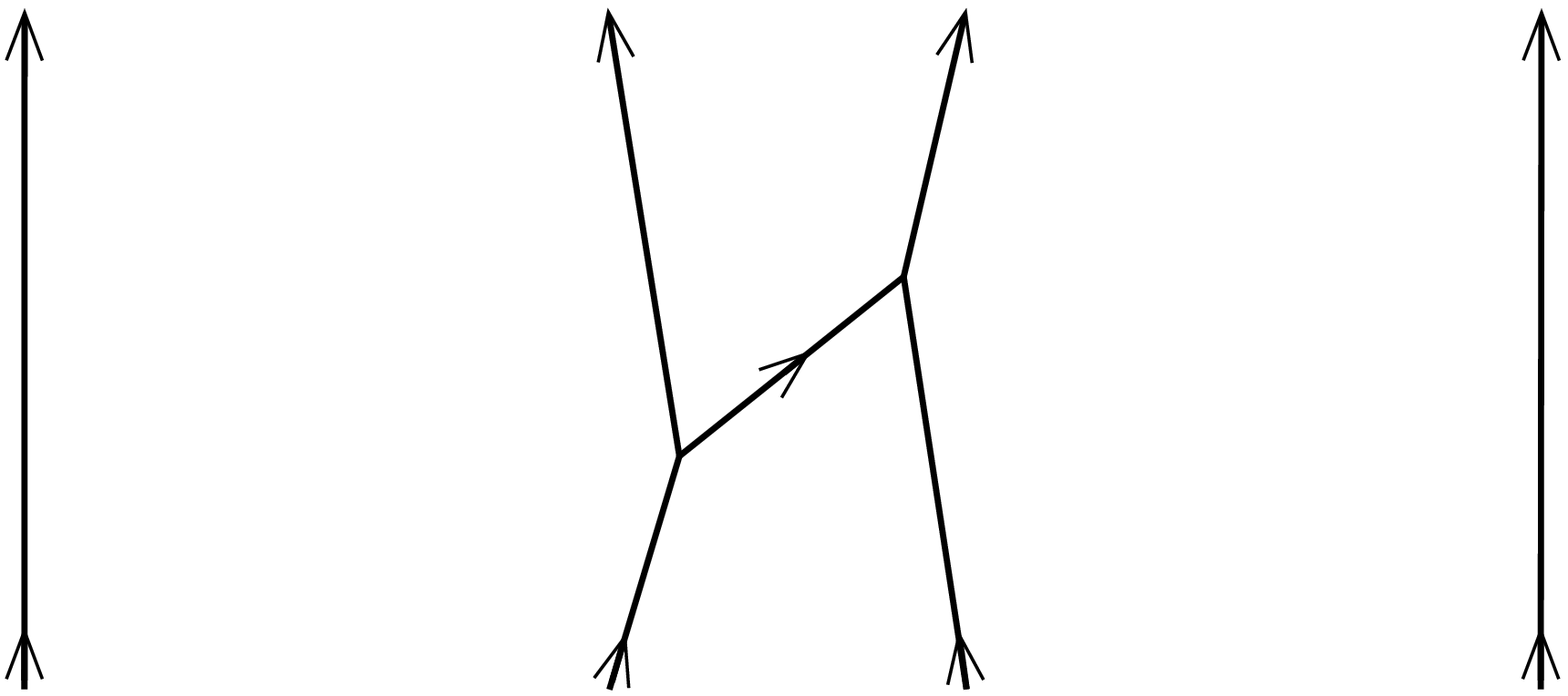}
\end{align*}
\end{lem}
This is the decategorification of the functor $\mathcal{F}_{Bim}$ from~\cite[Section 4]{MSV}.
From the results of~\cite[Section 4]{MSV} we have that the assignment above is well defined
and defines a map of algebras.

This map generalizes to divided differences yielding

\begin{align*}
E_{+i}^{(\kappa)}1_\lambda\ \
&\mapsto
\qquad
\labellist
\tiny\hair 2pt
 \pinlabel $\lambda_1$   at   5 235    
\pinlabel $\lai+\kappa$  at 170 235
\pinlabel $\laii-\kappa$ at 320 235
\pinlabel $\lambda_n$    at 485 235 
\pinlabel $\kappa$       at 243 76
\pinlabel $\lai$         at 170 15
\pinlabel $\laii$        at 350 15 
\pinlabel $\dotsc$       at 100 100 
\pinlabel $\dotsc$       at 400 100  
\endlabellist
\figins{-23}{0.7}{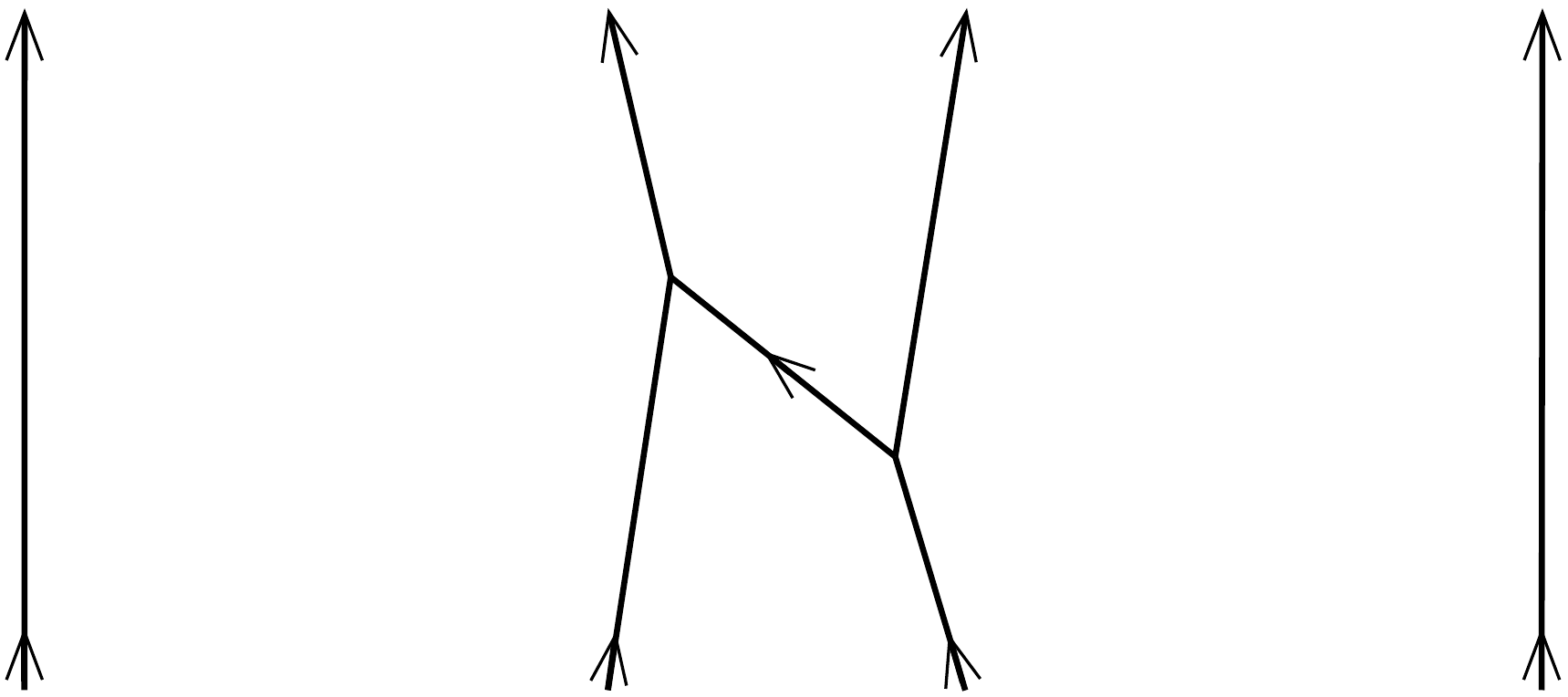}
\displaybreak[0] \\[3ex]
E_{-i}^{(\kappa)}1_\lambda \ \
&\mapsto
\qquad
\labellist
\tiny\hair 2pt
\pinlabel $\lambda_1$ at   5 235    
\pinlabel $\lai-\kappa$      at 170 235
\pinlabel $\laii+\kappa$     at 320 235
\pinlabel $\lambda_n$ at 485 235 
\pinlabel $\kappa$ at 243 76
\pinlabel $\lai$ at 170 15
\pinlabel $\laii$ at 350 15 
\pinlabel $\dotsc$ at 100 100 
\pinlabel $\dotsc$ at 400 100  
\endlabellist
\figins{-23}{0.7}{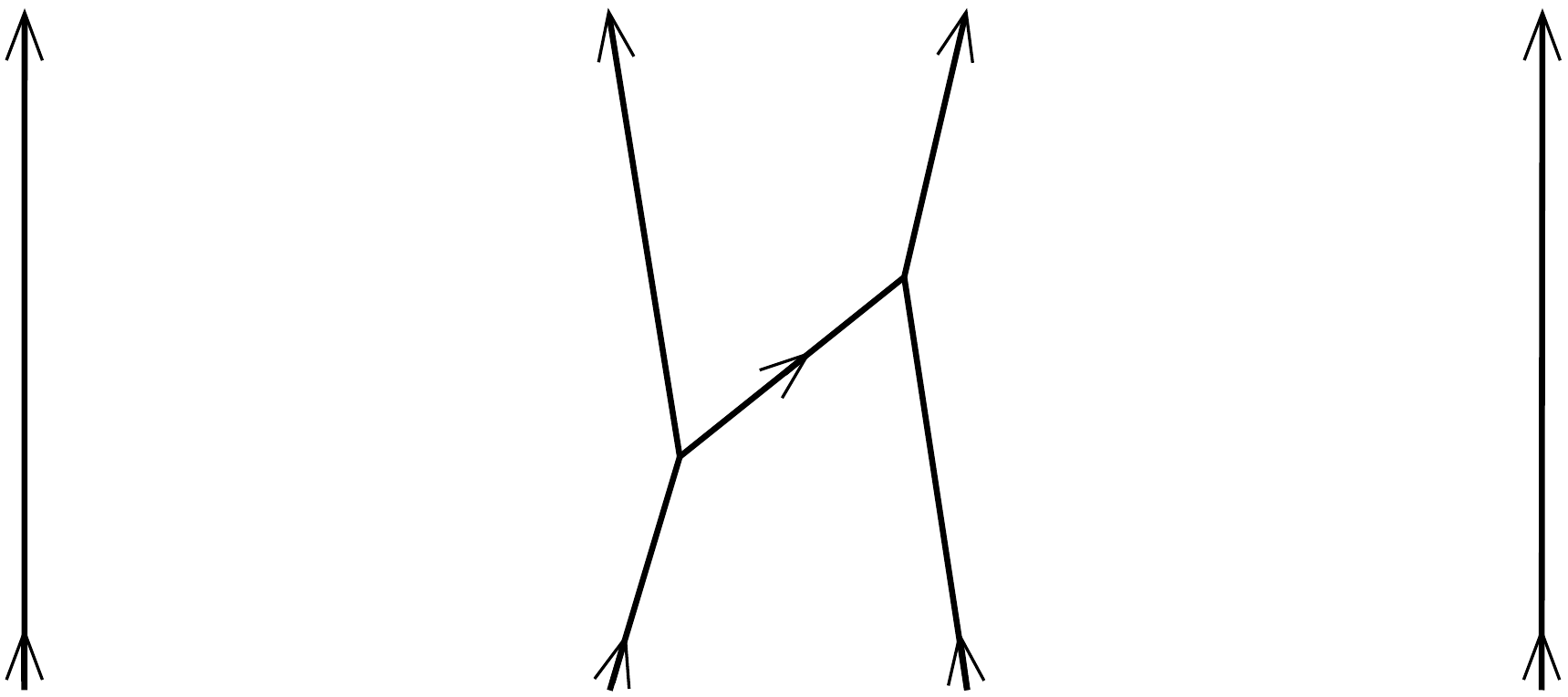}
\end{align*}

\medskip

We now turn to a 1-parameter specialization of the Skein algebra of 
Subsection~\ref{ssec:skein}.
Let $a=q^N$ and consider $\skein_q(n,N): = \skein_n(q^N,q)$. 
The grading in $\skein_n(a,q)$ 
descends to $\skein_q(n,N)$ and we have (compare with Equation~\eqref{eq:skeingrad})
\begin{equation*}
\skein_q(n,N) = \bigoplus\limits_{n_\pm =-n}^n \bigl(\skein_q(n,N)\bigr)_{n_\pm} .
\end{equation*}

\medskip

From now on we take $\Delta_N = \{1,N-1\}$ and for this particular $\Delta$ we denote 
the $\Delta$-$q$-Schur algebra by $\schurD$.

\medskip

The rest of this section is devoted to prove the following
\begin{prop}
\label{prop:arc-schur}
There is an injective homomorphism of algebras 
\begin{equation*}
\alpha\colon \skein_q(n,N)\to \schurD .
\end{equation*}
\end{prop}

This result will follow from the lemma below.
\begin{lem}
For each pair $(n_+,n_-)$ with $n_++n_-=n$ there is an injective homomorphism of algebras 
\begin{equation*}
\alpha(n_\pm) \colon \skein_q(n,N)_{n_\pm}\to S_q(n,n_++(N-1)n_-).
\end{equation*}
\end{lem}
\begin{proof}
The algebra $\skein_q(n,N)$ has an interpretation as a diagrammatic description of the algebra 
of intertwiners between tensor products
of the fundamental representation $V$ of $\mathfrak{sl}_N$ and its dual $V^*$~\cite{moy}.
The isomorphism of representations $\Phi\colon \wedge^kV^* \xra{\ \cong\ } \wedge^{N-k}V$
induce injections
\begin{equation*}
\phi(n_\pm)\colon
\skein_q(n,N)_{n_\pm}
\hookrightarrow \MOY_q(n,n_++(N-1)n_-,N)
\end{equation*}
which is given on generators by

\begin{align*}
\figins{-10}{0.35}{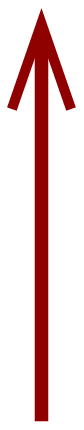}\quad\ 
&{\longmapsto}\mspace{60mu}
\labellist
\tiny\hair 2pt
\pinlabel $1$ at  25 180
\endlabellist
\figins{-25}{0.80}{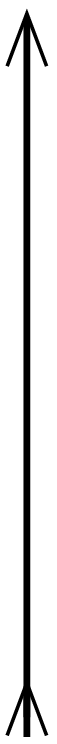}
&
\figins{-10}{0.35}{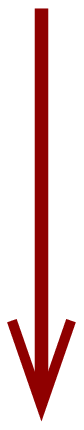}\quad\ 
&{\longmapsto}\mspace{60mu}
\labellist
\tiny\hair 2pt
\pinlabel $N-1$ at  55 180
\endlabellist
\figins{-25}{0.80}{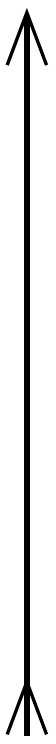}
\displaybreak[0]\\
\figins{-2}{0.15}{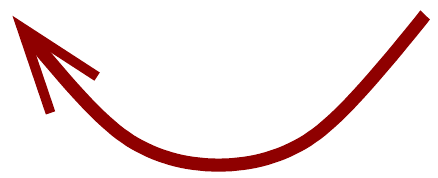}\  \ 
&{\longmapsto}\
\labellist
\tiny\hair 2pt
\pinlabel $0$   at  -5 40
\pinlabel $N$   at 125 40
\pinlabel $1$ at -5 180
\pinlabel $N-1$   at 155 180
\pinlabel $1$   at  60 120
\endlabellist
q^{N-1}\ \
\figins{-30}{0.90}{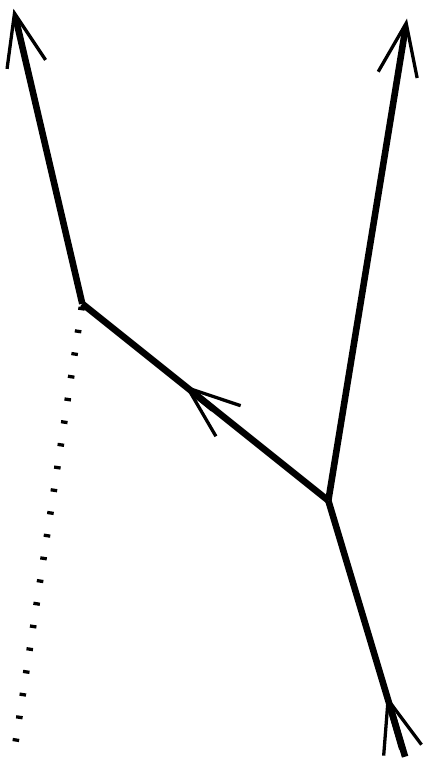}
&
\figins{-2}{0.15}{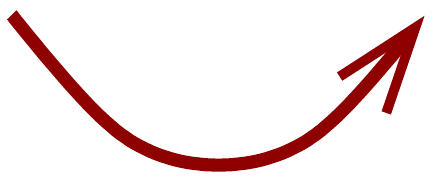}\ \ 
&{\longmapsto}\qquad
\labellist
\tiny\hair 2pt
\pinlabel $N$   at  -0 40
\pinlabel $0$   at 130 40
\pinlabel $N-1$   at   -35 180
\pinlabel $1$ at 130 180
\pinlabel $1$   at  65 120
\endlabellist
\figins{-30}{0.90}{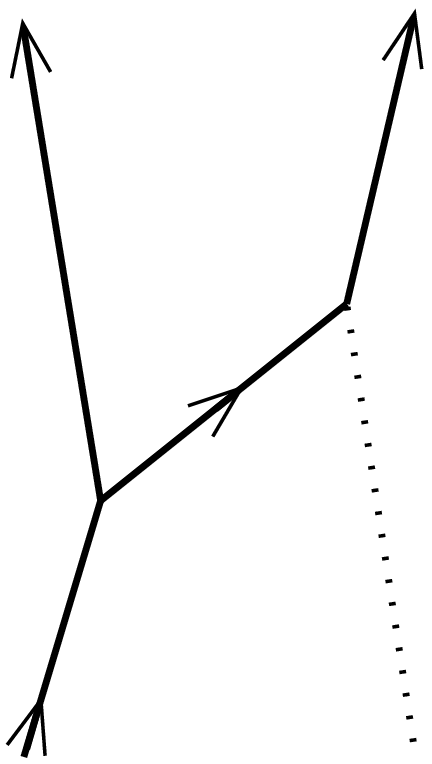}
\displaybreak[0]\\
\figins{-2}{0.15}{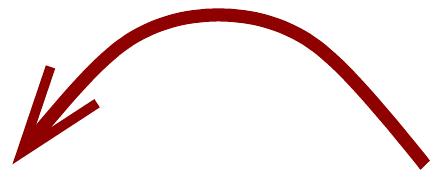}\ \ 
&{\longmapsto}\
\labellist
\tiny\hair 2pt
\pinlabel $N-1$ at -30 40
\pinlabel $1$   at 125 40
\pinlabel $N$   at  -5 180
\pinlabel $0$   at 135 180
\pinlabel $1$   at  65 120
\endlabellist
q^{-N+1}\ \ \ 
\figins{-30}{0.90}{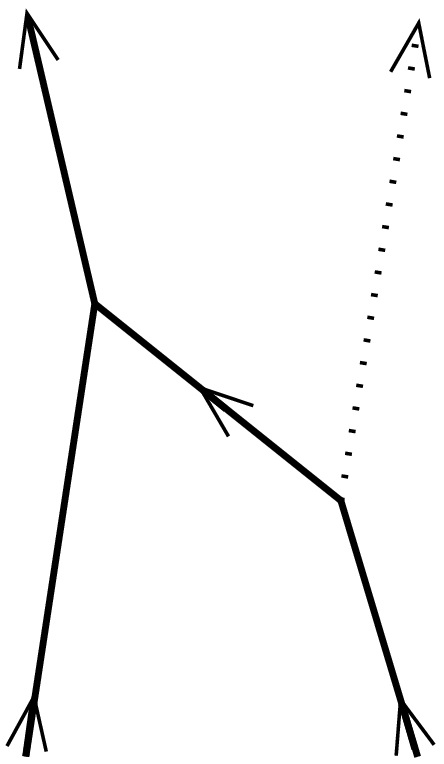}
&
\figins{-2}{0.15}{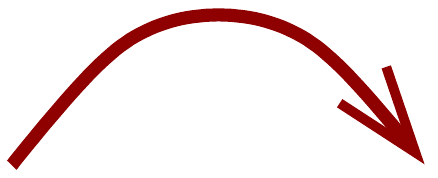}\ \  
&{\longmapsto}\qquad
\labellist
\tiny\hair 2pt
\pinlabel $1$   at   0 40
\pinlabel $N-1$ at 155 40
\pinlabel $0$   at  -5 180
\pinlabel $N$   at 130 180
\pinlabel $1$   at  60 120
\endlabellist
\figins{-30}{0.90}{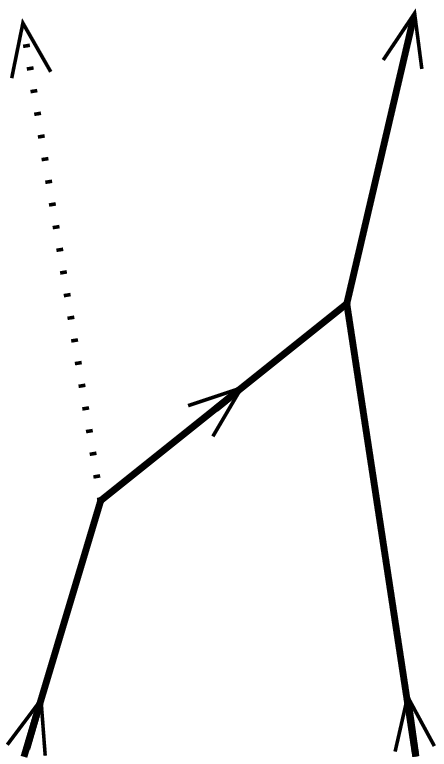}
\displaybreak[0]\\
\figins{-10}{0.35}{crossl}\ \ 
&{\longmapsto}\qquad
\labellist
\tiny\hair 2pt
\pinlabel $1$   at 130  30
\pinlabel $N-1$ at -30  30
\pinlabel $N-1$ at 155 190
\pinlabel $1$   at -5 190
\pinlabel $N-2$ at  60 145
\pinlabel $\color[rgb]{.5,.5,.5}{\downarrow}$ at 60 125
\endlabellist
\figins{-30}{0.90}{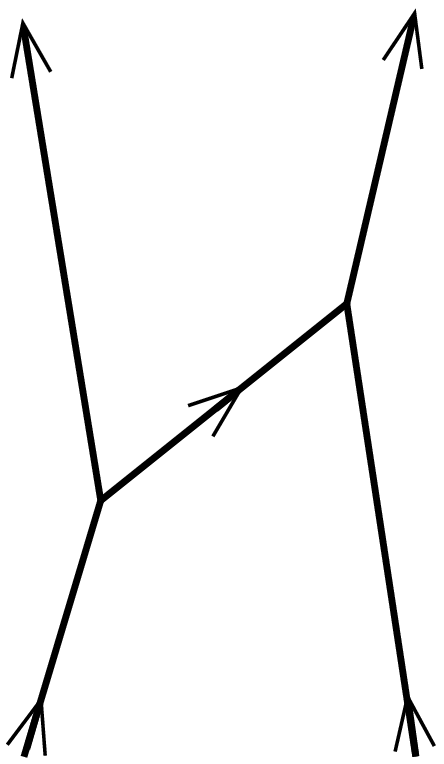}
&
\figins{-10}{0.35}{crossr}\ \  
&{\longmapsto}\qquad
\labellist
\tiny\hair 2pt
\pinlabel $1$   at  -5  30
\pinlabel $N-1$ at 155  30
\pinlabel $N-1$ at -35 190
\pinlabel $1$   at 135 190
\pinlabel $N-2$ at  65 145
\pinlabel $\color[rgb]{.5,.5,.5}{\downarrow}$ at 65 125
\endlabellist
\figins{-30}{0.90}{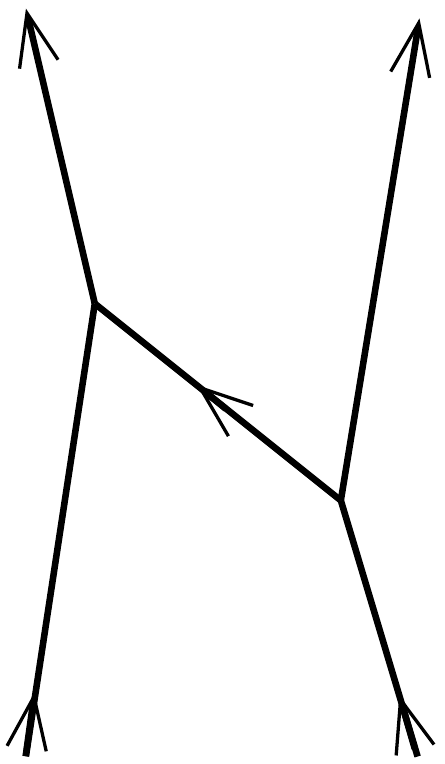}
\displaybreak[0]\\
\figins{-10}{0.35}{crossu}\ \ 
&{\longmapsto}\qquad
\labellist
\tiny\hair 2pt
\pinlabel $1$ at -5  30
\pinlabel $1$ at 115  30
\pinlabel $1$ at -5 190
\pinlabel $1$ at 115 190
\pinlabel $2$ at  35 110
\pinlabel $0$ at 110 110
\endlabellist
\figins{-30}{0.90}{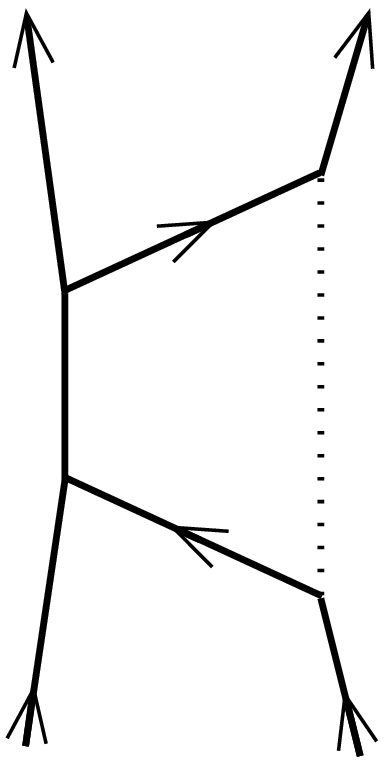}
&
\figins{-10}{0.35}{crossd}\ \  
&{\longmapsto}\qquad
\labellist
\tiny\hair 2pt
\pinlabel $N-1$ at -35  30
\pinlabel $N-1$ at 145  30
\pinlabel $N-1$ at -35 190
\pinlabel $N-1$ at 145 190
\pinlabel $N$   at   0 110
\pinlabel $N-2$ at 135 110
\endlabellist
\figins{-30}{0.90}{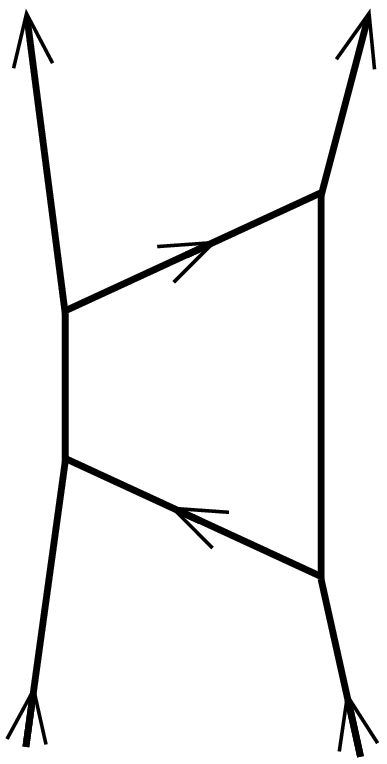}
\end{align*}

\medskip

It is clear that
\begin{equation*}
\Image\phi(n_\pm)\subset\Image f .
\end{equation*}
We next define homomorphisms 
$\alpha(n_\pm)\colon \skein_q(n,N)_{n_\pm}\to S_q(n,n_++(N-1)n_-)$
that makes the following diagram commute 
\begin{equation}
\label{eq:cd-ArcSchur}
\vcenter{
\xymatrix@C=2.5mm{
\skein_q(N,n)_{n_{\pm}}\ar[rr]^{\phi(n_\pm)}\ar[dr]_{\alpha(n_\pm)} && \MOY_q(n,n_++(N-1)n_-,N)
\\
&  S_q(n, n_++(N-1)n_-)\ar[ur]_{f} & 
}}
\end{equation}

First define $\lambda(\und{\ell})\in\bN^n$ as 
\begin{equation*}
\lambda_i(\und{\ell})=
\begin{cases} 
1  & \text{\ if\ } \und{\ell}_i=+
\\
N-1 & \text{\ if\ } \und{\ell}_i=-
\end{cases}
\end{equation*}
We also denote $\imath_{(a,b)}\lambda(\und{\ell})$ 
the sequence obtained 
from $\lambda(\und{\ell})$
by a translation 
$\lambda_{j}(\und{\ell})\mapsto\lambda_{j+2}(\und{\ell}) $
followed by taking $\lambda_i(\und{\ell})=a$ and 
$\lambda_{i+1}(\und{\ell})=b$.

\smallskip

$\bullet$ A diagram in $\skein_q(N,n)$ consisting of $n$ vertical strands 
is sent to $1_{\lambda(\und{\ell})}$.

$\bullet$ For the remaining generators we assume that there are $i-1$ vertical strands on the left of the 
$\skein_n$-diagram depicted and that $\ell$ intersects generically
the diagram under consideration in the neighborhood
of its bottom boundary.
\begin{align*}
\figins{-2}{0.15}{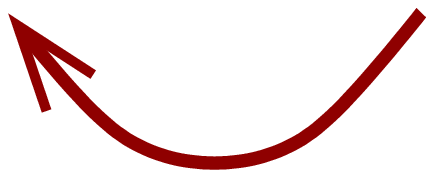}\  \ 
&{\longmapsto}\ \
q^{N-1}\
E_i1_{\imath_{(0,N)}\lambda(\und{\ell})}
&
\figins{-2}{0.15}{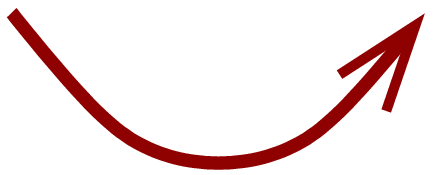}\ \ 
&{\longmapsto}\quad
E_{-i}1_{\imath_{(N,0)}\lambda(\und{\ell})}
\displaybreak[0]\\[2ex]
\figins{-2}{0.15}{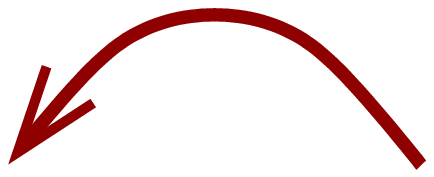}\ \ 
&{\longmapsto}\ \
q^{-N+1}\
E_{i}1_{\lambda(\und{\ell})}
&
\figins{-2}{0.15}{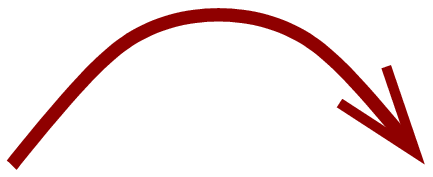}\ \  
&{\longmapsto}\quad
E_{-i}1_{\lambda(\und{\ell})}
\displaybreak[0]\\[2ex]
\figins{-10}{0.35}{crossl}\ \ 
&{\longmapsto}\quad
E_{-i}^{(N-2)}1_{\lambda(\und{\ell})}
&
\figins{-10}{0.35}{crossr}\ \  
&{\longmapsto}\quad
E_{i}^{(N-2)}1_{\lambda(\und{\ell})}
\displaybreak[0]\\[2ex]
\figins{-10}{0.35}{crossu}\ \ 
&{\longmapsto}\quad
E_{-i}E_{i}1_{\lambda(\und{\ell})}
&
\figins{-10}{0.35}{crossd}\ \  
&{\longmapsto}\quad
E_{-i}E_{i}1_{\lambda(\und{\ell})}
\end{align*}
Notice that due to the particular form of the $\lambda(\und{\ell})$'s involved in the upward pointing and 
downward pointing vertices we also have
\begin{equation*}
\alpha\biggl(
\figins{-10}{0.35}{crossu}\biggr)  
=
E_{i}E_{-i}1_{\lambda(\und{\ell})}
\qquad\text{and}\qquad
\alpha\biggl(
\figins{-10}{0.35}{crossd}\biggr)  
=
E_{i}E_{-i}1_{\lambda(\und{\ell})}\ .
\end{equation*}
This ends the definition of $\alpha(n_\pm)$.

It is immediate that the diagram in Equation~\eqref{eq:cd-ArcSchur} commutes.
Therefore we conclude that the homomorphism $\alpha(n_\pm)$ is injective.
\end{proof}

\begin{proof}[Proof of Proposition~\ref{prop:arc-schur}]
Define $\alpha$ as the sum of all the $\alpha(n_\pm)$.
The claim now follows from the fact that the images of the $\alpha(n_\pm)$ are disjoint 
for distinct values of $n_{\pm}$.
\end{proof}

\medskip

Although $f$ is not known to be injective nor surjective
the collection of $\MOY$ diagrams is useful when dealing with $q$-Schur algebras.
To keep working diagrammatically with $S_q(n,d)$ we need a 
more precise version of the $\MOY$ algebra, i.e. an algebra generated by the set of $\MOY$ 
diagrams modulo a set of \emph{complete relations}.
To this end we define
\begin{defn}
$A_{\MOY}(d,N)$ is the algebra generated by the $\MOY$ diagrams in $\Image{f}$
modulo the relations coming from $S_q(n,d)$.
\end{defn}

Unless otherwise stated all the $\MOY$ diagrams occurring from now on 
will refer to generators of  $A_{\MOY}(d,N)$.


\subsection{The BMW-embedding in $q$-Schur}   %
\label{ssec:embds}                            %

Denote $\bmw_q(n,N)$ the specialization $\bmw_n(q^N,q)$. 
This specialization is related with the representation theory of the quantum group $U_q(\mathfrak{so}_{2n})$ 
(see for instance~\cite{kauff-phys}).

Composing the $q^N$-specialization of Jaeger's homomorphism $\psi\colon\bmw_q(n,N)\to\skein_q(n,N)$ 
with the homomorphism $\alpha\colon\skein_q(n,N)\to \schurD$ of 
Subsection~\ref{ssec:moy}
we obtain a homomorphism $\theta\colon\bmw_q(n,N)\to\schurD$. 

\smallskip

In terms of the generators of $\bmw_q(n,N)$
the homomorphism $\theta$ reads

\begin{align*}
 \figins{-10}{0.35}{cupcap}\
&\overset{\theta}{\longmapsto}\quad\
\labellist
\tiny\hair 2pt
\pinlabel $N-1$ at -30 30
\pinlabel $1$   at 115 30
\pinlabel $1$   at  -5 190
\pinlabel $N-1$ at 145 190
\pinlabel $N$   at   0 110
\pinlabel $0$   at 110 110
\endlabellist
\figins{-30}{0.90}{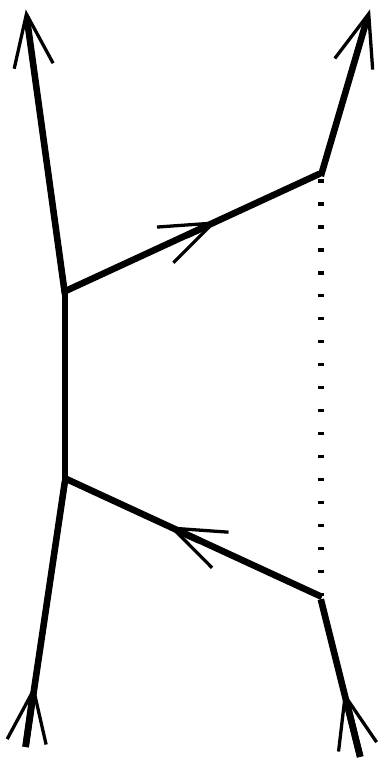}\qquad + \qquad
\labellist
\tiny\hair 2pt
\pinlabel $1$   at  -5  30
\pinlabel $N-1$ at 140  30
\pinlabel $N-1$ at -35 190
\pinlabel $1$   at 115 190
\pinlabel $N$   at   0 110
\pinlabel $0$   at 110 110
\endlabellist
\figins{-30}{0.90}{sq-NO}\qquad + q^{-N+1}\qquad
\labellist
\tiny\hair 2pt
\pinlabel $N-1$ at -35  30
\pinlabel $1$ at 115  30
\pinlabel $N-1$ at -35 190
\pinlabel $1$   at 115 190
\pinlabel $N$   at  35 110
\pinlabel $0$   at 110 110
\endlabellist
\figins{-30}{0.90}{sq-NO}\quad + q^{N-1}\ \
\labellist
\tiny\hair 2pt
\pinlabel $1$   at  -5  30
\pinlabel $N-1$ at 140  30
\pinlabel $1$   at  -5 190
\pinlabel $N-1$ at 140 190
\pinlabel $0$   at  35 110
\pinlabel $N$   at 110 110
\endlabellist
\figins{-30}{0.90}{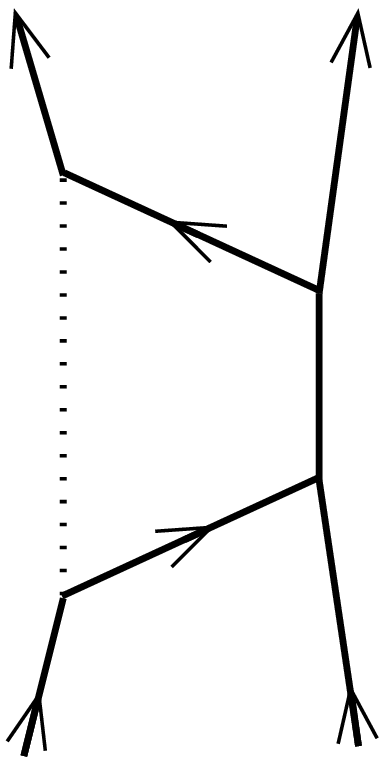}
\displaybreak[0]\\[2ex]
\figins{-10}{0.35}{4vert}\
&\overset{\theta}{\longmapsto}\ \
q^{N-2}\
\labellist
\tiny\hair 2pt
\pinlabel $1$   at  -5  30
\pinlabel $N-1$ at 140  30
\pinlabel $1$   at  -5 190
\pinlabel $N-1$ at 140 190
\pinlabel $0$   at  35 110
\pinlabel $N$   at 110 110
\endlabellist
\figins{-30}{0.90}{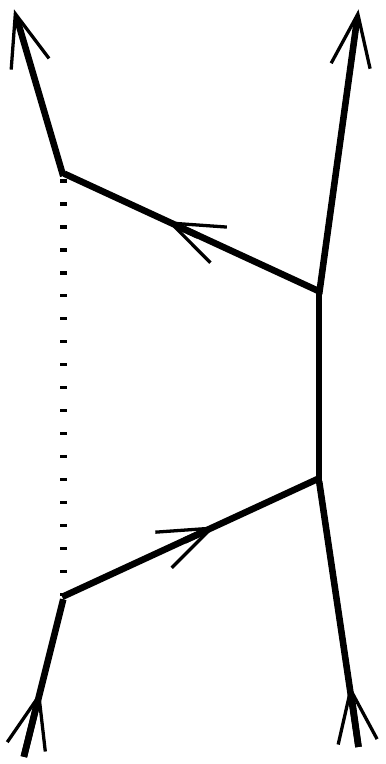}\qquad +\ \
q^{-1}\quad\ \
\labellist
\tiny\hair 2pt
\pinlabel $N-1$  at  -35 190
\pinlabel $1$    at  120 190
\endlabellist
\figins{-30}{0.90}{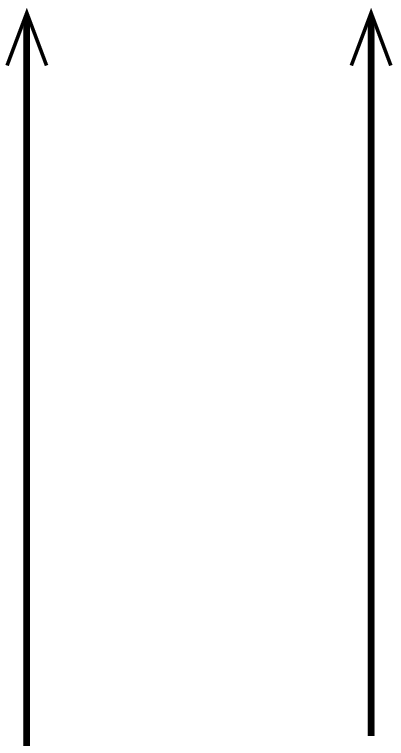}\ \ \  
+ q^{-N+2}\qquad
\labellist
\tiny\hair 2pt
\pinlabel $N-1$ at -30  30
\pinlabel $1$   at 115  30
\pinlabel $N-1$ at -30 190
\pinlabel $1$   at 115 190
\pinlabel $N$   at  35 110
\pinlabel $0$   at 110 110
\endlabellist
\figins{-30}{0.90}{sq-NO}\ \ +\ \
q\ \;
\labellist
\tiny\hair 2pt
\pinlabel $1$    at   -5 190
\pinlabel $N-1$  at  145 190
\endlabellist
\figins{-30}{0.90}{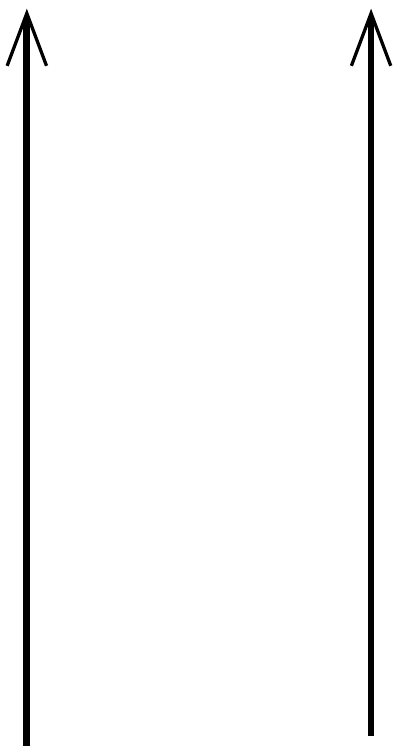}
\\[2ex] 
& \mspace{60mu} + \ \
\labellist
\tiny\hair 2pt
\pinlabel $1$ at -5  30
\pinlabel $1$ at 115  30
\pinlabel $1$ at -5 190
\pinlabel $1$ at 115 190
\pinlabel $2$ at  35 110
\pinlabel $0$ at 110 110
\endlabellist
\figins{-30}{0.90}{sq-NO}\ \ + \qquad
\labellist
\tiny\hair 2pt
\pinlabel $1$   at  -5  30
\pinlabel $N-1$ at 155  30
\pinlabel $N-1$ at -35 190
\pinlabel $1$   at 135 190
\pinlabel $N-2$ at  65 145
\pinlabel $\color[rgb]{.5,.5,.5}{\downarrow}$ at 65 125
\endlabellist
\figins{-30}{0.90}{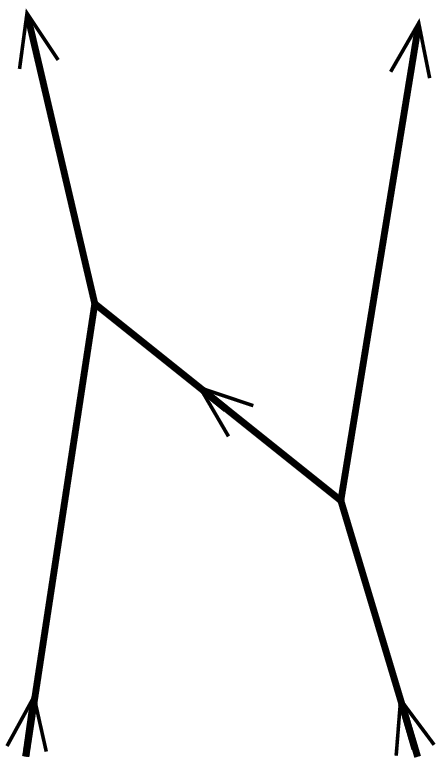}\qquad + \qquad
\labellist
\tiny\hair 2pt
\pinlabel $1$   at 130  30
\pinlabel $N-1$ at -30  30
\pinlabel $N-1$ at 155 190
\pinlabel $1$   at -5 190
\pinlabel $N-2$ at  60 145
\pinlabel $\color[rgb]{.5,.5,.5}{\downarrow}$ at 60 125
\endlabellist
\figins{-30}{0.90}{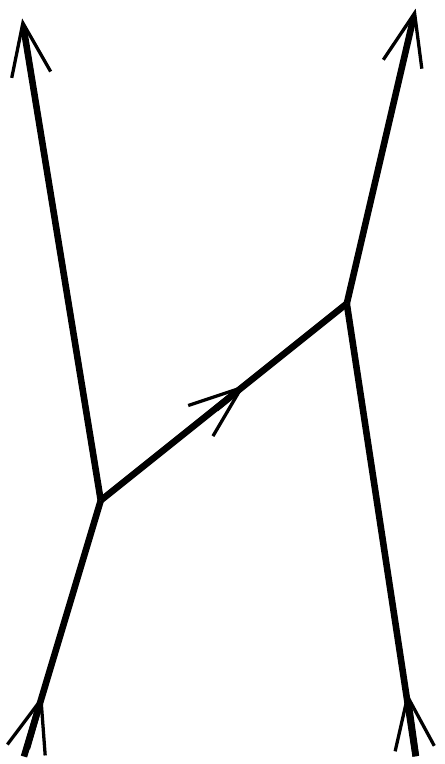}\qquad + \qquad
\labellist
\tiny\hair 2pt
\pinlabel $N-1$ at -35  30
\pinlabel $N-1$ at 145  30
\pinlabel $N-1$ at -35 190
\pinlabel $N-1$ at 145 190
\pinlabel $N$   at   0 110
\pinlabel $N-2$ at 135 110
\endlabellist
\figins{-30}{0.90}{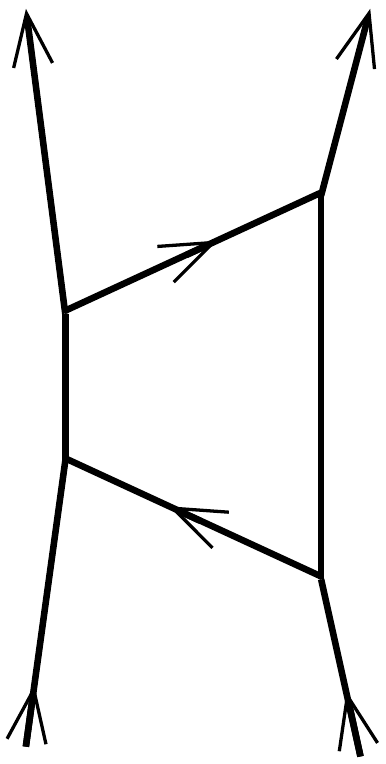}
\displaybreak[0]\\[2ex]
\figins{-10}{0.35}{one}\quad\ 
&\overset{\theta}{\longmapsto}
\qquad
\labellist
\tiny\hair 2pt
\pinlabel $1$ at 25 190
\endlabellist
\figins{-30}{0.90}{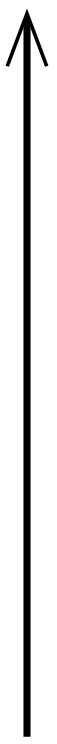}\quad + \quad
\labellist
\tiny\hair 2pt
\pinlabel $N-1$ at 50 190
\endlabellist
\figins{-30}{0.90}{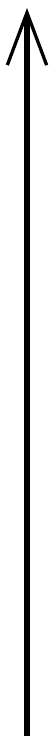}
\end{align*}

Notice that 
\begin{equation*}
\labellist
\tiny\hair 2pt
\pinlabel $j$   at  -5  30
\pinlabel $i$   at 110  30
\pinlabel $i$   at  -5 190
\pinlabel $j$   at 110 190
\pinlabel $i+j$ at -15 110
\pinlabel $0$   at 110 110
\endlabellist
\figins{-30}{0.90}{sq-NO}
\qquad =\qquad
\labellist
\tiny\hair 2pt
\pinlabel $j$   at  -5  30
\pinlabel $i$   at 110  30
\pinlabel $i$   at  -5 190
\pinlabel $j$   at 110 190
\pinlabel $0$   at   0 110
\pinlabel $i+j$ at 125 110
\endlabellist
\figins{-30}{0.90}{sq-ON}
\end{equation*}
Notice also that $EF1_{a,a} = FE1_{a,a}$ for all $a\in \bN$ 
or in pictures,
\begin{equation*}
\labellist
\tiny\hair 2pt
\pinlabel $N-1$ at -35  30
\pinlabel $N-1$ at 145  30
\pinlabel $N-1$ at -35 190
\pinlabel $N-1$ at 145 190
\pinlabel $N$   at   0 110
\pinlabel $N-2$ at 135  90
\endlabellist
\figins{-30}{0.90}{sqweb-r}
\qquad =\qquad
\labellist
\tiny\hair 2pt
\pinlabel $N-1$ at -35  30
\pinlabel $N-1$ at 145  30
\pinlabel $N-1$ at -35 190
\pinlabel $N-1$ at 145 190
\pinlabel $N-2$ at -30  90
\pinlabel $N$   at 115 110
\endlabellist
\figins{-30}{0.90}{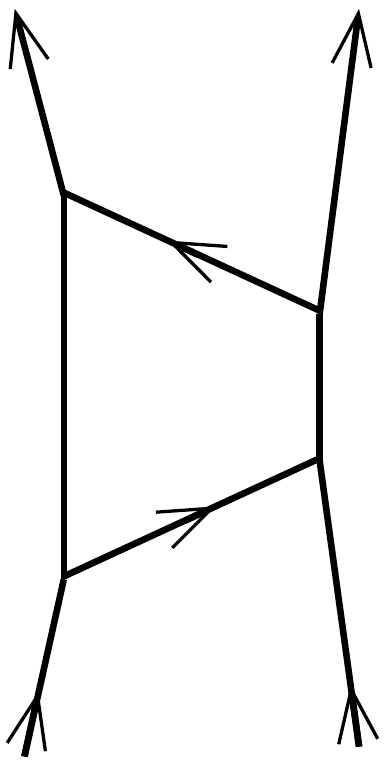}\quad .
\end{equation*}
We see that $\theta$ preserves the symmetry under the operation that simultaneously 
sends $q$ to $q^{-1}$ and reflects the diagram around a vertical axis
passing through the middle.

\bigskip

In this language, it is easy to check by direct computation that
\begin{align*}
\theta\left(\ \
\figins{-24}{0.75}{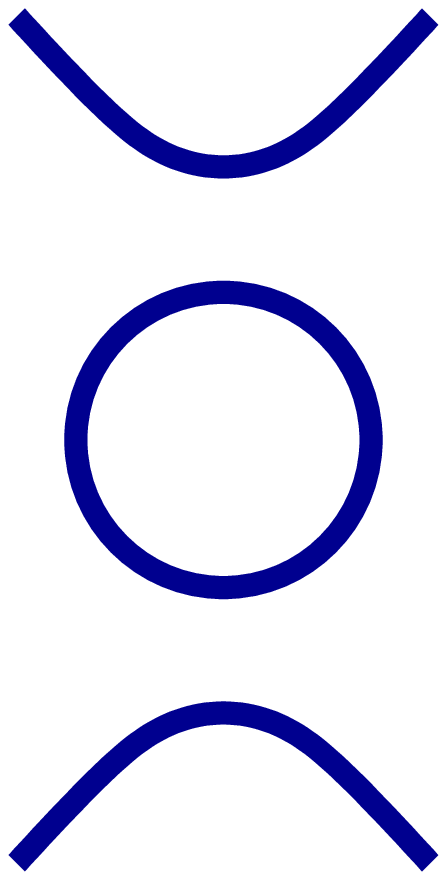}\ \
\right)
= 
\bigl( [N]q^{-N+1} + [N]q^{N-1}\bigr)
\theta\Biggl(\
\figins{-10}{0.35}{cupcap}\
\Biggr)
= 
\bigl( [2N-1] + 1\bigr)
\theta\Biggl(\
\figins{-10}{0.35}{cupcap}\
\Biggr).
\end{align*}

\bigskip

For the following we define 
\begin{align*}
\hat{1}(i)_{a,b} &= \sum\limits_{j_k\in \{1,N-1\}} 1_{j_1,\dotsc ,j_{i-1},a,b, j_{i+2},\dotsc, j_n}
\intertext{and}
\hat{1}(i) &= \sum\limits_{j_k\in \{1,N-1\}} 1_{j_1, \dotsc, j_n}.
\end{align*}

For the sake of completeness we give the homomorphism $\theta$ in algebraic terms,
which reads
\begin{align*}
\labellist
\tiny\hair 2pt
\pinlabel $i$    at   0 -15
\pinlabel $i+1$  at 125 -17
\endlabellist
\figins{-10}{0.35}{cupcap}\ \ \
&\overset{\theta}{\longmapsto}\
E_{-i}^{(N-1)}E_{i}^{}\hat{1}(i)_{N-1,1} + 
E_{-i}^{}E_{i}^{(N-1)}\hat{1}(i)_{1,N-1} 
\\[2.5ex]
& \mspace{40mu}
+ q^{-N+1}
 E_{-i}^{}E_{i}^{}\hat{1}(i)_{N-1,1} +
q^{N-1}
E_{i}^{}E_{-i}^{}\hat{1}(i)_{1,N-1}
\displaybreak[0]\\[2.5ex]
\labellist
\tiny\hair 2pt
\pinlabel $i$    at   0 -15
\pinlabel $i+1$  at 125 -17
\endlabellist
\figins{-10}{0.35}{4vert}\ \ \
&\overset{\theta}{\longmapsto}\
q^{N-2}E_{i}^{}E_{-i}^{}\hat{1}(i)_{1,N-1}
+ q^{-1}\hat{1}(i)_{N-1,1} +
q^{N-2}E_{-i}^{}E_{i}^{}\hat{1}(i)_{N-1,1} 
+ q^{}\hat{1}(i)_{1,N-1} 
\\[2.5ex]
& \mspace{40mu}
+ E_{-i}^{}E_{i}^{}\hat{1}(i)_{1,1}
+ E_{i}^{(N-2)}\hat{1}(i)_{1,N-1}
+ E_{-i}^{(N-2)}\hat{1}(i)_{N-1,1}
+ E_{-i}^{}E_{i}^{}\hat{1}(i)_{N-1,N-1}
\displaybreak[0]\\[2.5ex]
\labellist
\tiny\hair 2pt
\pinlabel $i$   at 2 -15
\endlabellist
\figins{-10}{0.35}{one}\quad\ \ \
&\overset{\theta}{\longmapsto}\
\hat{1}(i)
\end{align*}

\bigskip

Theorem~\ref{thm:jaegmap} and Proposition~\ref{prop:arc-schur} together imply the following.
\begin{prop}\label{prop:injschur}
The homomorphism $\theta$ is injective.
\end{prop}

Recall the Iwahori-Hecke algebra can be obtained as a quotient of the $\bmw$ and Schur algebras.
The following are easy consequences of the results in the previous section.
Recall the idempotents 
$e_{(+)^n}$ and $e_d\in\schurD$ defined in Equations~\eqref{eq:special-idemp-skein} 
and~\eqref{eq:idemp-S}.
\begin{prop}
We have isomorphisms
\begin{align*}
H_n(q) 
&\cong\ 
e_{(+)^n}\psi\bigl(\bmw_q(n,N)\bigr)e_{(+)^n} 
\cong 
1_{(1)^n}\schurD 1_{(1)^n}.
\displaybreak[0]\\[2ex]
\bmw_q(n,N) &\cong\ 
\theta(\bmw_q(n,N)) \cong \bigoplus_{d\in\lev_n(\Delta)} e_d \theta(\bmw_q(n,N)) e_d.
\end{align*} 
\end{prop}

\begin{prop}
The  projection of $\bmw_q(n,N)$ onto $H_q(d)$ factors through $\schurD$
\begin{equation*}
\xymatrix{
\bmw_q(n,N)\ar[rr]^{ }\ar[dr]_{ } && H_d(q)
\\
& \schurD \ar[ur]_{} & 
}
\end{equation*} 
\end{prop}



\section{Categorifications, loose ends and speculations} %
\label{sec:cat}                                         %

In this section we explain how the Jaeger's homomorphism can be used to
produce categorifications of the $\bmw_q(n,N)$. 
Although not strictly necessary to understand the main ideas in this section, 
some familiarity with~\cite{KR1} and~\cite{MSV} would be desirable.

\medskip

We can think of the 1-variable specialization of the 
Jaeger's homomorphism as targeting two algebras, the 
$a=q^N$-specialization of the HOMFLY-PT skein algebra of Subsection~\ref{ssec:skein} on one side, 
and the $\Delta_N$-$q$-Schur algebra of Subsection~\ref{ssec:schur} on the other 
(denoted $\psi$ and $\theta$ there).
Both this algebras have been categorified, and in more than one way.
Below in Subsections~\ref{ssec:matfac} and~\ref{ssec:scat} we give 
the main idea of the categorifications of the algebras $\skein_q(n,N)$ and $S_q(n,d)$
leaving the details to~\cite{KR1} and~\cite{MSV} respectively.

\medskip

Let us first recall the philosophy of categorification.
The split Grothendieck group $K_0$ of an additive category $\cC$ is the free abelian group
generated by the isomorphism classes $[M]$ of objects $M$ of $\cC$ modulo the relation 
$[C]= [A] + [B]$  whenever $C\cong A\oplus B$.
When $\cC$ has a monoidal structure the Grothendieck group is a ring, with multiplication given by
$[A\otimes B] = [A] [B]$. Moreover, if $\cC$ is a graded category then $K_0(\cC)$ has a  
structure of $\bZ[q,q^{-1}]$-module, where $[M\{k\}] = q^k[M]$.

Let $R$ be a commutative ring with $1$, $\cA$ and algebra over $R$ and $\{a_i\}_{i\in I}$ a basis of $\cA$. 
By a (weak) \emph{categorification of $(\cA,\{a_i\}_{i\in I})$} we mean 
an additive monoidal category $\cA$ together with an isomorphism
\begin{equation}
\label{eq:K0A} 
\gamma\colon R\otimes_\bZ K_0(\cC) \to \cA
\end{equation}
sending the class of each indecomposable object of $\cC$ to a basis element of $\cA$ 
(see~\cite{KMS} for a detailed discussion).

\subsection{Matrix Factorizations and the $\skein_q(n,N)$ categorification}  %
\label{ssec:matfac}                                                          %

In~\cite{KR1} Khovanov and Rozansky constructed a link homology theory 
categorifying the quantum $\mathfrak{sl}_N$-invariant $P_N$
of links. 
The starting point is the diagrammatic MOY state-sum model~\cite{moy} of $P_N$, 
whose underlying algebraic structure is exactly $\skein_q(n,N)$ 
(this was the main motivation for the presentation given in Definition~\ref{def:moy-aq}).
The procedure consists of expanding a link diagram $D$ in an alternating sum in $\skein_q(n,N)$,
each term being evaluated to a polynomial in $\bZ[q,q^{-1}]$ using the defining rules 
of $\skein_q(n,N)$ from Definition~\ref{def:moy-aq}.

The main ingredient of~\cite{KR1} is the use of Matrix Factorizations.
Let $R$ be a commutative ring and $W\in R$. A matrix factorization of $W$ consists
of a free $\bZ/2\bZ$-graded $R$-module $M$ together with a map $D\in\End(M)$ 
of degree 1 satisfying $D^2=W.\id_M$.

In~\cite{KR1} Khovanov and Rozansky associated to each graph $\Gamma$ in $\skein_q(n,N)$ a 
certain graded matrix factorization $M(\Gamma)$ and show that for each relation 
$\Gamma = \sum_i \Gamma_i$ in Definition~\ref{def:moy-aq} (with $a=q^N$) we have a 
direct sum decomposition $M(\Gamma)\cong \oplus_i M(\Gamma_i)$.
To a link diagram they associate a complex of matrix factorizations
and prove that the direct sum decompositions they obtain are sufficient to have topological
invariance up to homotopy.

The reader now may ask why not use the bigraded matrix factorizations from~\cite{KR2} to obtain a categorification 
of the 2-variable BMW algebra. 
Unfortunately the matrix factorization from~\cite{KR2} associated to the left hand side 
of Equation~\eqref{R2bMOY} is not isomorphic to the direct sum of the matrix factorizations 
associated to the right hand side. 
This is the main reason why the HOMFLY-PT homologies that exist are defined only for braids 
and closures of braids and not for tangles.

\subsection{The $q$-Schur categorification}  %
\label{ssec:scat}                            %

In~\cite{MSV} a diagrammatic categorification of the $q$-Schur 
algebra was constructed using a quotient of Khovanov and Lauda's 
categorified quantum groups from~\cite{KL, KL:err}.
Khovanov and Lauda's categorified quantum $\mathfrak{sl}_n$ 
consists of a 2-category $\dot{\cU}(\mathfrak{sl}_n)$ defined from the following data. 
The objects are \emph{weights} $\lambda\in\bZ^{n-1}$. The 1-morphisms are products of symbols 
$\lambda'\mathcal{E}_{\pm i}\lambda$ (with 
$\lambda'_j = \lambda_j-1$ if $j=i\pm 1$, $\lambda'_j = \lambda_j+2$ if $j=i$, 
and $\lambda'_j = \lambda_j$ otherwise) with the convention that says that
$\lambda'\mathcal{E}_{\pm i}\mu\nu\mathcal{E}_{\pm i}\lambda$ is zero unless $\mu=\nu$.
The 2-morphism of $\dot{\cU}(\mathfrak{sl}_n)$ are given by planar diagrams in a strip generated 
by oriented arcs that can intersect transversely and can be decorated with dots 
(closed oriented 1 manifolds are allowed). 
The boundary of each arc is decorated with a 1-morphism. 
These 2-morphisms are subject to a set of relations which we do not give here 
(see~\cite{KL, MSV} for details).

The main insight of~\cite{MSV} was to upgrade Khovanov and Lauda's categorified 
quantum $\mathfrak{sl}_n$ to a categorification $\dot{\cU(\mathfrak{gl}_n)}$ of quantum $\mathfrak{gl}_n$
(taking Khovanov and Lauda's diagrams and relations of $\dot{\cU(\mathfrak{sl}_n)}$ with $\mathfrak{gl}_n$-weights)
and define the categorification of $S_q(n,d)$ as the quotient of $\dot{\cU}(\mathfrak{gl}_n)$
by 2-morphisms factoring through a weight not in $\Lambda(n,d)$.
The main result of~\cite{MSV} is that $K_0\bigl(\dot{\cU}(\mathfrak{gl}_n)_{/\sim}\bigr)$ is isomorphic to $S_q(n,d)$
from Subsection~\ref{ssec:schur}. It is not hard to guess how to combine 
the results of~\cite{stosic} to lift divided powers and therefore obtain 
a categorification of the integral form $S_q(n,d)_\bZ$. 
Going from there to the $\Delta_N$-$q$-Schur algebra $\schurD$ is an obvious step.

\subsection{The general picture}             %

We now explain abstractly our procedure. We use the symbol $\cY$ to refer to both the categorifications 
of $\skein_q(n,N)$ and of $\schurD$ and symbols $\{Y_j\}_{j\in J}$ to denote its indecomposable objects. 
Each of these categorifications has the \emph{Krull-Schmidt property}, meaning that 
each object decomposes into direct sum of indecomposable objects 
which is unique up to permutation (see~\cite[Sec. 2.2]{Ringel}).
This implies that the classes of the indecomposables in $K_0(\cY)$ form a basis of $K_0(\cY)$. In addition this basis is positive 
that is, all the multiplication coefficients in this basis are nonnegative since 
they count multiplicities in direct sum decompositions.

As explained before we expand every element $x$ of $\bmw_q(n,N)$ as a linear combination of elements 
of another algebra, the latter admitting a categorification. We write it abstractly as
\begin{equation}\label{eq:decalg}
x = \sum\limits_{j\in J}c_j y_j 
\end{equation}
where each $y_j$ is a basis element of $\skein_q(n,N)$ or $\schurD$ and
$c_j\in\bN[q,q^{-1}]$.

Homomorphism $\gamma$~\eqref{eq:K0A} sends $[Y_j]$ to $y_j$  and therefore we think of 
the object $Y_j$ as the lift to $\cY$ of the basis element $y_j$.
This results in a well-defined object $X$ of $\cY$ given by
\begin{equation}\label{eq:catdec}
X = \bigoplus\limits_{j\in J}Y_j\{c_j\} 
\end{equation}
where we use the notation $Y\{q^{i_1}+\dotsc+q^{i_k}\} = Y\{i_1\}\oplus \dotsc \oplus Y\{i_k\}$.

\medskip

We now define an additive  monoidal category $\cX$ from this data.
\begin{defn}
Category $\cX$ is the (monoidal) full subcategory of $\cY$ generated by products of the 
objects $X$ given by Equations~\eqref{eq:catdec} which are images under Jaeger's homomorphism 
of the generators of $\bmw_q(n,N)$ from Equations~\eqref{eq:genX}-\eqref{eq:genum}.
The morphisms of $\cX$ are the obvious ones from $\cY$.
\end{defn}

Given a basis $(x_i)_{i\in I}$ of $\bmw_q(n,N)$ consider the element 
$X_i$ constructed above for each $x_i$. 
Since the relations in $\bmw_q(n,N)$ lift to relations in $\cY$, 
it follows that the $\{[X_i]\}_{i\in I}$ generates the Grothendieck ring $K_0(\cX)$. 
Recall that the $\bmw_q(n,N)$ is naturally equipped with a non-degenerate bilinear form  
given by the Kauffman polynomial. It follows that if there was non-trivial relations satisfied by the 
$X_i$'s in $\cY$ it would contradict the non-degeneracy of this bilinear form. Hence we can deduce 
that the $[X_i]$'s are linearly independent in $K_0(\cX)$ and form a basis of $K_0(\cX)$.
Using this remark and the results of~\cite{KR1} and~\cite{MSV} it is not hard to prove the following
\begin{prop}
We have an isomorphism $K_0(\cX)\cong \bmw_q(n,N)$.
\end{prop}

Unfortunately category $\cX$ does not have the Krull-Schmidt property, which is a desirable property
for the reason explained above.
To get a categorification with the Krull-Schmidt property we need to add some objects to $\cX$.
This yields another category $\cX'$ as follows.
\begin{defn}
An object $A$ of $\cY$ is an object of $\cX'$ if there are objects $B$ and $C$ of $\cX$ such that
$A\oplus B\cong C$.
\end{defn}

The construction of category $\cX'$ resembles the construction 
of the category of special bimodules in~\cite{soerg} (see also~\cite[Sec. 3.1]{MS}).
Notice we still have $K_0(\cX')\cong K_0(\cX)$.
We were able to prove by hand that $\cX'$ has indeed the Krull-Schmidt 
property in the cases up to $\bmw_q(3,N)$.

\begin{conj}
In the case of the categorifications of $\skein_q(n,N)$ and $\schurD$ 
the category $\cX'$ has the Krull-Schmidt property.
\end{conj}

One could feel tempted of taking the Karoubi envelope of $\cX$
to guarantee Krull-Schmidt property.
Recall the Karoubi envelope of a category $\cC$ consists of adding more objects to $\cC$ 
which are images of idempotents. In the Karoubi envelope every idempotent splits
and consequently we have the Krull-Schmidt property~\cite{Ringel}. 
It is easy to see that this procedure would add too many objects making the Grothendieck ring
too large to be isomorphic to $\bmw_q(n,N)$.

\smallskip

We are suggesting a category having the Krull-Schmidt property which is not Karoubian.
Such categories are known to exist. For example the category of super-vector spaces 
with odd dimension and even dimension both equal is not Karoubian but has the Krull-Schmidt property.

It would be interesting to relate the lift of the $4$-vertex using matrix factorizations
with the one Khovanov and Rozansky did in~\cite{KR3} using convolutions of matrix factorizations.

%

\vspace*{1cm}

\subsection*{Acknowledgments}
P.V. was financially supported by Portuguese funds via the FCT - Funda\c c\~ao para  a Ci\^encia e Tecnologia, 
through project number PTDC/MAT/101503/2008,  New Geometry and Topology and  
through the post-doctoral fellowship SFRH/BPD/46299/ 2008.
E.W. has been partially supported by a FABER Grant n¡X110CVHCP-2011 and by the French ANR project ANR-11-JS01-002-01.
E.W. thanks Paul Martin for useful E-mail correspondence.


\end{document}